\documentclass[10pt,reqno]{amsart}
\include{formatAndDefs}
\usepackage{amsfonts,amssymb,amsmath}

\parskip        \baselineskip
\usepackage{aliascnt}
\usepackage[mathscr]{eucal}
\usepackage{graphicx}
\usepackage{latexsym}
\usepackage{psfrag}
\usepackage{epsfig}
\usepackage[utf8]{inputenc}
\usepackage[english]{babel}
\usepackage{tikz}
\usepackage[foot]{amsaddr}
\usepackage{lineno, blindtext}
\usetikzlibrary{decorations.pathreplacing}

\newcommand{\suchthat}{\;\ifnum\currentgrouptype=16 \middle\fi|\;}

\renewcommand{\leq}{\leqslant}
\renewcommand{\geq}{\geqslant}
\renewcommand{\div}{\mbox{div}\,}
\newcommand{\curl}{{\mbox{curl}\,}}
\newcommand{\ds}{\displaystyle}

\numberwithin{equation}{section}
\newtheorem{thm}{Theorem}
\numberwithin{thm}{section}

\newaliascnt{lemma}{thm}
\newtheorem{lem}[lemma]{Lemma}
\aliascntresetthe{lemma}

\newaliascnt{proposition}{thm}

\aliascntresetthe{proposition}

\newaliascnt{corollary}{thm}
\newtheorem{corollary}[corollary]{Corollary}
\aliascntresetthe{corollary}

\newaliascnt{definition}{thm}

\aliascntresetthe{definition}

\newaliascnt{remark}{thm}
\newtheorem{remark}[remark]{Remark}
\aliascntresetthe{remark}

\usepackage[colorlinks=true]{hyperref}

\begin{document}
		\title[Observability of the adjoint of a compressible fluid-structure model]{Observability and unique continuation of the adjoint of a linearized compressible fluid-structure model in a 2d channel}
		
		\date{\today}
		\author{Sourav Mitra}
		\thanks{{Acknowledgments}: The work was partially done  when the author was a member of  Institut de Math\'ematiques de Toulouse. The author wishes to thank the ANR project ANR-15-CE40-0010 IFSMACS as well as the Indo-French Centre for Applied Mathematics (IFCAM) for the funding provided during that period. The author is presently a member of Institute of Mathematics, University of Würzburg where the present article is finalized. } 
		\address{Sourav Mitra, Institute of Mathematics, University of Würzburg, 97074, Germany}
		\email{sourav.mitra@mathematik.uni-wuerzburg.de,\,Tél: +49 931 31-89531,\,Fax: +49 931 31-80944}
		
		\begin{abstract}
			Our objective is to study the observability and unique continuation property of the adjoint of a linearized compressible fluid structure interaction model in a 2D channel. Concerning the structure we will consider a damped Euler-Bernoulli beam located on a portion of the boundary. In the present article we establish an observability inequality for the adjoint of the linearized fluid structure interaction problem under consideration which in principle is equivalent with the null controllability of the linearized system. As a corollary of the derived observability inequality we also obtain a unique continuation property for the adjoint problem.
		\end{abstract}
		
		\maketitle
		\noindent{\bf{Key words}.} Observability, unique continuation, adjoint, compressible Navier-Stokes, damped beam, fluid-structure, Carleman estimate. 
		\smallskip\\
		\noindent{\bf{AMS subject classifications}.}
		 76N25, 76N10, 93B05, 93B07, 93B18.
		\section{Introduction}
		
		This article deals with the observability and unique continuation properties of the adjoint of a linearized compressible fluid structure interaction problem. In order to introduce our model in a fixed domain it is first important to present the non linear fluid structure interaction dynamics and obtain the linear model via a suitable linearization procedure. We remark that this linearization process is not unique and depend on the structure of the map which we will use to bring the time dependent domain to a fixed reference configuration.
		
		\subsection{Motivation}\label{motivationchap4thesis}
		In this section, we introduce the full non-linear compressible fluid structure interaction model which we aim at studying from the controllability point of view, even though our work is only a preliminary work in this direction. 
		\\
		Our goal here is to explain how, starting from a control problem for a compressible fluid-structure interaction model, we derive a linear model (cf. Section \ref{linear}) which should, in principle, contain some of the main difficulties related to the non-linear model. 
		%
		\subsubsection{The non-linear model}\label{thenonlinearmodel}
		We first define a few notations corresponding to the fluid and the structural domain. Let $d>0$ be a constant and $\Omega=(0,d)\times(0,1).$ We set 
		\begin{equation}\nonumber
		\begin{array}{l}
		\Gamma_{s}=(0,d)\times \{1\},
		\quad
		\Gamma_{\ell}=(0,d)\times\{0\},\quad\Gamma=\Gamma_{s}\cup\Gamma_{\ell}.
		\end{array}
		\end{equation}     
		For a given function $$\beta: \Gamma_{s}\times (0,\infty)\rightarrow (-1,\infty),$$ which will correspond to the displacement of the one dimensional beam, let us denote by $\Omega_{t}$ and $\Gamma_{s,t}$ the following sets
		\begin{equation}\nonumber
		\begin{array}{ll}
		\Omega_{t}=\{(x,y) \suchthat x\in (0,L),\quad 0<y<1+\beta(x,t)\}&=\,\mbox{domain of the fluid at time $t$},
		\vspace{1.mm}\\
		\Gamma_{s,t}=\{(x,y)\suchthat x\in (0,L),\quad y=1+\beta(x,t)\}&=\,\mbox{the beam at time $t$}.
		\end{array}
		\end{equation}    
		The reference configuration of the beam is $\Gamma_{s}$ and we set
		\begin{equation}\label{{setnotchp3}}
		\begin{array}{ll}
		\Sigma_{T}=\Gamma\times (0,T),&\quad\Sigma^{s}_{T}=\Gamma_{s}\times (0,T),\\ \widetilde{\Sigma^{s}_{T}}=\cup_{t\in (0,T)}\Gamma_{s,t}\times \{t\},&\quad\Sigma^{\ell}_{T}=\Gamma_{\ell}\times (0,T),\\  Q_{T}=\Omega\times (0,T),&\quad \widetilde{Q_{T}}=\cup_{t\in(0,T)} \Omega_{t}\times \{t\}.
		\end{array}
		\end{equation}  
		\begin{figure}[h!]
			\centering
			\begin{tikzpicture}[scale=0.75]
			\coordinate (O) at (0,5);
			\coordinate (A) at (10,5);
			\coordinate (B) at (6.5,5);
			\draw (5,0)node [below] {$\Gamma_{\ell}$};
			\draw (0,0)node [below] {$0$};
			\draw (10,0)node [below right] {$L$};
			\draw (0,5)node [above] {$1$};
			\draw (3.2,5.2) node [above right] {$\beta(x,t)$};
			\draw (5,5) node [below] {$\Gamma_{s}$};
			\draw (0,0) -- (10,0);
			\draw (10,0) -- (10,5);
			\draw[dashed] (10,5) -- (0,5);
			\draw (0,5) -- (0,0);
			\draw[color=blue, ultra thick] (O) to [bend left=30] (B);
			\draw[color=blue, ultra thick] (B)  to [bend right=30] (A);
			\draw[<->, color=red, thick] (3.2,5) -- (3.2,6);
			\end{tikzpicture}
			\caption{Domain $\Omega_{t}$.}
		\end{figure} \\ 
		We consider a fluid with density $\rho$ and velocity ${u}.$ The fluid structure interaction system coupling the compressible Navier-Stokes and the damped Euler-Bernoulli beam equation is modeled by the following equations
		\begin{equation}\label{1.1chp3}
		\left\{
		\begin{array}{ll}
		\partial_{t}\rho+\mbox{div}(\rho {u})=0\quad &\mbox{in} \quad \widetilde{Q_{T}},
		\vspace{1.mm}\\
		\rho ( \partial_{t}{ u}+ ({u}.\nabla){u})-\mu\Delta u-(\mu+\mu')\nabla\mbox{div}u +\nabla p(\rho) =0\quad &\mbox{in} \quad \widetilde{Q_T},
		\vspace{1.mm}\\
		\partial_{tt}\beta-\partial_{txx}\beta+\partial_{xxxx}\beta=(T_{f})_{2} \quad& \mbox{on}\quad \Sigma^{s}_{T}.
		\end{array} \right.
		\end{equation}
		We assume that at the fluid structure interface the following impermeability condition holds
		\begin{equation}\label{interface}
		\begin{array}{l}
		{ u}(\cdot,t)\cdot n_{t}=(0,\partial_{t}\beta)\cdot n_{t}\quad \mbox{on}\quad \widetilde{\Sigma^{s}_{T}},
		\end{array}
		\end{equation}
		where ${n}_{t}$ is the outward unit normal to $\Gamma_{s,t}$ given by
		$$
		{ n}_{t}=-\frac{\partial_{x}\beta}{\sqrt{1+(\partial_{x}\beta)^{2}}}\vec{e}_{1}+\frac{1}{\sqrt{1+(\partial_{x}\beta)^{2}}}\vec{e}_{2}, 
		\qquad (\vec{e}_{1}=(1,0) \text{ and }\vec{e}_{2}=(0,1)).
		$$
		The fixed boundary $\Sigma^{\ell}_{T}$ is assumed to be impermeable and here the impermeability condition is given as follows
		\begin{equation}\label{u.n0}
		\begin{array}{l}
		u(\cdot,t)\cdot n=0\quad \mbox{on}\quad \Sigma^{\ell}_{T},
		\end{array}
		\end{equation}
		where $n$ is the unit outward normal to $\Gamma_{\ell}.$ The fluid boundary is supplemented with the following slip condition
		\begin{equation}\label{curl0}
		\begin{array}{l}
		\mbox{curl}({u})=0\quad \mbox{on}\quad \widetilde{\Sigma^{s}_{T}}\cup \Sigma^{l}_{T},
		\end{array}
		\end{equation}
		where $\mbox{curl}u=(\frac{\partial u_{1}}{\partial y}-\frac{\partial u_{2}}{\partial x}),$ denotes the vorticity of the vector field $u.$ In the system \eqref{1.1chp3}, the real constants $\mu,$ $\mu'$ are the Lam\'{e} coefficients which are supposed to satisfy
		$$\mu>0,\quad (\mu'+2\mu)>0.$$ 
		In our case the fluid is isentropic i.e. the pressure $p(\rho)$ is only a function of the fluid density $\rho$ and is given by
		$$p(\rho)=a\rho^{\gamma},$$
		where $a>0$ and $\gamma>1$ are positive constants.\\ 
		We assume that there exists a constant external force $P_{ext}>0$ which acts on the beam. We then introduce the positive constant $\overline{\rho}$ defined by the relation 
		$$ P_{ext}= a\overline{\rho}^{\gamma}.$$
		To incorporate this external forcing term $P_{ext}$ into the system of equations \eqref{1.1chp3}, we introduce the following 
		\begin{equation}\label{1.2chp3}
		P(\rho)=p(\rho)-P_{ext}=a\rho^{\gamma}-a\overline{\rho}^{\gamma}.
		\end{equation}
		The non-homogeneous source term of the beam equation $(T_{f})_{2}$ is the net surface force on the structure which is the resultant of force exerted by the fluid on the structure and the external force $P_{ext}$ and it is assumed to be of the following simplified form
		\begin{equation}\label{1.3chp3}
		(T_{f})_{2}=(-(\mu{'}+2\mu)(\mbox{div}{u}){I}_{d}\cdot {{n}_{t}}+P{{ n}_{t}})\mid_{\Gamma_{s,t}}\sqrt{1+(\partial_{x}\beta)^{2}}\cdot \vec{e}_{2}\quad\mbox{on}\quad \Sigma^{s}_{T},
		\end{equation}
		where ${I}_{d}$ is the identity matrix.
		
		\begin{remark}[The physical model and simplification.] 
			\label{Remark-Force-Beam}
			The stress tensor corresponding to a Newtonian fluid with velocity $u$ and pressure $p$ is of the following form:
			\begin{equation}\label{stress}
			\begin{split}
			\mathbb{S}
			(u,p)=(2\mu D(u)+\mu'{\mathrm{div}}\,uI_{d})-pI_{d},
			\end{split}
			\end{equation}
			where $D(u)$ is the symmetric gradient given by 
			$$D(u)=\frac{1}{2}(\nabla u+\nabla^{T}u).$$
			In view of the expression \eqref{stress} of the stress tensor, the net force acting on the beam should be given as follows:
			\begin{equation}\label{stressor}
			(T_{f})^{ph}_{2}=([-2\mu D({u})-\mu{'}{\div}{u}{I}_{d}]\cdot {{n}_{t}}+P{{ n}_{t}})\mid_{\Gamma_{s,t}}\sqrt{1+(\partial_{x}\beta)^{2}}\cdot \vec{e}_{2}\quad\mbox{on}\quad \Sigma^{s}_{T}.
			\end{equation}
			Instead of using the force \eqref{stressor}, we assume, for technical reasons (see Remark \ref{rem-q-good-unknow}), that the net force acting on the beam is given by \eqref{1.3chp3}. 
			\\
			Although this might seem physically irrelevant, let us point out that the resulting simplified model \eqref{1.1chp3}-\eqref{interface}-\eqref{u.n0}-\eqref{curl0}-\eqref{1.2chp3}-\eqref{1.3chp3} admits an energy equality which is explained in the following.
			\\
			Assuming the data and the unknowns $\rho,$ $u$ and $\beta$ are periodic in the $x$ direction, we can formally derive the following energy dissipation law for the system \eqref{1.1chp3}-\eqref{interface}-\eqref{u.n0}-\eqref{curl0}-\eqref{1.2chp3}-\eqref{1.3chp3} (the detailed computation can be found in Section \cite[p. 211, Appendix]{phdthesis})
			\begin{equation}\label{1.15lchp3}
			\begin{split}
			&\frac{1}{2}\frac{d}{dt}\left(\int\limits_{\Omega_{\beta(t)}}\rho|{ u}|^{2}\,dx\right)+\frac{d}{dt}\left(\int\limits_{\Omega_{\beta(t)}}\frac{a}{(\gamma-1)}\rho^{\gamma}\,dx \right) +\frac{1}{2}\frac{d}{dt}\left(\int\limits_{0}^{L}|\partial_{t}\beta|^{2}\,dx\right)\\
			&+\frac{1}{2}\frac{d}{dt}\left(\int\limits_{0}^{L}|\partial_{xx}\beta|^{2}\,dx\right)
			+\mu\int\limits_{\Omega_{\beta(t)}}| \mathrm{curl}{ u}|^{2}\,dx+(\mu'+2\mu)\int\limits_{\Omega_{\beta(t)}}|\mathrm{div} u|^{2}\,dx\\
			&+\int\limits_{0}^{L}|\partial_{tx}\beta|^{2}\,dx =-P_{ext}\int\limits_{\Gamma_{s}}\partial_{t}\beta.
			\end{split}
			\end{equation}
			Let us also point out that a similar simplified expression of stress tensor was considered in \cite{flori1chp3} and \cite{flori2chp3}.
			We would like to refer the readers to the Remark \ref{extensionremfrcbm} for the technical details behind considering the simplified model \eqref{1.1chp3}-\eqref{interface}-\eqref{u.n0}-\eqref{curl0}-\eqref{1.2chp3}-\eqref{1.3chp3}. 
		\end{remark} 
		\subsubsection{Control problem and extension arguments}\label{Controlproandexten}
		Our goal will be to discuss a control problem with controls acting from the boundary in the $x$-variable. 
		%
		%
		So far, we did not make precise the boundary conditions in the $x$-variable, as the controls we shall consider will precisely act on these boundaries. But in fact, the boundary control functions will never appear explicitly, as we will first do an extension argument in the direction of the channel and then study the distributed controllability for \eqref{1.1chp3}-\eqref{interface}-\eqref{u.n0}-\eqref{curl0}-\eqref{1.2chp3}-\eqref{1.3chp3} in the extended domain with controls localized in the extension of the domain. We thus take $L >0$ and embed $\Gamma_{s}$ into $\mathbb{T}_{L}\times\{1\}$ and $\Omega$ into $\mathbb{T}_{L}\times(0,1)$ where $\mathbb{T}_{L}$ is the one dimensional torus identified with $(-L,d+L)$ with periodic conditions.
		%
		%
		Then we consider the controls $v_{\rho}\chi_{\widetilde{\omega}_t}$ (for the density), $v_{u}\chi_{\widetilde{\omega}_t}$ (for the velocity) and $v_{\beta}\chi_{\widetilde{\omega}_{1,t}}$ (for the beam), where $\chi_{\widetilde\omega_t}$ and $\chi_{\widetilde{\omega}_{1,t}}$ are the characteristics functions of the sets $\widetilde\omega_t$ and $\widetilde{{\omega}}_{1,t}$ which are defined as follows
		\begin{equation}\label{tomom1}
		\begin{array}{l}
		\widetilde{\omega}_t=\{(x,y)\suchthat x\in[-L,0),\,\, 0<y<1+\beta(x,t)\}
		\cup \{(x,y)\suchthat x\in(d,d+L],\\
		\qquad\qquad\qquad\qquad\qquad\qquad\qquad\qquad\qquad\qquad\qquad\qquad 0<y<1+\beta(x,t)\},\\ 
		\widetilde{\omega}_{1,t}=\{(x,y)\suchthat x\in[-L,0),\,\,  y=1+\beta(x,t)\}
		\cup \{(x,y)\suchthat x\in(d,d+L],\\
		\qquad\qquad\qquad\qquad\qquad\qquad\qquad\qquad\qquad\qquad\qquad\qquad\qquad y=1+\beta(x,t)\}.
		\end{array}
		\end{equation}
		To write the control system we further introduce the following notations
		\begin{equation}\nonumber
		\begin{array}{ll}
		\Omega^{ex}_{t}=\{(x,y) \suchthat x\in \mathbb{T}_{L},\quad 0<y<1+\beta(x,t)\}&=\,\mbox{extended domain of the fluid at time $t$},
		\vspace{1.mm}\\
		\Gamma^{ex}_{s,t}=\{(x,y)\suchthat x\in \mathbb{T}_{L},\quad y=1+\beta(x,t)\}&=\,\mbox{the extended beam at time $t$}.
		\end{array}
		\end{equation} 
		Our control system then reads as follows
		\begin{equation}\label{1.1*}
		\left\{
		\begin{array}{ll}
		\displaystyle\partial_{t}\rho+\mbox{div}(\rho {u})=v_{\rho}\chi_{\widetilde\omega_t}\, &\mbox{in} \, \cup_{t\in(0,T)}\Omega^{ex}_{t}\times\{t\},
		\vspace{1.mm}\\
		\displaystyle \rho( \partial_{t}{ u}+({u}.\nabla){u})-\mu\Delta u-(\mu+\mu')\nabla\mbox{div}u +\nabla P(\rho) =v_{u}\chi_{\widetilde\omega_t}\, &\mbox{in} \, \cup_{t\in(0,T)}\Omega^{ex}_{t}\times\{t\},
		\vspace{1.mm}\\
		\displaystyle{u}_{2} =\partial_{t}\beta+\partial_{x}\beta{u}_{1}\, & \mbox{on}\, \cup_{t\in(0,T)}\Gamma^{ex}_{s,t}\times\{t\},
		\vspace{1.mm}\\
		\displaystyle u_{2}=0\, &\mbox{on}\, (\mathbb{T}_{L}\times\{0\})\times(0,T),
		\vspace{1.mm}\\
		\displaystyle \mbox{curl}({u})=0\, &\mbox{on}\, \cup_{t\in(0,T)}\Gamma^{ex}_{s,t}\times\{t\},
		\vspace{1.mm}\\
		\displaystyle \mbox{curl}({u})=0\, &\mbox{on}\,(\mathbb{T}_{L}\times\{0\})\times(0,T) ,
		\vspace{1.mm}\\
		\displaystyle{ u}(\cdot,0)={ u}_{0}\, & \mbox{in} \, \mathbb{T}_{L}\times(0,1),
		\vspace{1.mm}\\
		\displaystyle\rho(\cdot,0)=\rho_{0}\, &\mbox{in}\, \mathbb{T}_{L}\times(0,1),
		\vspace{1.mm}\\
		\displaystyle\partial_{tt}\beta-\partial_{txx}\beta+\partial_{xxxx}\beta=(T_{f})_{2}+v_{\beta}\chi_{\widetilde\omega_{1,t}} \,& \mbox{on}\, (\mathbb{T}_{L}\times\{1\})\times(0,T),
		\vspace{1.mm}\\
		\displaystyle\beta(\cdot,0)=\beta_{0}\quad \mbox{and}\quad \partial_{t}\beta(\cdot,0)=\beta_{1}\, &\mbox{in}\, \mathbb{T}_{L}\times\{1\}.
		\end{array} \right.
		\end{equation} 
		It is standard to deduce a boundary controllability result for the system \eqref{1.1chp3}-\eqref{interface}-\eqref{u.n0}-\eqref{curl0}-\eqref{1.2chp3}-\eqref{1.3chp3} from a controllability result of the system \eqref{1.1*} by restricting the data at the boundaries in the $x$-variable.
		\subsubsection{Transformation of the problem to a fixed domain}
		\label{Subsubsec-fixed}
		To transform the system \eqref{1.1*} in the reference configuration, for $\beta$ satisfying
		$
		1+\beta(x,t)>0$ for all $(x,t)\in \mathbb{T}_{L}\times(0,T),
		$
		we introduce the following changes of variables
		\begin{equation}\label{1.14chp3}
		\begin{array}{l}
		\displaystyle
		{\Phi}_{\beta(t)}:\Omega^{ex}_{t}\longrightarrow \mathbb{T}_{L}\times(0,1)\quad\mbox{defined by}\\
		\qquad\qquad\qquad\qquad\qquad\qquad\displaystyle {\Phi}_{\beta(t)}(x,y)=(x,z)=\left(x,\frac{y}{1+\beta(x,t)}\right),\\
		\displaystyle
		{\Phi}_{\beta}:\cup_{t\in(0,T)}\Omega^{ex}_{t}\times\{t\}\longrightarrow (\mathbb{T}_{L}\times(0,1))\times(0,T)\quad\mbox{defined by}\\
		\qquad\qquad\qquad\qquad\qquad\qquad\displaystyle{\Phi}_{\beta}(x,y,t)=(x,z,t)=\left(x,\frac{y}{1+\beta(x,t)},t\right).
		\end{array} 
		\end{equation}
		\begin{remark}
			It is easy to prove that for each $t\in[0,T),$ the map ${\Phi}_{\beta(t)}$ is a $C^{1}-$ diffeomorphism from $\Omega^{ex}_{t}$ onto $\mathbb{T}_{L}\times(0,1)$ provided that $(1+\beta(x,t))>0,$ for all $x\in\mathbb{T}_{L}$ and that $\beta(\cdot,t)\in C^{1}(\Gamma_{s}).$
		\end{remark}
		Observe that the map ${\Phi}_{\beta(t)}$ can be uniquely extended to the boundary $\Gamma^{ex}_{s,t}$ with values in $\mathbb{T}_{L}\times\{1\},$ by using the same formula \eqref{1.14chp3}$_{1}.$ With the change of variable $\Phi_{\beta(t)}$ (introduced in \eqref{1.14chp3}), the control zones in the reference configuration are written as follows
		\begin{equation}\label{conznf}
		\begin{array}{l}
		\omega=((-L,0)\times(0,1))\cup((d,d+L]\times(0,1))=\{\Phi_{\beta(t)}(x,y)\suchthat (x,y)\in \widetilde{\omega}_t\},\\
		\omega_{1}=((-L,0)\times\{1\})\cup((d,d+L]\times\{1\})=\{\Phi_{\beta(t)}(x,y)\suchthat (x,y)\in \widetilde{\omega}_{1,t}\}.
		\end{array}
		\end{equation} 
		We set the following notations
		\begin{equation}\label{notcv}
		\begin{array}{l}
		\widehat{\rho}(x,z,t)=\rho(\Phi^{-1}_{\beta}(x,z,t)),\,\,\widehat{{ u}}(x,z,t)=(\widehat{u}_{1},\widehat{u}_{2})={u}(\Phi^{-1}_{\beta}(x,z,t)),\\
		\qquad\qquad\qquad\qquad\qquad\qquad\qquad\qquad\qquad\forall  (x,z,t)\in (\mathbb{T}_{L}\times(0,1))\times(0,T),\\
		\widehat{\rho}_{0}(x,z)=\rho_{0}(\Phi^{-1}_{\beta(0)}(x,z)),\,\,\widehat{u}_{0}(x,z)=u_{0}(\Phi^{-1}_{\beta(0)}(x,z)),\quad\forall (x,z)\in\mathbb{T}_{L}\times(0,1),\\
		v_{\widehat{\rho}}\chi_{\omega}(x,z,t)=v_{\rho}\chi_{\widetilde{\omega}}(\Phi^{-1}_{\beta}(x,z,t)),\,\,v_{\widehat{u}}\chi_{\omega}(x,z,t)=v_{u}\chi_{\widetilde{\omega}}(\Phi^{-1}_{\beta}(x,z,t)),\\
		\qquad\qquad\qquad\qquad\qquad\qquad\qquad\qquad\qquad\forall (x,z,t)\in (\mathbb{T}_{L}\times(0,1))\times(0,T),\\
		v_{\beta}\chi_{{\omega}_{1}}(x,1,t)=v_{\beta}\chi_{\widetilde{\omega}_{1}}(\Phi^{-1}_{\beta}(x,1,t)),\quad\forall x\in\mathbb{T}_{L}\,\,\mbox{and}\,\,\forall t\in(0,T).
		\end{array}
		\end{equation}
		After transformation the nonlinear control problem \eqref{1.1*} is rewritten as following 
		\begin{equation}\label{1.16chp3}
		\left\{
		\begin{array}{ll}
		\partial_{t}{\widehat{\rho}}+\begin{bmatrix}
		\widehat {u}_{1}\\
		\frac{1}{(1+\beta)}(\widehat {u}_{2}-z\partial_{t}\beta-z\widehat{u}_{1}\partial_{x}\beta) 
		\end{bmatrix}
		\cdot \nabla\widehat{\rho}+\widehat{\rho}\mbox{div}\widehat{u}&\\
		\qquad\qquad\qquad\qquad\qquad\qquad\qquad={F}_{1}(\widehat{\rho},\widehat{u},\beta)+v_{\widehat{\rho}}\chi_{\omega}\,& \mbox{in}\, (\mathbb{T}_{L}\times(0,1))\times(0,T),
		\vspace{1.mm}\\
		\widehat\rho (\partial_{t}\widehat {u}+(\widehat{u}\cdot\nabla)\widehat{u})-\mu\Delta \widehat {u}-(\mu'+\mu)\nabla(\mbox{div}\widehat {u})+ \nabla  P(\widehat{\rho})\\[2.mm] \qquad\qquad\qquad\qquad\qquad\qquad\qquad\qquad\quad={F}_{2}(\widehat \rho,\widehat {u},\beta)+v_{\widehat{u}}\chi_{\omega}\, &\mbox{in} \, (\mathbb{T}_{L}\times(0,1))\times(0,T),
		\vspace{1.mm}\\
		\widehat{u}\cdot n=\widehat{u}_{2}=\partial_{t}\beta+\partial_{x}\beta\widehat{u}_{1}\,& \mbox{on}\,(\mathbb{T}_{L}\times\{1\})\times(0,T),\\[1.mm]
		\widehat {u}(\cdot,t)\cdot n=0\,& \mbox{on}\, (\mathbb{T}_{L}\times\{0\})\times(0,T),
		\vspace{1.mm}\\
		\displaystyle\mbox{curl}\widehat{u}=\frac{\beta\partial_{z}\widehat{u}_{1}}{(1+\beta)}-\frac{z\partial_{x}\beta\partial_{z}\widehat{u}_{2}}{(1+\beta)}\,&\mbox{on}\,((\mathbb{T}_{L}\times\{0,1\})\times(0,T),
		\vspace{1.mm}\\
		\widehat {u}(\cdot,0)=\widehat{u}_{0}\,& \mbox{in} \, \mathbb{T}_{L}\times(0,1),
		\vspace{1.mm}\\
		\widehat{\rho}(\cdot,0)={\widehat\rho_{0}}\,& \mbox{in}\, \mathbb{T}_{L}\times(0,1),
		\vspace{1.mm}\\
		\partial_{tt}\beta-\partial_{txx}\beta+\partial_{xxxx}\beta\\
		\qquad\quad=\big(-(\mu'+2\mu)\mbox{div}u+P(\widehat{\rho})\big)+F_{3}(\widehat{\rho},\widehat{u},{\beta})+v_{\beta}\chi_{\omega_{1}} \,& \mbox{on}\, (\mathbb{T}_{L}\times\{1\})\times(0,T),
		\vspace{1.mm}\\
		\beta(0)=\beta_{0}\quad \mbox{and}\quad\partial_{t}\beta(0)=\beta_{1}\,&\mbox{in}\, \mathbb{T}_{L}\times\{1\},
		\end{array} \right.
		\end{equation}
		where the non homogeneous terms $F_{1}(\widehat{\rho},\widehat{u},\beta),$ $F_{2}(\widehat{\rho},\widehat{u},\beta)$ and $F_{3}(\widehat{\rho},\widehat{u},\beta),$ are non linear in its arguments and are given by
		\begin{equation}\label{F123chp3}
		\begin{aligned}
		\displaystyle
		&{F}_{1}(\widehat {\rho},\widehat {u},\beta)=\frac{1}{(1+\beta)}\big(z\partial_{z}\widehat{u}_{1}\partial_{x}\beta\widehat{\rho}+\beta\widehat{\rho}\partial_{z}\widehat{u}_{2}\big),\\
		&{F}_{2}(\widehat \rho,\widehat { u},\beta)=-\beta\widehat{\rho}\partial_{t}\widehat { u}+z\widehat\rho\partial_{z}\widehat {u}\partial_{t}\beta-\beta\widehat \rho\widehat {u}_{1}\partial_{x}\widehat {u}+\widehat { u}_{1}\partial_{z}\widehat {u}\partial_{x}\beta\widehat \rho z+\mu \big(\beta\partial_{xx}\widehat {u}-\frac{\beta \partial_{zz}\widehat {u}}{(1+\beta)}\\
		&-2\partial_{x}\beta z\partial_{xz}\widehat {u}
		+\frac{\partial_{zz}\widehat {u}z^{2}(\partial_{x}\beta)^{2}}{(1+\beta)}
		+\partial_{z}\widehat {u}\big( \frac{(1+\beta)z\partial_{xx}\beta-2(\partial_{x}\beta)^{2}z}{(1+\beta)}\big)\big)+(\mu+\mu')\cdot\\
		& \begin{bmatrix}
		\displaystyle\beta\widehat \partial_{xx}{u}_{1}-\partial_{xz}\widehat {u}_{1}z\partial_{x}\beta-\partial_{x}\beta z\big(\partial_{xz}\widehat { u}_{1}-\frac{\partial_{zz}\widehat {u}_{1}z\partial_{x}\beta}{(1+\beta)}\big)+\partial_{z}\widehat { u}_{1}\big(\frac{(1+\beta)z\partial_{xx}\beta-2(\partial_{x}\beta)^{2}z}{(1+\beta)}\big)\\\displaystyle -\frac{\partial_{x}\beta\partial_{z}\widehat {u}_{2}}{(1+\beta)}-\frac{\partial_{x}\beta z\partial_{zz}\widehat { u}_{2}}{(1+\beta)}\\[8.mm]
		\displaystyle-\frac{\partial_{x}\beta\partial_{z}\widehat {u}_{1}}{(1+\beta)}-\frac{\partial_{x}\beta z\partial_{zz}\widehat { u}_{1}}{(1+\beta)}-\frac{\beta\partial_{zz}\widehat {u}_{2}}{(1+\beta)}
		\end{bmatrix}\\
		&  -(\beta \partial_{x}P(\widehat{\rho})-\partial_{z}P(\widehat{\rho})z\partial_{x}\beta)\vec{e_{1}},\smallskip\\
		& F_{3}(\widehat{\rho},\widehat {u},\beta)=(\mu'+2\mu)\left(\frac{z\partial_{z}\widehat{u}_{1}\partial_{x}\beta}{(1+\beta)}-\frac{\beta\partial_{z}\widehat{u}_{2}}{(1+\beta)}\right),
		\end{aligned}
		\end{equation}
	and $n$ denotes the unit normal to the boundary $(\mathbb{T}_{L}\times\{0,1\})$ of the extended reference domain $(\mathbb{T}_{L}\times(0,1)),$ $i.e.$
		\begin{equation}\nonumber
		n=\left\{ \begin{array}{ll}
		(0,1)\quad&\mbox{on}\quad \Gamma_{s},\\
		(0,-1)\quad&\mbox{on}\quad \Gamma_{l}.
		\end{array}\right.
		\end{equation}
		Since $\beta=0$ on $\mathbb{T}_{L}\times\{0\},$ (as we do not consider a structure on $\mathbb{T}_{L}\times\{0\}$) \eqref{1.16chp3}$_{5}$ of course implies that $\curl\widehat{u}=0$ on $\mathbb{T}_{L}\times\{0\}.$
		\subsubsection{A control problem for the linearized model around the stationary state $(\overline{\rho},\overline{u},0)$}\label{linear}
		We first pose the question of local exact controllability to the steady state $(\overline{\rho},\overline{u},0)$, where
		\begin{equation}\label{baru}
		\begin{array}{l}
		\overline{\rho}>0\,\,\mbox{is such that }\,\,P_{ext}=a\overline{\rho}^{\gamma}\,\,\mbox{and}\,\,\overline{u}=\begin{pmatrix}
		\overline{u}_{1}\\0
		\end{pmatrix}, \,\,\overline{u}_{1}>0\,\,\mbox{is a constant},
		\end{array}
		\end{equation}
		which is obviously a stationary solution of the system \eqref{1.1chp3}-\eqref{interface}-\eqref{u.n0}-\eqref{curl0}-\eqref{1.2chp3}-\eqref{1.3chp3} or \eqref{1.1*} without the control functions.
		\\
		To be more precise, if $(\rho_0, u_0, \beta_0, \beta_1)$ is close in a suitable topology to $(\overline\rho, \overline u, 0, 0)$, then the question of local exact controllability of \eqref{1.1*}, aims in finding control functions $(v_\rho, v_u, v_\beta)$ such that the solution of \eqref{1.1*} satisfies $(\rho(T), u(T), \beta(T), \partial_t \beta(T)) = (\overline \rho, \overline u, 0, 0)$. 
		\\
	Since we give a sense to the solution of system \eqref{1.1*} by posing it into a fixed domain, i.e the system \eqref{1.16chp3} we will rather talk about the controllability of $(\widehat{\rho},\widehat{u},\beta,\partial_{t}\beta)$ around $(\overline{\rho},\overline{u},0,0).$ Of course $(\overline{\rho},\overline{u},0)$ is also a stationary solution to \eqref{1.16chp3} without the control functions.\\ 
		The following change of unknowns
		\begin{equation}\label{coun}
		\begin{array}{l}
		\widetilde{\sigma}=\widehat{\rho}-\overline{\rho},\quad \widetilde{u}=\widehat{u}-\overline{u},\quad \beta=\beta-0, 
		\end{array}
		\end{equation}
		reduces the controllability problem to the state $(\overline\rho, \overline u,0,0)$ for $(\widehat\rho, \widehat u, \beta, \partial_t \beta)$ into a local null controllability problem for $(\widetilde\sigma, \widetilde u, \beta, \partial_t \beta)$.
		\\
		We further introduce the following notations corresponding to the control functions which are consistent with the new unknowns defined in \eqref{coun}.
		\begin{equation}\label{nconn}
		\begin{array}{l}
		v_{\widetilde{\sigma}}\chi_{\omega}=v_{\widehat{\rho}}\chi_{\omega},\quad v_{\widetilde{u}}\chi_{\omega}=v_{\widehat{u}}\chi_{\omega}.
		\end{array}
		\end{equation}
		In fact, as we would like to obtain a local null-controllability problem for $(\widetilde\sigma, \widetilde u, \beta, \partial_t \beta)$, it seems reasonable to start by considering the linearized problem around the state $(0,0,0,0)$.
		Starting from system \eqref{1.16chp3} and the non-linear terms \eqref{F123chp3} and dropping all the non-linear terms in $(\widetilde \sigma, \widetilde u, \beta, \partial_t \beta)$, we obtain 
		\begin{equation}\label{1.16llnr}
		\left\{
		\begin{array}{ll}
		\displaystyle
		\partial_{t}{\widetilde{\sigma}}+\overline{u}_{1}\partial_{x}\widetilde{\sigma}+\overline{\rho}\mbox{div}\widetilde{u}=v_{\widetilde{\sigma}}\chi_{\omega}\,& \mbox{in}\, (\mathbb{T}_{L}\times(0,1))\times(0,T),
		\vspace{1.mm}\\
		\displaystyle
		\overline{\rho}(\partial_{t}\widetilde {u}+\overline{u}_{1}\partial_{x}\widetilde{u})-\mu\Delta \widetilde {u}-(\mu'+\mu)\nabla(\mbox{div}\widetilde {u})\\
		\qquad\qquad\qquad\qquad\qquad+ P'(\overline{\rho})\nabla\widetilde{\sigma}=v_{\widetilde{u}}\chi_{\omega}\, &\mbox{in} \, (\mathbb{T}_{L}\times(0,1))\times(0,T),
		\vspace{1.mm}\\
		\widetilde{u}\cdot n=\widetilde{u}_{2}=\partial_{t}\beta+\overline{u}_{1}\partial_{x}\beta\,& \mbox{on}\,(\mathbb{T}_{L}\times\{1\})\times(0,T),\\[1.mm]
		\widetilde {u}(\cdot,t)\cdot n=0\,& \mbox{on}\, (\mathbb{T}_{L}\times\{0\})\times(0,T),
		\vspace{1.mm}\\
		\displaystyle\mbox{curl}\widetilde{u}=0\,&\mbox{on}\,((\mathbb{T}_{L}\times\{0,1\}) \times(0,T),
		\vspace{1.mm}\\
		\widetilde{u}(\cdot,0)=\widehat{u}_{0}-\overline{u}=\widetilde{u}_{0}\,& \mbox{in} \, \mathbb{T}_{L}\times(0,1),
		\vspace{1.mm}\\
		\widetilde{\sigma}(\cdot,0)=\widehat{\rho}_{0}-\overline{\rho}={\widetilde\sigma_{0}}\,& \mbox{in}\, \mathbb{T}_{L}\times(0,1),
		\vspace{1.mm}\\
		\partial_{tt}\beta-\partial_{txx}\beta+\partial_{xxxx}\beta\\
		\qquad=\big(-(\mu'+2\mu)\mbox{div}\widetilde{u}+P'(\overline{\rho})\widetilde{\sigma})+v_{\beta}\chi_{\omega_{1}} \,& \mbox{on}\, (\mathbb{T}_{L}\times\{1\})\times(0,T),
		\vspace{1.mm}\\
		\beta(0)=\beta_{0}\quad \mbox{and}\quad\partial_{t}\beta(0)=\beta_{1}\,&\mbox{in}\, \mathbb{T}_{L}\times\{1\}.
		\end{array} \right.
		\end{equation}
		In the present article we will prove an observability inequality corresponding to the adjoint of the system \eqref{1.16llnr}. The observability inequality proved here implies by duality the null controllability of the linear system \eqref{1.16llnr}. But the observation will be obtained using strong norms of the unknowns which will only allow the attainment of the null controllability result for \eqref{1.16llnr} in a very weak sense ($i.e$ in some negative order Sobolev spaces), which are not enough to pass from the null controllability of the linear system \eqref{1.16llnr} to the local exact controllability of the non linear model \eqref{1.16chp3}. This is the reason why we will only present the observability result for the adjoint of \eqref{1.16llnr} without stating the corresponding controllability of \eqref{1.16llnr} in negative order Sobolev spaces.  
		%
		
		\subsection{Main result: Observability of the adjoint of \eqref{1.16llnr}}
		In order to study the null controllability of the linearized problem \eqref{1.16llnr}, the classical strategy is to prove the observability of the adjoint system of \eqref{1.16llnr}. Here, the adjoint system of \eqref{1.16llnr}, computed with respect to the scalar product $L^2(\mathbb{T}_{L}\times(0,1)) \times L^2(\mathbb{T}_{L}\times(0,1)) \times  L^2(\mathbb{T}_{L}\times\{1\}) \times L^2(\mathbb{T}_L\times\{1\})$, reads as follows:
		\begin{equation}\label{adjsys}
		\left\{ \begin{array}{ll}
		\displaystyle
		-\partial_{t}\sigma-\overline{u}_{1}\partial_{x}\sigma-P'(\overline{\rho})\mbox{div}v=0\,& \mbox{in}\, (\mathbb{T}_{L}\times(0,1))\times(0,T),
		\vspace{1.mm}\\
		\displaystyle
		-\overline{\rho}(\partial_{t}v+\overline{u}_{1}\partial_{x}v)-\mu\Delta v-(\mu'+\mu)\nabla(\mbox{div}v)-\overline{\rho}\nabla\sigma=0\, &\mbox{in} \, (\mathbb{T}_{L}\times(0,1))\times(0,T),
		\vspace{1.mm}\\
		v\cdot n={v}_{2}=\psi\,& \mbox{on}\,(\mathbb{T}_{L}\times\{1\})\times(0,T),\\[1.mm]
		{v}(\cdot,t)\cdot n=0\,& \mbox{on}\, (\mathbb{T}_{L}\times\{0\})\times(0,T),
		\vspace{1.mm}\\
		\displaystyle\mbox{curl}v=0\,&\mbox{on}\,((\mathbb{T}_{L}\times\{0,1\})\times(0,T),
		\vspace{1.mm}\\
		v(\cdot,T)=v_{T}\,& \mbox{in} \, \mathbb{T}_{L}\times(0,1),
		\vspace{1.mm}\\
		{\sigma}(\cdot,T)=\sigma_{T}\,& \mbox{in}\, \mathbb{T}_{L}\times(0,1),
		\vspace{1.mm}\\
		\partial_{tt}\psi+\partial_{txx}\psi+\partial_{xxxx}\psi\\
		\qquad \qquad=(\partial_{t}+\overline{u}_{1}\partial_{x})[(\mu'+2\mu)\mbox{div}\,v+\overline{\rho}\sigma] \,& \mbox{on}\, (\mathbb{T}_{L}\times\{1\})\times(0,T),
		\vspace{1.mm}\\
		\psi(T)=\psi_{T}\quad \mbox{and}\quad\partial_{t}\psi(T)=\psi_{T}^{1}\,&\mbox{in}\, \mathbb{T}_{L}\times\{1\}.
		\end{array}\right.
		\end{equation}
		The well-posedness of system \eqref{adjsys} is stated as the following result, which is proved in Section \ref{lemin}:
		\begin{thm}\label{lemincreg}
			Let 
			\begin{equation}\label{indatadj}
			\begin{split}
			(\sigma_{T},v_{T},\psi_{T},\psi^{1}_{T})\in H^{2}(\mathbb{T}_{L}\times(0,1))\times H^{3}(\mathbb{T}_{L}\times(0,1))&\times H^{3}(\mathbb{T}_{L}\times\{1\})\\
			&\times H^{1}(\mathbb{T}_{L}\times\{1\})
			\end{split}
			\end{equation}
			and the following compatibility relations hold
			\begin{equation}\label{newcompat}
			\begin{array}{llll}
			&(i)\,\,(a)& v_{T}\cdot n=(v_{T})_{2}=\psi_{T},\,\,&\mbox{on}\,\, \mathbb{T}_{L}\times\{1\},\\
			&\quad\,\,\,(b)& v_{T}\cdot n=(v_{T})_{2}=0,\,\,&\mbox{on}\,\, \mathbb{T}_{L}\times\{0\},\\
			&(ii)\,\,& \mathrm{curl}\,v_{T}=0,\,\,&\mbox{on}\,\,\mathbb{T}_{L}\times\{0,1\}\\
			&(iii)\,\,(a)& -\overline{u}_{1}\partial_{x}(v_{T})_{2}+\frac{\mu}{\overline{\rho}}\Delta(v_{T})_{2}+\frac{(\mu+\mu')}{\overline{\rho}}\partial_{z}(\mathrm{div}\,v_{T})\\
			&&-\partial_{z}\sigma_{T}=\psi^{1}_{T}& \mbox{on}\,\, \mathbb{T}_{L}\times\{1\},\\[2.mm]
			&\quad\quad\,(b)& -\overline{u}_{1}\partial_{x}(v_{T})_{2}+\frac{\mu}{\overline{\rho}}\Delta(v_{T})_{2}+\frac{(\mu+\mu')}{\overline{\rho}}\partial_{z}(\mathrm{div}\,v_{T})\\
			&&-\partial_{z}\sigma_{T}=0&\mbox{on}\,\, \mathbb{T}_{L}\times\{0\}.\\
			\end{array}
			\end{equation}
			Then the system \eqref{adjsys} admits a unique solution $(\sigma,v,\psi)$ which satisfies the following regularity
			\begin{equation}\label{imregv}
			\left\{ \begin{array}{ll}
			& \sigma\in C^{0}([0,T];H^{2}(\mathbb{T}_{L}\times(0,1)))\cap C^{1}([0,T];H^{1}(\mathbb{T}_{L}\times(0,1))),\\
			& v\in L^{2}(0,T;H^{3}(\mathbb{T}_{L}\times(0,1)))\cap H^{1}(0,T;H^{2}(\mathbb{T}_{L}\times(0,1)))\\
			&\qquad\cap H^{2}(0,T;L^{2}(\mathbb{T}_{L}\times(0,1))),\\
			& \psi\in L^{2}(0,T;H^{4}(\mathbb{T}_{L}\times\{1\}))\cap H^{1}(0,T;H^{2}(\mathbb{T}_{L}\times\{1\}))\\
			&\qquad\cap H^{2}(0,T;L^{2}(\mathbb{T}_{L}\times\{1\})).
			\end{array}\right.
			\end{equation} 
		\end{thm} 
		\begin{remark}
			A similar well posedness result can be proved for the linear primal problem \eqref{1.16llnr} which we do not state here. For the statement and the proof of this result one can consult \cite[Theorem 4.1.3, p. 153]{phdthesis} and \cite[Sec 4.2.2, p. 172]{phdthesis}.
		\end{remark}
		%
		The central result of the present article is the observability inequality of the adjoint system \eqref{adjsys}:
		\begin{thm}\label{main3}
			Let $(\overline{\rho},\overline{u},0)$ be as in \eqref{baru}, $T>0$ be such that 
			\begin{equation}
			\label{Tgrtr}
			T  >\frac{d}{\overline{u}_{1}},	
			\end{equation}
			and 
			\begin{equation}\label{fixL}
			\begin{array}{l}
			L=3\overline{u}_{1}T>0.
			\end{array}
			\end{equation}
			\\
			There exists a positive constant $C$ such that for all
			\begin{equation}\label{Thmregas}
			\begin{split}
			(\sigma_{T},v_{T},\psi_{T},\psi^{1}_{T})\in H^{2}(\mathbb{T}_{L}\times(0,1))\times H^{3}(\mathbb{T}_{L}\times(0,1))&\times H^{3}(\mathbb{T}_{L}\times\{1\})\\
			&\times H^{1}(\mathbb{T}_{L}\times\{1\}),
			\end{split}
			\end{equation}
			satisfying the compatibility conditions \eqref{newcompat},
			then the solution $(\sigma,v,\psi)$ of the problem \eqref{adjsys} (in the sense of Theorem \ref{lemincreg}) satisfies the following observability inequality:
			\begin{multline}\label{obsinq}
			\| \sigma(\cdot, 0) \|_{H^1(\mathbb{T}_L \times (0,1))}
			+
			\| v (\cdot, 0)\|_{H^2(\mathbb{T}_L \times (0,1))}
			+
			\|( \psi(\cdot, 0), \partial_t \psi(\cdot, 0)) \|_{H^3(\mathbb{T}_L \times \{1\})) \times H^1(\mathbb{T}_L \times \{1\})}
			\\
			\leq 
			C \| \psi \|_{L^2(\omega_1^T)}
			+ 
			C \| v \|_{L^2(0,T; H^2(\omega)) \cap H^1(0,T; H^{1}(\omega))}
			+C\|\sigma\|_{L^{2}(0,T;H^{1}(\omega))}.
			\end{multline}
			where the notations $\omega$ and $\omega_{1}$ for the observation sets were introduced in \eqref{conznf} and $\omega_T = \omega \times (0,T)$, $\omega_1^T = \omega_1 \times (0,T)$.
		\end{thm}
		$\mathit{Comments\,\,on\,\, the\,\, choice\,\, of\,\,T\,\,and\,\,L\,\,in\,\,Theorem\,\,\ref{main3}}$: 
		We recall \eqref{Tgrtr} and \eqref{fixL}. The condition \eqref{Tgrtr} means that the time of observability should be greater than the time taken to cross the channel length $d$ by a particle moving with a velocity $(\overline{u}_{1},0).$ This condition is imposed due to the hyperbolic nature of the transport equation satisfied by $\sigma$ in the system \eqref{adjsys}. In fact this condition plays a key role in obtaining the observability estimate for a hyperbolic transport equation in Section \ref{sectrans}. Our proof in Section \ref{sectrans} depends on a duality argument and a controllability result for the dual to the problem considered in Section \ref{sectrans} which is obtained in \cite{ervguglachp3}. Hence one can consult \cite{ervguglachp3} for a more explicit construction of the controlled trajectory where the restriction of the final time \eqref{Tgrtr} comes into play.\\
		The role of \eqref{fixL} can be better explained after we define some Carleman weights in Section \ref{Sec-CarlemanS}. Hence we refer the readers to the Remark \ref{roleLst} for the explanation behind choosing $L$ as in \eqref{fixL}. Still, let us mention that from the control point of view, the value of $L>0$ does not play any role in the restriction argument presented in Section \ref{Controlproandexten}.\\
		The following is a corollary to Theorem \ref{main3} and corresponds to the unique continuation property for the system \eqref{adjsys}:
		\begin{corollary}\label{uniquecontinuation}
			Let $(\overline{\rho},\overline{u},0)$ be as in \eqref{baru}, $T>0$ and $L>0$ satisfies \eqref{Tgrtr} and \eqref{fixL}. Further let $(\sigma_{T},v_{T},\psi_{T},\psi^{1}_{T})$ satisfies \eqref{Thmregas}. If the solution $(\sigma,v,\psi)$ of the problem \eqref{adjsys} (in the sense of Theorem \ref{lemincreg}) solves $(\sigma,v)=(0,0)$ in $\omega_{T}$ and $\psi=0$ in $\omega^{T}_{1},$ then
			$$(\sigma,v,\psi)=0\quad\mbox{in}\quad ((\mathbb{T}_{L}\times(0,1))\times(0,T))^{2}\times((\mathbb{T}_{L}\times\{1\})\times(0,T)),$$
			where $\omega_T = \omega \times (0,T)$, $\omega_1^T = \omega_1 \times (0,T).$
		\end{corollary}
		The proof of Corollary \ref{uniquecontinuation} will follow from an intermediate step in the proof of Theorem \ref{main3}. The proof is included in Section \ref{proofucont}.  
		\begin{remark}
			Roughly speaking observability inequality is a quantified version of unique continuation property. In that sense observability inequality is stronger than unique continuation property. As a special case the results Theorem \ref{main3} and Corollary \ref{uniquecontinuation} implies the observability and unique continuation of the adjoint of linearized compressible Navier-Stokes equations in a 2D channel where the fluid velocity satisfies the Navier-slip boundary condition without friction at the lateral boundaries $i.e$ 
			$$\,v\cdot n=0\quad\mbox{and}\quad(2\mu D(v)+\mu'\mbox{div}\,v I_{d})n\cdot \vec{\tau}=0
			\,\,\mbox{on}\,\,
			(\mathbb{T}_{L}\times\{0,1\}),$$
			where $n$ and $\vec{\tau}$ respectively denotes the normal and tangent to the boundary. To the best of our knowledge this result is new in itself. The only articles so far dealing with the controllability and observability issues of compressible Navier-Stokes equations in dimension more than one are \cite{ervguglachp3} and \cite{molina} where the problem is posed in a torus. 
		\end{remark}
		\subsection{Ideas and Strategy} 
		Now we will briefly discuss our ideas and strategy to prove Theorem \ref{main3}.\\
		The underlying idea behind the proof of Theorem \eqref{main3} is the identification of the suitable unknowns to track down the dynamics of $(\sigma,v,\psi).$ It is well known that the coupling of $\sigma$ and $v$ is strong. When considering the primal problem \eqref{1.16llnr}, the dynamics between $\widetilde{\sigma},$ $\widetilde{u}$ can be made simpler by introducing the effective viscous flux, see \cite{lionschp3} and \cite{feireislchp3}. For the adjoint problem, a similar quantity, already used in \cite{ervguglachp3}, also simplifies the description of the dynamics:
		\begin{equation}\label{dualefvifl}
		\begin{array}{l}
		q=(\mu'+2\mu)\mbox{div}\,v+\overline{\rho}\sigma.
		\end{array}
		\end{equation}
		This can be termed as the dual version of the effective viscous flux.  Now in our case it is important to identify the behavior of $q$ at the boundaries and specially at the fluid solid interface. This way we obtain a closed loop system solved by $(\sigma,q,\psi).$ For details we refer the readers to Section \ref{Sec-Well-posed}. One can in particular look into the system \eqref{adjsysq} to observe that unlike the coupling between $\sigma$ and $v$ in system \eqref{adjsys}, the coupling between $\sigma$ and $q$ is of lower order. Also it is easier to deal with $(\sigma,q,\psi),$ since it has less degrees of freedom in comparison with $(\sigma,v,\psi).$ We use this new set of unknowns $(\sigma,q,\psi)$ both to prove the well posedness result stated in Theorem \ref{lemincreg} and the observability Theorem \ref{main3}. In Section \ref{Sec-Well-posed} we prove Theorem \ref{lemincreg}.\\ 
		Next we focus in proving an observability inequality for the system satisfied by $(\sigma,q,\psi).$ In that direction we first separately study the observability inequalities of some scalar equations $i.e$ an adjoint damped beam equation, an adjoint heat equation and an adjoint transport equation. The observability estimates for the adjoint damped beam and the adjoint heat equations rely on Carleman estimates while for the adjoint transport equation we use a duality argument and some controllability estimates motivated from \cite{ervguglachp3}. The main difficulty here is to obtain these separate observability estimates with a single goal of combining them suitably to obtain an observability for the coupled system solved by $(\sigma,q,\psi).$ For the parabolic hyperbolic couplings this question is handled in the articles \cite{albano}, \cite{ervguglachp3} (for compressible Navier-Stokes equations), \cite{molina} (for compressible heat conducting fluid) and \cite{silvazz} (for damped viscoelasticity equations). The idea is to use compatible weight functions for the parabolic and hyperbolic equations so that the resulting observability estimates can be suitably combined. In our case along with a parabolic hyperbolic coupling there is a direct coupling between $q$ and $\psi$ at the fluid boundary (see the system \eqref{adjsysq} for details). Hence to use the ideas from \cite{albano}, \cite{silvazz} and \cite{ervguglachp3} in our case there is not many options other than considering a one dimensional weight function, since the beam is one dimensional. We introduce such  weight function in Section \ref{Conw8fn}. Then using this weight function we state a Carleman estimate for the adjoint damped beam equation, see Section \ref{Carldbeam} for details. This Carleman estimate is taken from a very recent article \cite{Carlemanbeam} or more appropriately from \cite[p. 182, Section 4.3.2]{phdthesis}. Then using the same weight function for an adjoint heat equation with Neumann boundary condition we recover a Carleman estimate proved in \cite{farnan}. Of course the weight functions considered in \cite{farnan} and the present article are not the same. This Carleman estimate is included in Section \ref{Carlheat}. In the beginning of Section \ref{Carlheat}, we also point out the difference between our weight function used in proving the Carleman estimate for an adjoint heat equation with the one standard in the literature. Then the same weight function is used to obtain an observability estimate for an adjoint transport equation in Section \ref{sectrans}. This is done first by obtaining some controllability estimates in the spirit of \cite{ervguglachp3} followed by a  duality argument.
		
		Next in Section \ref{obssys} we combine the Carleman estimates obtained in Section \ref{Sec-CarlemanS}. First using suitably large values of the Carleman parameters we are able to prove an inequality corresponding to the unique continuation property for the system satisfied by $(\sigma,q,\psi).$ This inequality is explicitly given by \eqref{obs-sigma-q-psi-after-carls}. It is not surprising to obtain an estimate of the form \eqref{obs-sigma-q-psi-after-carls} by combining three different observability estimates. In fact as it is well known that the Carleman parameters quantify the compactness of a system hence our strategy to obtain the inequality \eqref{obs-sigma-q-psi-after-carls} strongly relates on the proof of Step 1 in Theorem \ref{lemincreg}, where we prove the well posedness of \eqref{adjsysq} by gaining a time integrability of a suitable fixed point map. After this unique continuation estimate is used to show an observability estimate of $(\sigma,q,\psi)$ at some intermediate time, see \eqref{penobst-2T0} for details. We further use a well posedness result for the system satisfied by $(\sigma,q,\psi)$ to obtain an observability estimate over $(\sigma,q,\psi)$ at initial time $t=0.$ Finally using this observability estimate over $(\sigma,q,\psi)(\cdot,0)$ we recover an observability estimate of $(\sigma,\div\,v,\psi)(\cdot,0),$ which combined with an observability inequality of $\curl\,v(\cdot,0)$ furnishes the desired inequality \eqref{obsinq}. 
		\subsection{Related bibliography}
		Concerning the incompressible Navier-Stokes equations in a $2D$ domain one can find a result proving the local exact controllability to trajectories with localized boundary control in \cite{fursikov}. It is assumed in \cite{fursikov} that the fluid satisfies no vorticity boundary condition in the complement of the control part of the boundary. Local exact controllability to trajectories for incompressible Navier-Stokes equations in a $3D$ domain with distributed control and homogeneous Dirichlet boundary condition can be found in \cite{immanuvilov}. With less regularity assumption on the target trajectory the result in \cite{immanuvilov} was improved in \cite{guefarcara}. We would also like to mention the article \cite{guenolinnvsp} for the local exact distributed controllability to trajectories for incompressible Navier-Stokes equations in a $3D$ domain with non linear Navier-slip boundary condition. In all of these articles the fluid is assumed to be homogeneous $i.e$ the fluid density is constant. In a very recent article \cite{ervbad} the authors prove the local exact boundary controllability to smooth trajectories for a non homogeneous incompressible Navier-Stokes equation in a three dimensional domain. For global controllability results for incompressible Navier-Stokes equations we refer the readers to \cite{coronglobal}, \cite{chapouly} and the references therein.\\
		Now we quote a few articles dealing with the controllability issues of fluid structure interaction models. In fact to the best of our knowledge the only known results concerning the controllability issues of a fluid structure interaction problem in dimension greater than one deals with the motion of a rigid body inside a incompressible fluid modeled by Navier-Stokes equations where the structural motion are given by the balance of linear and angular momentum. Local null controllability of such an interaction problem in dimension two can be found in \cite{bouosses} and \cite{Imtaka}. In dimension three a local null controllability for such a system is proved in \cite{bouguejem}. The article \cite{raymondbeamchp3} deals with the problem of feedback stabilization (in infinite time) for an incompressible fluid structure interaction problem in a $2D$ channel where the structure appears at the fluid boundary and is modeled by an Euler-Bernoulli damped beam, the one we consider in \eqref{1.16llnr}$_{8}$-\eqref{1.16llnr}$_{9}.$ To our knowledge so far there does not exist any article dealing with the finite time controllability of a fluid structure interaction problem (neither for incompressible nor compressible fluids) in dimension more than one where the structure appears at the fluid boundary.
		We would also like to refer the readers to \cite{plateeqzuazua1}, \cite{platezhang} and \cite{plateeqzuzua2} for observability estimates individually for the Euler-Bernoulli plate equations and Kirchoff plate systems without damping.
		
		We also like to quote a few articles from the literature dealing with the controllability issues of compressible Navier-Stokes equation. In fact our strategy to handle the coupling of the fluid velocity and density in the system \eqref{1.16llnr} amounts in introducing a new unknown namely the effective viscous flux and this strategy is inspired from the article \cite{ervguglachp3}. The article \cite{eggpchp3} concerns the motion of a fluid in dimension one whereas \cite{ervguglachp3} and \cite{molina} deal with fluid flows in dimension two and three. For the controllability issues of one dimensional compressible Navier Stokes equations we also refer the readers to \cite{rammy} and \cite{raymcho}.\\
		Concerning the unique continuation property of incompressible Stokes and Oseen equations we refer the readers to \cite{fabrelebeau} and \cite{Triggianiucont}. For the use of Carleman estimates in proving controllability and unique continuation properties of parabolic and elliptic PDE's one can also consult \cite{rousseaulebeau}. The article \cite{ervguglachp3} dealing with the local exact controllability of the compressible Navier-Stokes equations set in a three dimensional torus of course proves observability results for the adjoint of the corresponding linearized system which is stronger than unique continuation. But to the best of our knowledge there is no unique continuation result available even for linear compressible Navier-Stokes equations (without the structure) in dimension more than one where the domain is not entirely a torus. As a particular situation Corollary \ref{uniquecontinuation} implies the unique continuation property for the adjoint of linear compressible Navier-Stokes equations in a 2D channel $\mathbb{T}_{L}\times(0,1)$ where the fluid velocity solves Navier-slip boundary conditions without friction at the lateral boundaries of the channel.
		 \subsection{Outline} Section \ref{Sec-Well-posed} is devoted for the proof of the well posedness result Theorem \ref{lemincreg}. In Section \ref{Sec-CarlemanS} we prove several observability estimates. More precisely in Section \ref{Carldbeam} we state a Carleman estimate for the adjoint of a damped beam equation, in Section \ref{Carlheat} we prove Carleman estimate for the adjoint of a heat equation and finally in Section \ref{sectrans} we prove some observability estimates for a transport equation. In Section \ref{obssys}, we suitably combine the Carleman estimates obtained in Section \ref{Sec-CarlemanS} to furnish the proof of the central result Theorem \ref{main3}. We also include the proof of Corollary \ref{uniquecontinuation} in Section \ref{proofucont}. The final Section \ref{appendix} contains the proof of Lemma \ref{lemconstp2}, which is a intermediate step in proving Theorem \ref{main3}, and a result on parabolic regularization, Lemma \ref{parabolicregularization} which in turn is used during the proof of Lemma \ref{Lem-Obs-sigma-q}.

		\section{Well posedness result for the adjoint problem \eqref{adjsys}}
		\label{Sec-Well-posed}
		This section is devoted for the proof of Theorem \ref{lemincreg}. As it turns out, this proof will also give some insights of the strength of the various coupling between the equations in \eqref{adjsys}, which will also help in the proof of Theorem \ref{main3}.
		%
		
		%
		
		\subsection{Proof of Theorem \ref{lemincreg}}\label{lemin}
		The proof is divided into two main steps. The first one consists in introducing the new unknown 
		\begin{equation}\label{efvisfl}
		\begin{array}{l}
		q=(\mu'+2\mu)\mbox{div}\,v+\overline{\rho}\sigma, 
		\end{array}
		\end{equation}
		and looking at the system satisfied by $(\sigma, q, \beta)$, which turns out to be easier to analyze than the full system \eqref{adjsys}. The second step will then consist in deducing from the regularity on $(\sigma, q, \beta)$ suitable estimates for the function $v$ in \eqref{adjsys}.
		
		\begin{remark}
			\label{rem-q-good-unknow}
			The unknown $q$ in \eqref{efvisfl} can be interpreted as the dual version of the effective viscous flux introduced for instance in \cite{lionschp3}, see also \cite{feireislchp3}. In fact, this quantity already appeared in \cite{ervguglachp3} when studying the controllability properties of a compressible fluid (without structure and controls acting on the whole boundary), where it helps to weaken the coupling of the parabolic and hyperbolic effects of the system.
			\\
			Here, the interesting point is that this quantity is also suitable to deal with the coupling with the structure lying on the boundary. There, we strongly use that the force acting on the beam is given by \eqref{1.3chp3} instead of the more natural one \eqref{stressor}, which would not yield such a clean understanding of the coupling between the fluid and the structure.
		\end{remark}
		
		\subsubsection{Step 1. The system satisfied by $(\sigma, q, \beta)$}
		
		We first derive the system that $(\sigma, q, \beta)$ should satisfy provided $(\sigma, v, \beta)$ satisfies \eqref{adjsys} and has the regularity given by \eqref{imregv}. Indeed, these regularities allow to take the trace of $\nabla^{2}v$ and $\nabla\sigma$ $a.e$ on $(\mathbb{T}_{L}\times\{0,1\})\times(0,T).$ Hence we can consider the trace of the equation \eqref{adjsys}$_{2}$ and use \eqref{adjsys}$_{3}$-\eqref{adjsys}$_{5}$ to have the following $a.e$ on $((\mathbb{T}_{L}\times\{1\}))\times(0,T):$
		\begin{equation}\label{nmnq}
		\begin{array}{ll}
		0 & = -\overline{\rho}(\partial_{t}v_{2}+\overline{u}_{1}\partial_{x}v_{2})-\mu(\partial_{xx}v_{2}+\partial_{zz}v_{2})-(\mu+\mu')(\partial_{xz}v_{1}+\partial_{zz}v_{2})-\overline{\rho}\partial_{z}\sigma
		\\
		& =  -\overline{\rho}(\partial_{t}\psi+\overline{u}_{1}\partial_{x}\psi)-(\mu'+2\mu)(\partial_{xz}v_{1}+\partial_{zz}v_{2})-\overline{\rho}\partial_{z}\sigma\,\,\\
		&\qquad\qquad\qquad\qquad\qquad\qquad\qquad\qquad\qquad(\mbox{using}\,\eqref{adjsys}_{3},\eqref{adjsys}_{5})\\
		&=-\overline{\rho}(\partial_{t}\psi+\overline{u}_{1}\partial_{x}\psi) -  \partial_{z}q.
		\end{array}
		\end{equation}
		Similarly one can obtain that on the boundary $(\mathbb{T}_{L}\times\{0\})\times(0,T),$ $q$ satisfies
		$$\partial_{z}q=0\quad a.e\,\,\mbox{on}\quad (\mathbb{T}_{L}\times\{0\})\times(0,T).$$
		Hence with the formal calculations above and using \eqref{adjsys} we obtain the following system satisfied by the unknowns $(\sigma,q,\psi):$
		\begin{equation}\label{adjsysq}
		\left\{ \begin{array}{ll}
		\displaystyle
		-\partial_{t}\sigma-\overline{u}_{1}\partial_{x}\sigma+\frac{P'(\overline{\rho})\overline{\rho}}{\nu}\sigma=\frac{P'(\overline{\rho})}{\nu}q\,& \mbox{in}\, (\mathbb{T}_{L}\times(0,1))\times(0,T),
		\vspace{1.mm}\\
		\displaystyle
		-(\partial_{t}q+\overline{u}_{1}\partial_{x}q)-\frac{\nu}{\overline{\rho}}\Delta q-\frac{P'(\overline{\rho})\overline{\rho}}{\nu}q
		=-\frac{P'(\overline{\rho})\overline{\rho}^{2}}{\nu}\sigma\, &\mbox{in} \, (\mathbb{T}_{L}\times(0,1))\times(0,T),
		\vspace{1.mm}\\
		\partial_{z}q=-\overline{\rho}(\partial_{t}\psi+\overline{u}_{1}\partial_{x}\psi)\,& \mbox{on}\,(\mathbb{T}_{L}\times\{1\})\times(0,T),\\[1.mm]
		\partial_{z}q=0\,& \mbox{on}\, (\mathbb{T}_{L}\times\{0\})\times(0,T),
		\vspace{1.mm}\\
		q(\cdot,T)=q_{T}\,& \mbox{in} \, \mathbb{T}_{L}\times(0,1),
		\vspace{1.mm}\\
		{\sigma}(\cdot,T)=\sigma_{T}\,& \mbox{in}\, \mathbb{T}_{L}\times(0,1),
		\vspace{1.mm}\\
		\partial_{tt}\psi+\partial_{txx}\psi+\partial_{xxxx}\psi=(\partial_{t}+\overline{u}_{1}\partial_{x})q\,& \mbox{on}\, (\mathbb{T}_{L}\times\{1\})\times(0,T),
		\vspace{1.mm}\\
		\psi(T)=\psi_{T}\quad \mbox{and}\quad\partial_{t}\psi(T)=\psi_{T}^{1}\,&\mbox{in}\, \mathbb{T}_{L}\times\{1\},
		\end{array}\right.
		\end{equation}
		where $\nu=(\mu'+2\mu)$, and $q_T  =\nu\mbox{div}v_{T}+\overline{\rho}\sigma_{T}$.
		The well-posedness result for the system  \eqref{adjsysq} is stated in form of the following lemma:
	    	\begin{lem}\label{lemconstp2}
			There exists a constant $C >0$ such that for any 
			\begin{equation}\label{indatadjq}
			(\sigma_{T},q_{T},\psi_{T},\psi^{1}_{T})\in H^{1}(\mathbb{T}_{L}\times(0,1))\times H^{2}(\mathbb{T}_{L}\times(0,1))\times H^{3}(\mathbb{T}_{L}\times\{1\})\times H^{1}(\mathbb{T}_{L}\times\{1\})
			\end{equation}
			satisfying the following compatibility conditions
			\begin{equation}\label{compatonqT}
			\begin{split}
			&(i)\,\partial_{z}q_{T}=-\overline{\rho}(\psi^{1}_{T}+\overline{u}_{1}\partial_{x}\psi_{T})\,\,\mbox{on}\,\,\mathbb{T}_{L}\times\{1\},\\
			&(ii)\,\partial_{z}q_{T}=0\,\,\mbox{on}\,\,\mathbb{T}_{L}\times\{0\},
			\end{split}
			\end{equation}
			the system \eqref{adjsysq} admits a unique solution $(\sigma,q,\psi)$ which satisfies
			 \begin{equation}\label{regsgqps}
			 \left\{ \begin{array}{ll}
			 &\sigma\in C^{0}([0,T];H^{1}(\mathbb{T}_{L}\times(0,1)))\cap C^{1}([0,T];L^{2}(\mathbb{T}_{L}\times(0,1))),\\
			 & q\in L^{2}(0,T;H^{3}(\mathbb{T}_{L}\times(0,1)))\cap H^{3/2}(0,T;L^{2}(\mathbb{T}_{L}\times(0,1))),\\
			 & \psi\in L^{2}(0,T;H^{4}(\mathbb{T}_{L}\times\{1\}))\cap H^{1}(0,T;H^{2}(\mathbb{T}_{L}\times\{1\}))\\
			 &\qquad\quad\cap H^{2}(0,T;L^{2}(\mathbb{T}_{L}\times\{1\})).
			 \end{array}\right.
			 \end{equation}
			 and the following estimate
			\begin{equation}\label{corstep2}
			\begin{split}
			&\|\sigma\|_{C^{0}([0,T];H^{1}(\mathbb{T}_{L}\times(0,1)))\cap C^{1}([0,T];L^{2}(\mathbb{T}_{L}\times(0,1)))}+\|q\|_{C^{0}([0,T];H^{2}(\mathbb{T}_{L}\times(0,1)))}\\
			& +\|\psi\|_{C^{0}([0,T];H^{3}(\mathbb{T}_{L}\times\{1\}))\cap C^{1}([0,T];H^{1}(\mathbb{T}_{L}\times\{1\}))}\\
			& \leqslant C(\|q_{T}\|_{H^{2}(\mathbb{T}_{L}\times(0,1))}+\|\sigma_{T}\|_{H^{1}(\mathbb{T}_{L}\times(0,1))}+\|(\psi_{T},\psi^{1}_{T})\|_{H^{3}(\mathbb{T}_{L}\times\{1\})\times H^{1}(\mathbb{T}_{L}\times\{1\})}).
			\end{split}
			\end{equation}
		\end{lem}
		The proof of Lemma \ref{lemconstp2} is included in Section \ref{appendix}.\\
		Next using Lemma \ref{lemconstp2} we will obtain the regularities for $v$ and further prove that the triplet $(\sigma,v,\psi)$ solves the system \ref{adjsys}.

		\subsubsection{Step 2: Constructing $v$.}
		In order to complete the proof of the existence of a solution $(\sigma, v, \psi)$ of \eqref{adjsys}, we first set $q_T =\nu\mbox{div}v_{T}+\overline{\rho}\sigma_{T}$ and solve the system \eqref{adjsysq} with initial data $(\sigma_T, q_T, \psi_T, \psi_T^1)$. Note that the regularity and compatibility conditions of the initial data $(\sigma_T, v_T, \psi_T, \psi_T^1)$ in Theorem \ref{lemincreg} precisely imply the regularity and compatibility conditions of $(\sigma_T, q_T, \psi_T, \psi_T^1)$ required by Lemma \ref{lemconstp2}, so that Lemma \ref{lemconstp2} applies, yielding functions $(\sigma, q, \psi)$ solving \eqref{adjsysq}. Besides, if $\sigma_T \in H^2(\mathbb{T}_{L}\times(0,1))$, it is clear that the solution $\sigma$ of \eqref{adjsysq}$_{(1)}$ in fact satisfies 
		$$
		\sigma\in C^{0}([0,T];H^{2}(\mathbb{T}_{L}\times(0,1)))\cap C^{1}([0,T];H^{1}(\mathbb{T}_{L}\times(0,1))),
		$$
		since the source term of the transport equation belongs to $C^{0}([0,T];H^{1}(\mathbb{T}_{L}\times(0,1)))$
		(by interpolation from \eqref{regsgqps}$_{2}$). For a detailed proof one can imitate the arguments used in proving \cite[Theorem 2.4]{smitraexistence}.
		\\
		We then construct the function $v$ by solving the equation
		\begin{equation}\label{sysv}
		\left\{ \begin{array}{lll}
		&-\overline{\rho}(\partial_{t}v+\overline{u}_{1}\partial_{x}v)-\mu\Delta v-(\mu'+\mu)\nabla(\mathrm{div}v)
		-\overline{\rho}\nabla\sigma=0\, &\mbox{in} \, (\mathbb{T}_{L}\times(0,1))\times(0,T),
		\vspace{1.mm}\\
		& v\cdot n={v}_{2}=\psi\,& \mbox{on}\,(\mathbb{T}_{L}\times\{1\})\times(0,T),\\
		&{v}(\cdot,t)\cdot n=0\,& \mbox{on}\, (\mathbb{T}_{L}\times\{0\})\times(0,T),
		\vspace{1.mm}\\
		&\displaystyle\mathrm{curl}v=0\,&\mbox{on}\,(\mathbb{T}_{L}\times\{0,1\})\times(0,T),\\ 
		& v(\cdot,T)=v_{T}\,& \mbox{in} \, \mathbb{T}_{L}\times(0,1).
		\end{array}\right.
		\end{equation}
		Note that we already know the regularity of the source terms $\overline\rho \nabla \sigma$ and $\psi$ from \eqref{regsgqps}, so this step only consists in constructing a solution to a parabolic equation with source terms having given regularities.\\
    	One can prove the following regularities of $v:$
		\begin{equation}\label{proregv}
		\begin{split}
		v\in L^{2}(0,T;H^{3}(\mathbb{T}_{L}\times(0,1)))\cap H^{1}(0,T;H^{2}(\mathbb{T}_{L}\times(0,1)))
		\cap H^{2}(0,T;L^{2}(\mathbb{T}_{L}\times(0,1)))
		\end{split}
		\end{equation}
		solving the system \eqref{sysv}. The proof of such a parabolic regularity result for the system 
		\eqref{sysv} with no vorticity boundary condition is classical in the literature. For the details of the proof we refer the readers to \cite[p. 168, Section 4.2.1.2]{phdthesis}.\\
		With the strong regularity framework \eqref{proregv} and with $\sigma$ in the space \eqref{regsgqps} one can verify the formal argument \eqref{nmnq} and conclude that the triplet $(\sigma,v,\psi)$ solves the system \eqref{adjsys} in the functional framework \eqref{imregv}.
		\\
		Now in order to prove the uniqueness of the solution $(\sigma,v,\psi)$ of \eqref{adjsys} in the spaces \eqref{imregv}, we proceed as follows. Let us assume that there exists two solutions $(\sigma_{1},v_{1},\psi_{1})$ and $(\sigma_{2},v_{2},\psi_{2})$ to the problem \eqref{adjsys} in the functional spaces \eqref{imregv} with the same initial datum. The strong regularities of $v_{1}$ (as well as $v_{2}$) and $\sigma_{1}$ (as well as $\sigma_{2}$) allows us to verify that $(\sigma_{1},q_{1},\psi_{1})$ and $(\sigma_{2},q_{2},\psi_{2})$ solves \eqref{adjsysq}, where
		$$q_{1}=(\nu\mbox{div}\,v_{1}+\overline{\rho}\sigma_{1})\,\,\mbox{and}\,\, q_{2}=(\nu\mbox{div}\,v_{2}+\overline{\rho}\sigma_{2}).$$
		But the solution of the system \eqref{adjsysq} is unique (thanks to the Banach fixed point argument used in $Step\,2$) in the framework \eqref{regsgqps}. Hence $\sigma_{1}=\sigma_{2}$ and $\psi_{1}=\psi_{2}.$ Now one observes that $v_{1}$ (respectively $v_{2}$) solves \eqref{sysv} with $(\sigma_{1},\psi_{1})$ (respectively $(\sigma_{2},\psi_{2})$). Since $(\sigma_{1},\psi_{1})=(\sigma_{2},\psi_{2}),$ from the uniqueness of the solution to the linear problem \eqref{sysv} we infer that $v_{1}=v_{2}.$ Hence the solution to the problem \eqref{adjsys} is unique in the functional framework \eqref{imregv}.\\
		This concludes the proof of Theorem \ref{lemincreg}.
	 
		\begin{remark}\label{extensionremfrcbm}
			At this point the role of the dual version of effective viscous flux $q$ is clear in order to prove the existence result Theorem \ref{lemincreg}. We now explain more clearly the reason behind considering the simplified model \eqref{1.1chp3}-\eqref{interface}-\eqref{u.n0}-\eqref{curl0}-\eqref{1.2chp3}-\eqref{1.3chp3} instead of using the more physical expression \eqref{stressor} (recall Remark \ref{Remark-Force-Beam}). In our analysis we use the dual version of effective viscous flux to reduce the strength of the parabolic hyperbolic coupling in the system \eqref{adjsys}. With the simplified expression \eqref{1.3chp3} of the net surface force on the beam we are able to write the adjoint of the linearization of the system \eqref{1.1chp3}-\eqref{interface}-\eqref{u.n0}-\eqref{curl0}-\eqref{1.2chp3}-\eqref{1.3chp3} as a closed loop system in terms of the unknowns: dual of fluid density, beam displacement and $q.$ On the other hand it does not seem to be possible if one considers the expression \eqref{stressor} for the net force acting on the beam. In some sense the effective viscous flux does not prove to be a good unknown to weaken the parabolic hyperbolic coupling near the boundary for the adjoint to the system  \eqref{1.1chp3}-\eqref{interface}-\eqref{u.n0}-\eqref{curl0}-\eqref{1.2chp3}-\eqref{stressor}.  
		\end{remark}
		\section{Carleman estimates for scalar equations}\label{Sec-CarlemanS}
		From now onwards we fix our final time horizon $T$ and the length $L$ of our torus such that they satisfy \eqref{Tgrtr} and \eqref{fixL} respectively.\\
		The goal of this section is to provide Carleman estimates for the various equations involved in the system \eqref{adjsys}, in particular:
		\begin{itemize}
			\item a Carleman estimate for the beam equation set in the torus; 
			\item a Carleman estimate for the heat equation with non-homogeneous Neumann boundary conditions; 
			\item a Carleman estimate for the transport equation. 
		\end{itemize}
		One of the difficulties of our work is that, in order to prove Theorem \ref{main3}, one should be able to couple all these Carleman estimates in a suitable way. In order to do that, we will consider one weight function which allows to derive Carleman estimates for the beam equation, the heat equation and the transport equation simultaneously.
		\subsection{Construction of the weight function}\label{Conw8fn}
		$1.$ We first introduce a function $\eta$ on $\mathbb{T}_{L}$ such that 
		\begin{equation}\label{eta*}
		\begin{array}{l}
		\eta\in C^{6}({\mathbb{T}}_{L}),\,\, \eta(x)>0\,\,\mbox{in}\,\, \mathbb{T}_{L},
		\\
		\mbox{inf} \left\{|\nabla\eta(x)|
		\suchthat x\in\mathbb{T}_{L}
		\setminus\{(-3\overline{u}_{1}T,-2\overline{u}_{1}T)\cup
		(d+\overline{u}_{1}T,d+3\overline{u}_{1}T)
		\}\right\}
		>0,
		\end{array}
		\end{equation}
		
		$2.$ Now we define $\eta^{0}\in C^{6}(\mathbb{T}_{L}\times[0,T])$ as follows
		\begin{equation}\label{defeta0}
		\begin{array}{l}
		\eta^{0}(x,t)=\eta(x-\overline{u}_{1}t)\quad\mbox{for all}\quad (x,t)\in \mathbb{T}_{L}\times[0,T],
		\end{array}
		\end{equation}
		$i.e$ $\eta^{0}$ solves
		\begin{equation}\label{transportweight}
		\begin{array}{l}
		(\partial_{t}+\overline{u}_{1}\partial_{x})\eta^{0}=0\quad\mbox{in}\quad (\mathbb{T}_{L}\times(0,1))\times(0,T).
		\end{array}
		\end{equation}
		In view of \eqref{eta*}, one can easily verify the following
		\begin{equation}\label{eta0grtr0}
		\begin{array}{l}
		\mbox{inf}\left\{|\nabla\eta^{0}(x,t)|\suchthat (x,t)\in [-\overline{u}_1 T, d + \overline u_1 T]\times[0,T]\right\}>0,
		\end{array}
		\end{equation}
		%
		$3.$ Next we will define a weight function in the time variable. Let  $T_{0}>0,$ $T_{1}>0,$  small enough such that
		\begin{equation}\label{relT01ep}
		\begin{split}
		2 T_{0}+2T_{1}<T-\frac{d}{\overline{u_{1}}}.
		\end{split}
		\end{equation}
		Now we choose a weight function $\theta(t) \in C^4(0,T)$ such that
		\begin{equation}\label{theta}
		\theta(t)=\left\{ \begin{array}{lll}
		\displaystyle
		&\displaystyle \frac{1}{t^{2}},\,\,&\forall\,\, t\in[0,T_{0}],\smallskip\\
		&\theta\, \mbox{is strictly decreasing}\,\,&\forall\,\, t\in[T_{0}, 2T_0],\\
		& 1\,\,&\forall\,\, t\in[2 T_{0},T-2T_{1}],\\
		& \theta\, \mbox{is strictly increasing}\,\,&\forall\,\, t\in[T-2T_{1},T-T_{1}],\\
		& \displaystyle \frac{1}{(T-t)^{2}},\,\,&\forall\,\, t\in[T-T_{1},T].
		\end{array}\right.
		\end{equation}
		Observe that $\theta(t)$ blows up at the terminal points $\{0\}$ and $\{T\}$ of the interval $(0,T).$\\
		$4.$ In view of $\eta^{0}(x,t)$ and $\theta(t)$ we finally introduce the following weight functions in $\mathbb{T}_{L}\times[0,T],$
		\begin{equation}\label{w8fn}
		\left\{ \begin{array}{l}
		\displaystyle\phi(x,t)=\theta(t)(e^{6\lambda\|\eta^{0}\|_{\infty}}-e^{\lambda(\eta^{0}(x,t)+4\|\eta^{0}\|_{\infty})}),\\
		{\xi}(x,t)=\theta(t)e^{\lambda(\eta^{0}(x,t)+4\|\eta^{0}\|_{\infty})},
		\end{array}\right.
		\end{equation}
		where $\lambda \geq 1$ is a positive parameter. 
		\begin{remark}\label{roleLst}
			Recall that we have fixed $L=3\overline{u}_{1}T.$ The reason lies in the choice \eqref{defeta0} of $\eta^{0}$ which travels along $\mathbb{T}_{L}$ with a velocity $\overline{u}_{1}.$ Of course this choice plays a very important role in Section \ref{sectrans} while obtaining an observability estimate for a hyperbolic transport equation. In that case the domain $(0,d)\times(0,1)$ needs to be embedded in $\mathbb{T}_{L}\times(0,1),$ for $L$ large enough such that $\inf\{|\nabla\eta^{0}(x,t)|\}$ is positive in a neighborhood of $(0,d)\times(0,1),$ since this is crucial to obtain parabolic Carleman estimates. Now the choice $L=3\overline{u}_{1}T$ serves this purpose and provides enough room so that \eqref{eta0grtr0} holds.
		\end{remark}
		From now on we will denote by $c,$ a generic strictly positive small constant and by $C,$ a large constant, where both of them are independent of the parameters $s$ ($\geq 1$) and $\lambda$ ($\geq 1$). 
		\\
		In our computations afterwards we will frequently use the following estimates, valid on $\mathbb{T}_L \times (0,T)$:
		\begin{equation}\label{prilies}
		\begin{array}{l}
		\ds |\partial^{(i)}_{x}\phi|\leqslant C\lambda^{i}{\xi} \quad \mbox{for all}\,i\in\{1,2,3,4\},
		\\
		\ds |\partial_{t}\phi|\leqslant C\lambda{\xi}^{3/2},\quad 
		|\partial_{tt}\phi|\leqslant C\lambda^{2}{\xi}^{2},\quad
		|\partial_{tx}\phi|\leqslant C\lambda^{2}{\xi}^{3/2},\quad 
		|\partial_{txx}\phi|\leqslant C\lambda^{3}{\xi}^{3/2},
		\\
		\ds	|\partial_{txxx}\phi|\leqslant C\lambda^{4}{\xi}^{3/2},\quad
		|\partial_{ttx}\phi|\leqslant C\lambda^{3}{\xi}^{2}\quad \mbox{and}\quad|\partial_{ttxx}\phi|\leqslant C\lambda^{4}{\xi}^{2},
		\end{array}
		\end{equation}
		and 
		\begin{equation}\label{prilies*}
		\begin{array}{l}
		\ds |\partial^{(i)}_{x}{\xi}|\leqslant C\lambda^{i}{\xi}\quad\mbox{for all}\,\, i\in\{1,2,3,4\},\\
		\ds 	|\partial_{t}{\xi}|\leqslant C\lambda{\xi}^{3/2},\quad
		|\partial_{tt}{\xi}|\leqslant C\lambda^{2}{\xi}^{2},\quad
		|\partial_{tx}{\xi}|\leqslant C\lambda^{2}{\xi}^{3/2},\quad
		|\partial_{txx}{\xi}|\leqslant C\lambda^{3}{\xi}^{3/2}
		\\
		\ds 	|\partial_{txxx}{\xi}|\leqslant C\lambda^{4}{\xi}^{3/2},\quad
		|\partial_{ttx}\xi|\leqslant C\lambda^{3}{\xi}^{2}\quad
		\mbox{and}\quad |\partial_{ttxx}\xi|\leqslant C\lambda^{4}{\xi}^{2},
		\end{array}
		\end{equation}
		and, for $\lambda$ large enough, for all $(x,t) \in [- \overline u_1 T, d + \overline u_1 T] \times (0,T)$ and $i \in \{2, 4\}$, 
		\begin{equation}
		\label{Positivity-Weights}
		-\partial^{(i)}_{x}\phi  = \partial^{(i)}_{x}\xi \geqslant c\lambda^{i}{\xi}.
		\end{equation}
			For convenience, we further introduce the following shorthand notations which will be used from now onwards mainly for writing the domain of integrals.
			\begin{equation}
			\label{inshrthnd}
			\begin{array}{l}
			{Q^{ex}_{T}}=(\mathbb{T}_{L}\times(0,1))\times(0,T),\quad \omega_{T}=\omega\times(0,T), 
			\smallskip\\
			\mathbb{T}^{1}_{T} = (\mathbb{T}_{L}\times\{1\})\times(0,T), \quad 
			\mathbb{T}^{0}_{T} = (\mathbb{T}_{L}\times\{0\})\times(0,T)
			,\smallskip\\
			\omega^{T}_{1}=\omega_{1}\times(0,T),
			\end{array}
			\end{equation}
			where $\omega$ and $\omega_{1}$ was introduced in \eqref{conznf}.
		\subsection{Carleman estimate for an adjoint damped beam equation}\label{Carldbeam}
		In the following section we state a Carleman estimate for the adjoint of the damped  beam equation: \begin{equation}\label{adjbeam1}
		\left\{ \begin{array}{ll}
		{\partial_{tt}{\psi}}+{\partial_{txx}{\psi}}+{\partial_{xxxx}{\psi}}=f_\psi \,\,& \mbox{in}\,\, \mathbb{T}^{1}_{T},
		\\
		{\psi}(.,T)={\psi}_{T},\quad{\partial_{t}{\psi}}(.,T)={\psi}_{T}^{1}\,\,&\mbox{in}\,\, \mathbb{T}_{L}\times\{1\}.
		\end{array}\right.
		\end{equation}
		The main theorem of this section is stated as follows
		\begin{thm}\label{Carlbeamthm}
			There exist constants $C>0,$ $s_{0}\geqslant 1,$ $\lambda_{0}\geqslant 1$ such that for all $\psi$ solving \eqref{adjbeam1} with initial datum ${\psi}_{T}\in H^{3}(\mathbb{T}_{L})$ and ${\psi}_{T}^{1}\in H^{1}(\mathbb{T}_{L})$ and source term $f_\psi \in L^{2}(\mathbb{T}^{1}_{T})$, for all $s\geqslant s_{0},$ and $\lambda\geqslant\lambda_{0},$
			\begin{align}\label{crlestbmthm}
			& \displaystyle s^{7}\lambda^{8}\iint_{\mathbb{T}^{1}_{T}} {\xi}^{7}|{\psi}|^{2}e^{-2s\phi}+s^{5}\lambda^{6}\iint_{\mathbb{T}^{1}_{T}} {\xi}^{5}|{\partial_{x}{\psi}}|^{2}e^{-2s\phi}\notag
			\\
			&\displaystyle+s^{3}\lambda^{4}\iint_{\mathbb{T}^{1}_{T}} {\xi}^{3} ( |{\partial_{xx}{\psi}}|^{2}+ |{\partial_{t}{\psi}}|^{2}) e^{-2s\phi}+s\lambda^{2}\iint_{\mathbb{T}^{1}_{T}} {\xi}( |{\partial_{tx}{\psi}}|^{2}+|{\partial_{xxx}{\psi}}|^{2})e^{-2s\phi}\notag \\
			&+
			\displaystyle\frac{1}{s}\iint_{\mathbb{T}^{1}_{T}} \frac{1}{{\xi}}(|{\partial_{tt}{\psi}}|^{2}+|\partial_{txx}\psi|^{2}+|{\partial_{xxxx}{\psi}}|^{2})e^{-2s\phi}
			\\
			&
			\displaystyle\leqslant  
			C\iint_{\mathbb{T}^{1}_{T}} |f_\psi|^{2}e^{-2s\phi}
			+Cs^{7}\lambda^{8}\iint_{\omega^{T}_{1}} {\xi}^{7}|{\psi}|^{2}e^{-2s\phi},\notag
			\end{align}
			where the notations $\mathbb{T}^{1}_{T}$ and $\omega^{T}_{1}$ were introduced in \eqref{inshrthnd}.
		\end{thm}
		We will not go into the details of the proof of Theorem \ref{Carlbeamthm} but only comment on it with suitable reference.
		\begin{proof}[Comment on the proof of Theorem \ref{Carlbeamthm}]
			The Carleman estimate stated in Theorem \ref{Carlbeamthm} is already proved in the article \cite{Carlemanbeam} but with a difference in the weight functions $\eta^{0}(x,t)$ and $\theta(t).$ To be precise unlike \cite{Carlemanbeam}, in the present article we have used a weight function $\eta^{0}(x,t)$ which depends on the time variable and further the definition \eqref{theta} of $\theta(t)$ also varies from the one used in \cite[p. 2, (1.6)]{Carlemanbeam}. This only leads to minor modifications in the estimates \eqref{prilies} and \eqref{prilies*}. The proof of Theorem \ref{Carlbeamthm} can be carried out essentially in a similar way as the proof of Carleman estimate \cite[Theorem 1.3]{Carlemanbeam}. For a detailed proof of Theorem \ref{Carlbeamthm} with the weight functions defined in Section \ref{Conw8fn} we refer the readers to \cite[p. 182, Section 4.3.2]{phdthesis}.
			\end{proof}
		\subsection{Carleman estimate for an adjoint heat equation}\label{Carlheat}
		In this section we consider the following adjoint heat equation:
		\begin{equation}\label{adjheateq}
		\left\{ \begin{array}{ll}
		\displaystyle-\frac{\overline{\rho}}{\nu}\partial_{t}q-\Delta q=f_{1}\quad&\mbox{in}\quad Q^{ex}_{T},\\
		\displaystyle\partial_{z}q=f_{2}\quad&\mbox{on}\quad \mathbb{T}^{1}_{T},\\
		\displaystyle\partial_{z}q=0\quad&\mbox{on}\quad\mathbb{T}^{0}_{T},\\
		\displaystyle q(.,T)=q_{T}\quad & \mbox{in}\quad \mathbb{T}_{L}\times(0,1),
		\end{array}\right.
		\end{equation}
		where $f_{1}\in L^{2}(Q^{ex}_{T}),$ $f_{2}\in L^{2}(\mathbb{T}^{1}_{T}).$
		\\
		\begin{thm}\label{carlht}
			There exist positive constants ${C},$ ${s}_{1}\geq 1$ and ${\lambda}_{1}\geq 1$  such that for all 
			\begin{equation}\label{assmf12}
			\begin{split}
			f_{1}\in L^{2}({Q^{ex}_{T}}),\quad\mbox{and}\quad f_{2}\in L^{2}(\mathbb{T}^{1}_{T}),
			\end{split}
			\end{equation}
			for all $q_{T}\in L^{2}(\mathbb{T}_{L}\times(0,1)),$
			for all $s\geqslant{s}_{1}$ and $\lambda\geqslant{\lambda}_{1},$ the weak solution $q$ of \eqref{adjheateq} satisfies the following inequality
			\begin{equation}\label{obsnht}
			\begin{split}
			\displaystyle
			&\iint\limits_{{{Q^{ex}_{T}}}}e^{-2s\phi}(s\lambda^{2}{\xi}|\nabla q|^{2}+s^{3}\lambda^{4}{\xi}^{3}|q|^{2})
			+s^{2}\lambda^{3}\iint\limits_{{\mathbb{T}^{1}_{T}}}e^{-2s\phi}{\xi}^{2}|q|^{2}
			\\
			&\leqslant C\left( \iint\limits_{{{Q^{ex}_{T}}}}e^{-2s\phi}|f_{1}|^{2}+s\lambda\iint\limits_{{\mathbb{T}^{1}_{T}}}e^{-2s\phi}{\xi}|f_{2}|^{2}
			\displaystyle
			+s^{3}\lambda^{4}\iint\limits_{{\omega_{T}}}e^{-2s\phi}{\xi}^{3}|q|^{2}\right),
			\end{split}
			\end{equation}
			where the notations $Q^{ex}_{T},$ $\mathbb{T}^{1}_{T}$ and $\omega_{T}$ are introduced in \eqref{inshrthnd}.
		\end{thm}
		Theorem \ref{carlht} will be proved mainly by using the similar line of arguments as used to prove \cite[Theorem 1]{farnan}. The difference with \cite{farnan} occurs in the construction of the weight functions. To be precise, the weight $\theta(t)$ in time is defined in \cite{farnan} as $\theta(t)=\frac{1}{t(T-t)},$ whereas in our case $\theta(t)$ is as defined in \eqref{theta}. Further unlike \cite{farnan}, the function $\eta^{0}$ (defined in \eqref{defeta0}) travels in time with a constant velocity $\overline{u}_{1}.$ Both of these differences can be classically handled just by using the estimates \eqref{prilies}, \eqref{prilies*} and \eqref{Positivity-Weights} of $\xi$ and $\phi$. One can also consult \cite{ervbad} for similar issues. 
		\\
		Above all in \cite{farnan} and in most other articles in the literature it is assumed that $\eta^{0}$ vanishes at the boundary of the domain. Of course this assumption does not serve our purpose since we are working with a beam at the boundary. In our case, $\eta^{0}$ just depends on $(x,t)$ implying in particular
		\begin{equation}\label{norder}
		\begin{array}{l}
		\displaystyle
		\partial_{z}\phi=\partial_{z}\xi=0\quad\mbox{on}\quad \mathbb{T}^{1}_{T} \cup \mathbb{T}^{0}_{T} .
		\end{array}
		\end{equation}
		Using this property, we still recover the same Carleman estimate for the heat equation with non homogeneous boundary condition obtained in \cite{farnan}.\\
		Hence without going into the details we will just sketch the main steps of the proof of Theorem \ref{carlht} and other supporting lemmas, we comment with references for their proofs.
		
		For the proof of Theorem \ref{carlht} we will need an auxiliary result: a Carleman inequality for heat equation with homogeneous Neumann boundary conditions, stated below:
		\begin{lem}\label{Carlman}
			There exist positive constants $C,$ $s_{2}\geq 1$ and $\lambda_{2} \geq 1$ such that for all $s\geqslant s_{2},$ $\lambda\geqslant \lambda_{2},$ for all ${\vartheta}_{T}\in  L^{2}({\mathbb{T}_{L}\times(0,1)}),$ and for all $f_{3}\in L^{2}({{Q^{ex}_{T}}}),$ the solution ${\vartheta}$ of the following problem 
			\begin{equation}\label{hthne}
			\left\{ \begin{array}{ll}
			\displaystyle
			-\frac{\overline{\rho}}{\nu}\partial_{t}{\vartheta}-\Delta {\vartheta}=f_{3}\quad&\mbox{in}\quad {{Q^{ex}_{T}}},\\[2.mm]
			\displaystyle
			\partial_{z}{\vartheta}=0\quad&\mbox{on}\quad{\mathbb{T}^{1}_{T}}\cup {\mathbb{T}^{0}_{T}},\\
			\displaystyle
			{\vartheta}(.,T)={\vartheta}_{T}\quad & \mbox{in}\quad {\mathbb{T}_{L}\times(0,1)},
			\end{array}\right.
			\end{equation}
			satisfies
			\begin{equation}\label{hthne1}
			\begin{array}{l}
			\displaystyle\iint\limits_{{{Q^{ex}_{T}}}}e^{-2s\phi}\left(s\lambda^{2}{\xi}|\nabla{\vartheta}|^{2}+s^{3}\lambda^{4}{\xi}^{3}|{\vartheta}|^{2}\right)\\
			\leqslant \displaystyle C\left(\iint\limits_{{{Q^{ex}_{T}}}}e^{-2s\phi}|f_{3}|^{2}+s^{3}\lambda^{4}\iint\limits_{{\omega_{T}}}e^{-2s\phi}{\xi}^{3}|{\vartheta}|^{2} \right).
			\end{array}
			\end{equation}
		\end{lem}
		Lemma \ref{Carlman} will allow to construct by duality suitable solutions to a control problem:
		\begin{lem}\label{sublemreg}
			There exist positive constants $C>0,$ $s_{3} \geq 1$ and $\lambda_{3} \geq 1,$ such that for all $s\geqslant s_{3},$ $\lambda\geqslant\lambda_{3}$ and for all $G$ satisfying
			\begin{equation}\label{condG}
			\begin{array}{l}
			\displaystyle\iint\limits_{{{Q^{ex}_{T}}}}\xi^{-3}|G|^{2}e^{2s\phi}<\infty,
			\end{array}
			\end{equation}
			there exists a solution $(Y,H)$ of the following control problem 
			\begin{equation}\label{strng}
			\left\{ \begin{array}{ll}
			\displaystyle\frac{\overline{\rho}}{\nu}\partial_{t}Y-\Delta Y=G+H\chi_{\omega}&\quad\mbox{in}\quad {{Q^{ex}_{T}}},\\
			\displaystyle{\partial_{z} Y}=0&\quad\mbox{on}\quad {\mathbb{T}^{1}_{T}}\cup\mathbb{T}^{0}_{T},\\
			Y(\cdot,0)=0&\quad\mbox{in}\quad {\mathbb{T}_{L}\times(0,1)},\\
			Y(\cdot,T)=0&\quad \mbox{in}\quad \mathbb{T}_{L}\times(0,1),
			\end{array}\right.
			\end{equation}
			which satisfies
			$$
			Y\in L^{2}(0,T;H^{2}(\mathbb{T}_{L}\times(0,1)))\cap H^{1}(0,T;L^{2}(\mathbb{T}_{L}\times(0,1)))\cap C^{0}([0,T];H^{1}(\mathbb{T}_{L}\times(0,1))).
			$$ 
			and  the following estimate:
			\begin{equation}\label{inYH}
			\begin{array}{ll}
			&\displaystyle s^{3}\lambda^{4}\iint\limits_{{{Q^{ex}_{T}}}}|Y|^{2}e^{2s\phi}+s\lambda^{2}\iint\limits_{{{Q^{ex}_{T}}}}\xi^{-2}|\nabla Y|^{2}e^{2s\phi}
			\\
			&\displaystyle+s^{2}\lambda^{3}\iint\limits_{{\mathbb{T}^{1}_{T}}}\xi^{-1}|Y|^{2}e^{2s\phi}+\iint\limits_{{\omega_{T}}}\xi^{-3}|H|^{2}e^{2s\phi}\leqslant C\iint\limits_{{{Q^{ex}_{T}}}}\xi^{-3}|G|^{2}e^{2s\phi}.
			\end{array}
			\end{equation}
		\end{lem}
		
		The proof of Theorem \ref{carlht} then follows from Lemma \ref{sublemreg} by duality. 
		
		Details of the proof are given hereafter. Section \ref{Subsec-Lem-Carlman} is dedicated to the proof of Lemma \ref{Carlman}, Section \ref{Subsec-Lem-Sublemreg} to the proof of Lemma \ref{sublemreg}, and Section \ref{Subsubsec-Thm-carlht} to the proof of Theorem \ref{carlht}.
		
		In Section \ref{Subsubsec-Add-Cor}, we prove an additional result, which is a corollary of Theorem \ref{carlht}:
		
		\begin{corollary}\label{carlht**}
			There exist positive constants ${C},$ ${s}_{2}(>s_{1})$ and ${\lambda}_{2}(>\lambda_{1})$ (where $s_{1}$ and $\lambda_{1}$ are the constants as in Theorem \ref{carlht}) such that if all the assumptions of Theorem \ref{carlht} satisfied and if $s\geqslant{s}_{2}$ and $\lambda\geqslant{\lambda}_{2},$ then the solution $q$ of \eqref{adjheateq} satisfies the following inequality
			\begin{equation}\label{obsnht***}
			\begin{split}
			\displaystyle
			&\iint\limits_{{{Q^{ex}_{T}}}}e^{-2s\phi}(\frac{1}{s \xi}|\nabla q|^{2}+s\lambda^{2}{\xi}|q|^{2})
			+\lambda\iint\limits_{{\mathbb{T}^{1}_{T}}}e^{-2s\phi}|q|^{2}\\
			&\leqslant C\left( \frac{1}{s^{2}\lambda^{2}}\iint\limits_{{{Q^{ex}_{T}}}}e^{-2s\phi}\frac{1}{\xi^{2}}|f_{1}|^{2}+\frac{1}{s \lambda}\iint\limits_{{\mathbb{T}^{1}_{T}}}e^{-2s\phi}\frac{1}{\xi}|f_{2}|^{2}
			\displaystyle
			+s\lambda^{2}\iint\limits_{{\omega_{T}}}e^{-2s\phi}{\xi}|q|^{2}
			\right),
			\end{split}
			\end{equation}
			where the notations $Q^{ex}_{T},$ $\mathbb{T}^{1}_{T}$ are introduced in \eqref{inshrthnd}.
		\end{corollary}

		\subsubsection{Proof of Lemma \ref{Carlman}}\label{Subsec-Lem-Carlman}
		
		We set
		\begin{equation}\label{Theta}
		\begin{split}
		\displaystyle{\vartheta_{1}}=e^{-s\phi}{\vartheta}.
		\end{split}
		\end{equation}
		In view of \eqref{norder} one observes that 
		\begin{equation}\label{bTheta}
		\begin{array}{l}
		\displaystyle
		\partial_{z}{\vartheta_{1}}=0\quad\mbox{on}\quad {\mathbb{T}^{1}_{T}}\cup \mathbb{T}^{0}_{T}.
		\end{array}
		\end{equation}
		Besides, with $f_{3}$ as in \eqref{hthne}, ${\vartheta_{1}}$ satisfies
		$$e^{-s\phi}f_{3}=e^{-s\phi}(-\frac{\bar{\rho}}{\nu}\partial_{t}{\vartheta}-\Delta {\vartheta})=e^{-s\phi}(-\frac{\bar{\rho}}{\nu}\partial_{t}(e^{s\phi}{\vartheta_{1}})-\Delta(e^{s\phi}{\vartheta_{1}}))=\mathcal{P}_{\phi}{\vartheta_{1}},$$
		where the operator $\mathcal{P}_{\phi}$ can be written as
		$$\mathcal{P}_{\phi}=\mathcal{P}_{1}{\vartheta_{1}}+\mathcal{P}_{2}{\vartheta_{1}}+\mathcal{R}{\vartheta_{1}},$$
		where
		\begin{equation}\label{P12R}
		\begin{array}{l}
		\displaystyle\mathcal{P}_{1}{\vartheta_{1}}=-\frac{\bar{\rho}}{\nu}\partial_{t}{\vartheta_{1}}+2s\lambda{\xi}\nabla{{\eta^{0}}}\cdot\nabla{\vartheta_{1}}+2s\lambda^{2}|\nabla{{\eta^{0}}}|^{2}{\xi}{\vartheta_{1}},\\
		\displaystyle\mathcal{P}_{2}{\vartheta_{1}}=-\Delta{\vartheta_{1}}-\frac{\bar{\rho}}{\nu}s\partial_{t}\phi{\vartheta_{1}}-s^{2}\lambda^{2}{\xi}^{2}|\nabla{{\eta^{0}}}|^{2}{\vartheta_{1}},\\
		\displaystyle\mathcal{R}{\vartheta_{1}}=-s\lambda\Delta{{\eta^{0}}}{\xi}{\vartheta_{1}}+s\lambda^{2}|\nabla{{\eta^{0}}}|^{2}{\xi}{\vartheta_{1}}.
		\end{array}
		\end{equation}
		We then use that $\mathcal{P}_{1}{\vartheta_{1}}+\mathcal{P}_{2}{\vartheta_{1}}=f_{3}e^{-s\phi}-\mathcal{R}{\vartheta_{1}}$ and then
		\begin{equation}\label{tinq}
		\begin{split}
		\displaystyle
		\iint\limits_{{Q^{ex}_{T}}}|\mathcal{P}_{1}{\vartheta_{1}}|^{2}+\iint\limits_{{Q^{ex}_{T}}}|\mathcal{P}_{2}{\vartheta_{1}}|^{2}+2\iint\limits_{{Q^{ex}_{T}}}\mathcal{P}_{1}{\vartheta_{1}}\mathcal{P}_{2}{\vartheta_{1}}&=\iint\limits_{{Q^{ex}_{T}}}|f_{3}e^{-s\phi}-\mathcal{R}{\vartheta_{1}}|^{2}\\
		&\leqslant 2\iint\limits_{{Q^{ex}_{T}}}|f_{3}|^{2}e^{-2s\phi}+2\iint\limits_{{Q^{ex}_{T}}}|\mathcal{R}{\vartheta_{1}}|^{2}.
		\end{split}
		\end{equation}
		We now compute the scalar product of $\mathcal{P}_{1}{\vartheta_{1}}$ with $\mathcal{P}_{2}{\vartheta_{1}}$. In fact, these computations are very similar to the classical ones, and one should only remark that the integrations by parts do not yield any bad terms on the boundaries of the domain.
		\\
		We shall write $LO.T.$ to design lower order terms, that is terms which can be bounded as follows: 
		$$
		|L.O.T| \leq 
		C \left(\frac{1}{s} + \frac{1}{\lambda} \right)\iint\limits_{{{Q^{ex}_{T}}}}e^{-2s\phi}\left(s\lambda^{2}{\xi}|\nabla{\vartheta}|^{2}+s^{3}\lambda^{4}{\xi}^{3}|{\vartheta}|^{2}\right).
		$$
		From now on we will be frequently using the estimates \eqref{prilies}, \eqref{prilies*} and \eqref{Positivity-Weights} without mentioning them precisely all the time. This estimates will also be used to furnish the $'L.O.T'$ terms.\\
		Note in particular that we have 
		$$
		\iint\limits_{{Q^{ex}_{T}}}|\mathcal{R}{\vartheta_{1}}|^{2} = L.O.T.
		$$
		\\
		\textbf{Computations.} We write
		$$\iint\limits_{{Q^{ex}_{T}}}\mathcal{P}_{1}{\vartheta_{1}}\mathcal{P}_{2}{\vartheta_{1}}=\sum\limits_{i,j=1}^{3}{J}_{ij},$$
		where $J_{i,j}$ is the scalar product of the $i$-th term of $\mathcal{P}_{1}{\vartheta_{1}}$ with the $j$-th term of $\mathcal{P}_{2}{\vartheta_{1}}.$\\
		\textit{Computation of}\,${J}_{11}:$
		\begin{equation}\label{ti11}
		\begin{split}
		{J}_{11}=\frac{\bar{\rho}}{\nu}\iint\limits_{{Q^{ex}_{T}}} \partial_{t}{{\vartheta_{1}}}\Delta{\vartheta_{1}}&=-\frac{\bar{\rho}}{\nu}\iint\limits_{{Q^{ex}_{T}}}\partial_{t}\left(\frac{|\nabla{\vartheta_{1}}|^{2}}{2}\right)=0.\\
		\end{split}
		\end{equation}
		\textit{Computation of}\,${J}_{12}:$
		\begin{equation}\label{ti12}
		\begin{split}
		{J}_{12}&=\left(\frac{\bar{\rho}}{\nu}\right)^{2}s\iint\limits_{{Q^{ex}_{T}}}\left(\partial_{t}{\vartheta_{1}}\cdot{\vartheta_{1}}\right)\partial_{t}\phi=\frac{1}{2}\left(\frac{\bar{\rho}}{\nu}\right)^{2}s\iint\limits_{{Q^{ex}_{T}}}\partial_{t}|{\vartheta_{1}}|^{2}\partial_{t}\phi\\
		&=-\frac{1}{2}\left(\frac{\bar{\rho}}{\nu}\right)^{2}s\iint\limits_{{Q^{ex}_{T}}}|{\vartheta_{1}}|^{2}\partial_{tt}\phi = L.O.T.
		\end{split}
		\end{equation}
		\textit{Computation of}\,${J}_{13}:$
		\begin{equation}\label{ti13}
		\begin{split}
		{J}_{13} &=\frac{\bar{\rho}}{\nu}s^{2}\lambda^{2}\iint\limits_{{Q^{ex}_{T}}}{\xi}^{2}|\nabla{{\eta^{0}}}|^{2}\partial_{t}{\vartheta_{1}}\cdot{\vartheta_{1}}=\frac{\bar{\rho}}{\nu}\frac{s^{2}\lambda^{2}}{2}\iint\limits_{{Q^{ex}_{T}}}\partial_{t}|{\vartheta_{1}}|^{2}{{\xi}}^{2}|\nabla{{\eta^{0}}}|^{2}\\
		&=-\frac{\bar{\rho}}{\nu}\frac{s^{2}\lambda^{2}}{2}\iint\limits_{{Q^{ex}_{T}}}|{\vartheta_{1}}|^{2}\partial_{t}({{\xi}}^{2}|\nabla{{\eta^{0}}}|^{2})
		= L.O.T.
		\end{split}
		\end{equation}
		\textit{Computation of}\,${J}_{21}:$
		\begin{align}\label{ti21}
		{J}_{21}&=-2s\lambda\iint\limits_{{Q^{ex}_{T}}}{\xi}(\nabla{{\eta^{0}}}\cdot\nabla{\vartheta_{1}})\Delta{\vartheta_{1}}=2s\lambda\iint\limits_{{Q^{ex}_{T}}}\nabla({\xi}(\nabla{{\eta^{0}}}\cdot\nabla{\vartheta_{1}}))\cdot\nabla{\vartheta_{1}}\notag\\
		&=2s\lambda^{2}\iint\limits_{{Q^{ex}_{T}}}{\xi}|\nabla{{\eta^{0}}}\cdot\nabla{\vartheta_{1}}|^{2}+2s\lambda\iint\limits_{{Q^{ex}_{T}}}{\xi}(\nabla^{2}{{\eta^{0}}}\cdot\nabla{\vartheta_{1}})\cdot\nabla{\vartheta_{1}}+s\lambda\iint\limits_{{Q^{ex}_{T}}}{\xi}\nabla{{\eta^{0}}}\cdot\nabla|\nabla{\vartheta_{1}}|^{2}\notag\\
		&=2s\lambda^{2}\iint\limits_{{Q^{ex}_{T}}}{\xi}|\nabla{{\eta^{0}}}\cdot\nabla{\vartheta_{1}}|^{2}+2s\lambda\iint\limits_{{Q^{ex}_{T}}}{\xi}(\nabla^{2}{{\eta^{0}}}\cdot\nabla{\vartheta_{1}})\cdot\nabla{\vartheta_{1}}-s\lambda^{2}\iint\limits_{{Q^{ex}_{T}}}{\xi}|\nabla{{\eta^{0}}}|^{2}|\nabla{\vartheta_{1}}|^{2}\notag\\
		&\qquad-s\lambda\iint\limits_{{Q^{ex}_{T}}}{\xi}\Delta{{\eta^{0}}}|\nabla{\vartheta_{1}}|^{2}
		\\ 
		& =2s\lambda^{2}\iint\limits_{{Q^{ex}_{T}}}{\xi}|\nabla{{\eta^{0}}}\cdot\nabla{\vartheta_{1}}|^{2}-s\lambda^{2}\iint\limits_{{Q^{ex}_{T}}}{\xi}|\nabla{{\eta^{0}}}|^{2}|\nabla{\vartheta_{1}}|^{2}+L.O.T.\notag
		\end{align}
		where the third line from the second in \eqref{ti21} follows because the boundary integral $\displaystyle\iint\limits_{\mathbb{T}^{1}_{T}\cup \mathbb{T}^{0}_{T}}\xi\partial_{z}\eta^{0}|\nabla\vartheta_{1}|^{2}$ vanishes since $\eta^{0}$ is only a function of $(x,t).$\\
		\textit{Computation of}\,${J}_{22}:$
		\begin{equation}\label{ti22}
		\begin{split}
		{J}_{22}&=-2\frac{\bar{\rho}}{\nu}s^{2}\lambda\iint\limits_{{Q^{ex}_{T}}}{\xi}(\nabla{{\eta^{0}}}\cdot\nabla{\vartheta_{1}})\partial_{t}\phi{\vartheta_{1}}
		=\frac{\bar{\rho}}{\nu}s^{2}\lambda\iint\limits_{{Q^{ex}_{T}}}\mbox{div}\left(\partial_{t}\phi\nabla{{\eta^{0}}}{\xi}\right)|{\vartheta_{1}}|^{2}
		= L.O.T.
		\end{split}
		\end{equation}
		\textit{Computation of}\,${J}_{23}:$
		\begin{equation}\label{ti23}
		\begin{split}
		{J}_{23}&=-2s^{3}\lambda^{3}\iint\limits_{{Q^{ex}_{T}}}{\xi}^{3}(\nabla{{\eta^{0}}}\cdot\nabla{\vartheta_{1}})|\nabla{{\eta^{0}}}|^{2}{\vartheta_{1}}
		=s^{3}\lambda^{3}\iint\limits_{{Q^{ex}_{T}}}\mbox{div}(|\nabla{{\eta^{0}}}|^{2}\nabla{{\eta^{0}}}{\xi}^{3})|{\vartheta_{1}}|^{2}.
		\end{split}
		\end{equation}
		\textit{Computation of}\,${J}_{31}:$
		\begin{equation}\label{ti31}
		\begin{split}
		{J}_{31}&=-2s\lambda^{2}\iint\limits_{{Q^{ex}_{T}}}|\nabla{{\eta^{0}}}|^{2}{\xi}{\vartheta_{1}}\Delta{\vartheta_{1}}\\
		&=2s\lambda^{2}\iint\limits_{{Q^{ex}_{T}}}|\nabla{{\eta^{0}}}|^{2}{\xi}|\nabla{\vartheta_{1}}|^{2}+2s\lambda^{2}\iint\limits_{{Q^{ex}_{T}}}\nabla(|\nabla{{\eta^{0}}}|^{2}{\xi}){\vartheta_{1}} \cdot \nabla \vartheta_1
		\\
		&=2s\lambda^{2}\iint\limits_{{Q^{ex}_{T}}}|\nabla{{\eta^{0}}}|^{2}{\xi}|\nabla{\vartheta_{1}}|^{2}-s\lambda^{2}\iint\limits_{{Q^{ex}_{T}}}\Delta(|\nabla{{\eta^{0}}}|^{2}{\xi})|{\vartheta_{1}}|^2
		\\
		&  =2s\lambda^{2}\iint\limits_{{Q^{ex}_{T}}}|\nabla{{\eta^{0}}}|^{2}{\xi}|\nabla{\vartheta_{1}}|^{2} + L.O.T.
		\end{split}
		\end{equation}
		\textit{Computation of}\,${J}_{32}:$
		\begin{equation}\label{ti32}
		\begin{split}
		{J}_{32}&=-2s^{2}\lambda^{2}\frac{\bar{\rho}}{\nu}\iint\limits_{{Q^{ex}_{T}}}|\nabla{{\eta^{0}}}|^{2}{\xi}|{\vartheta_{1}}|^{2}\partial_{t}\phi = L.O.T.
		\end{split}
		\end{equation}
		\textit{Computation of}\,${J}_{33}:$
		\begin{equation}\label{ti33}
		\begin{split}
		{J}_{33}=-2s^{3}\lambda^{4}\iint\limits_{{Q^{ex}_{T}}}{\xi}^{3}|\nabla{{\eta^{0}}}|^{4}|{\vartheta_{1}}|^{2}.
		\end{split}
		\end{equation}
		Combining the above computations \eqref{ti11}-\eqref{ti33}, we obtain the following:
		\begin{align}
		\label{proP12}
		&\iint\limits_{{Q^{ex}_{T}}}\mathcal{P}_{1}{\vartheta_{1}}\mathcal{P}_{2}{\vartheta_{1}}
		=
		2s\lambda^{2}\iint\limits_{{Q^{ex}_{T}}}{\xi}|\nabla{{\eta^{0}}}\cdot\nabla{\vartheta_{1}}|^{2}
		+s\lambda^{2}\iint\limits_{{Q^{ex}_{T}}}{\xi}|\nabla{{\eta^{0}}}|^{2}|\nabla{\vartheta_{1}}|^{2}
		\\
		&+s^3 \lambda^3  \iint\limits_{{Q^{ex}_{T}}}|{\vartheta_{1}}|^{2}\Bigg(\mbox{div}(|\nabla{{\eta^{0}}}|^{2}\nabla{{\eta^{0}}}{\xi}^{3})-2\lambda{\xi}^{3}|\nabla{{\eta^{0}}}|^{4}\Bigg)
		+L.O.T.
		\nonumber
		\end{align}
		Now, it is not hard to check that 
		$$
		s^3 \lambda^3  \iint\limits_{{Q^{ex}_{T}}}|{\vartheta_{1}}|^{2}\Bigg(\mbox{div}(|\nabla{{\eta^{0}}}|^{2}\nabla{{\eta^{0}}}{\xi}^{3})-2\lambda{\xi}^{3}|\nabla{{\eta^{0}}}|^{4}\Bigg)\geqslant cs^3 \lambda^4 \iint\limits_{{Q^{ex}_{T}}}\xi^{3}|\nabla\eta^{0}|^{4}|\vartheta_{1}|^{2}-L.O.T.
		$$
		We thus immediately deduce from \eqref{eta0grtr0} and \eqref{tinq} that there exists $C>0$ such that for all $s$ and $\lambda$ large enough,
		\begin{equation}\label{mainest1}
		\begin{split}
		\displaystyle\iint\limits_{{Q^{ex}_{T}}}&s\lambda^{2}{\xi}|\nabla\vartheta_1|^{2}+\iint\limits_{{Q^{ex}_{T}}}s^{3}\lambda^{4}{\xi}^{3}|\vartheta_1|^{2}\\
		\displaystyle\leqslant & C\left(\iint\limits_{{Q^{ex}_{T}}}e^{-2s\phi}|f_{3}|^{2}+s^{3}\lambda^{4}\iint\limits_{{\tilde \omega_{T}}}{\xi}^{3}|\vartheta_1|^{2}+s\lambda^{2}\iint\limits_{{\tilde \omega_{T}}}\xi|\nabla\vartheta_1|^{2} \right),
		\end{split}
		\end{equation} 
		where $\tilde \omega_T = (\mathbb{T}_L \setminus [- \overline u_1 T, d+ \overline u_1 T] \times (0,1) )\times (0,T)$.
		Now there are two steps to obtain \eqref{hthne1} from \eqref{mainest1}:
		\begin{itemize}
			\item Absorbing the observation in \eqref{mainest1} involving $\nabla\vartheta_{1}$ on $\tilde \omega_T$: This can be done classically by considering a slightly larger set of observation. For instance one can follow the arguments used in \cite[p. 565]{ervbad} or \cite[p. 461]{farnan}.
			\item One should then come back to the original unknown $\vartheta$ from $\vartheta_{1}$. This can be done as it is done classically by recalling that $\vartheta = e^{s \phi} \vartheta_1$. 
		\end{itemize}
		This concludes the proof of Lemma \ref{Carlman}.

		\subsubsection{Proof of Lemma \ref{sublemreg}}\label{Subsec-Lem-Sublemreg}
		The proof of this lemma will follow the arguments used in \cite[Theorem 2.6]{ervbad} and \cite[p. 446]{farnan} with some modifications.
		\\
		In order to solve \eqref{strng} we will introduce a functional $J$ whose Euler Lagrange equation provides a solution of \eqref{strng}. For smooth functions $\vartheta$ on $\overline{{Q^{ex}_{T}}}$ with $\partial_{z}\vartheta=0$ on $\overline{\mathbb{T}^{1}_{T}}\cup\overline{\mathbb{T}^{0}_{T}},$ let us define
		\begin{equation}\label{Jfunc}
		\begin{split}
		J(\vartheta)=\frac{1}{2}\iint\limits_{{{Q^{ex}_{T}}}}|(-\frac{\overline{\rho}}{\nu}\partial_{t}-\Delta)\vartheta|^{2}e^{-2s\phi}+\frac{s^{3}\lambda^{4}}{2}\iint\limits_{{\omega_{T}}}\xi^{3}|\vartheta|^{2}e^{-2s\phi}-\iint\limits_{{{Q^{ex}_{T}}}}G\vartheta.
		\end{split}
		\end{equation} 
		We introduce  the following space
		\begin{equation}\label{comobs}
		\begin{split}
		X_{obs}=\overline{\{\vartheta\in C^{\infty}(\overline{{Q^{ex}_{T}}})\,\,\mbox{such that}\,\,\partial_{z}\vartheta=0\,\,\mbox{on}\,\,\overline{\mathbb{T}^{1}_{T}}\cup\overline{\mathbb{T}^{0}_{T}}\}}^{\|\cdot\|_{obs}},
		\end{split}
		\end{equation}
		where $\|\cdot\|_{obs}$ is the Hilbert norm defined by
		\begin{equation}\label{hilnrm}
		\begin{split}
		\|\vartheta\|^{2}_{obs}=\iint\limits_{{{Q^{ex}_{T}}}}|(-\frac{\overline{\rho}}{\nu}\partial_{t}-\Delta)\vartheta|^{2}e^{-2s\phi}+s^{3}\lambda^{4}\iint\limits_{{\omega_{T}}}\xi^{3}|\vartheta|^{2}e^{-2s\phi}.
		\end{split}
		\end{equation}
		We endow the space $X_{obs}$ with the Hilbert structure given by $\|\cdot\|_{obs}.$ Of course, the fact that $\|\cdot\|_{obs}$ is a norm follows from the Carleman estimate \eqref{hthne1}.\\
		Observe that, from the assumption \eqref{condG} and the Carleman estimate \eqref{hthne1}, one has
		\begin{equation}\label{causch}
		\begin{split}
		\displaystyle
		\left|\iint\limits_{{{Q^{ex}_{T}}}}G\vartheta\right|\leqslant C\|\vartheta\|_{obs}\left(\frac{1}{s^{3}\lambda^{4}}\iint\limits_{{{Q^{ex}_{T}}}}\xi^{-3}|G|^{2}e^{2s\phi}\right)^{1/2},
		\end{split}
		\end{equation}
		for some constant $C>0.$ Hence in view of \eqref{causch}, one observes that the functional $J$ can be uniquely extended as a continuous functional on $X_{obs}.$ We denote this extension by the notation $J$ itself. The inequality \eqref{causch} also infers the coercivity of $J$ on $X_{obs}.$ It is easy to check that $J$ is strictly convex on $X_{obs}.$ Hence $J$ admits a unique minimizer $\vartheta_{\min}$ on $X_{obs}.$\\
		We further set
		\begin{equation}\label{setYH}
		\begin{split}
		Y=(-\frac{\overline{\rho}}{\nu}\partial_{t}-\Delta)\vartheta_{\min}e^{-2s\phi}\quad\mbox{and}\quad H=-s^{3}\lambda^{4}\xi^{3}\vartheta_{\min}e^{-2s\phi}\chi_{\omega_{T}}.
		\end{split}
		\end{equation}
		From the Euler Lagrange equation of $J$ at $\vartheta_{\min},$ for all smooth functions $\vartheta$ on $\overline{{Q^{ex}_{T}}}$ such that $\partial_{z}\vartheta=0$ on $\overline{T}^{1}_{T}\cup\overline{\mathbb{T}^{0}_{T}},$
		\begin{equation}\label{eullag}
		\begin{split}
		\iint\limits_{{{Q^{ex}_{T}}}}Y(-\frac{\overline{\rho}}{\nu}\partial_{t}\vartheta-\Delta\vartheta)-\iint\limits_{{{Q^{ex}_{T}}}}G\vartheta-\iint\limits_{{{\omega}_{T}}}H\vartheta=0,
		\end{split}
		\end{equation}
		This equation easily implies that the solution $Y$ 
		of \eqref{strng}$_{(1,2,3)}$ in the sense of transposition with source term $G+ H \chi_\omega$ with $H$ given by \eqref{setYH} coincides with the function $Y$ given by \eqref{setYH} and satisfies the null controllability requirement $Y(\cdot, T) = 0$ in $\mathbb{T}_L \times (0,1)$.
		\\
		Note that, since $G\in L^{2}({Q^{ex}_{T}})$ and $H\in L^{2}({Q^{ex}_{T}}),$ 
		$$
		Y\in L^{2}(0,T;H^{2}(\mathbb{T}_{L}\times(0,1)))\cap H^{1}(0,T;L^{2}(\mathbb{T}_{L}\times(0,1)))\cap C^{0}([0,T];H^{1}(\mathbb{T}_{L}\times(0,1))).
		$$ 
		\\
		Moreover since $\vartheta_{\min}$ is the minimizer of $J$ on $X_{obs},$ using $J(\vartheta_{\min})\leqslant J(0)=0$ and \eqref{causch} one gets
		\begin{equation}\label{estYH}
		\begin{split}
		s^{3}\lambda^{4}\iint\limits_{{{Q^{ex}_{T}}}}|Y|^{2}e^{2s\phi}+\iint\limits_{{{\omega}_{T}}}\xi^{-3}|H|^{2}e^{2s\phi}\leqslant C\iint\limits_{{{Q^{ex}_{T}}}}\xi^{-3}|G|^{2}e^{2s\phi},
		\end{split}
		\end{equation} 
		for large enough values of the parameters $s$ and $\lambda.$\\
		One then follows the arguments used in proving \cite[Theorem 2.6]{ervbad}, which mainly consists in a suitable energy estimate, i.e. a multiplication of \eqref{strng} by $\xi^{-2} e^{2 s \phi} Y$, to show the following 
		\begin{equation}\label{finalSt3}
		\begin{array}{l}
		\displaystyle s\lambda^{2}\iint\limits_{{{Q^{ex}_{T}}}}\xi^{-2}|\nabla Y|^{2}e^{2s\phi}
		\leqslant C\iint\limits_{{{Q^{ex}_{T}}}}\xi^{-3}|G|^{2}e^{2s\phi},
		\end{array}
		\end{equation}
		for large enough values of the parameters $s$ and $\lambda.$\\
		In order to finish proving \eqref{inYH} one only needs to show
		\begin{equation}\label{bndestYH}
		\begin{array}{l}
		\displaystyle s^{2}\lambda^{3}\iint\limits_{{\mathbb{T}^{1}_{T}}}\xi^{-1}|Y|^{2}e^{2s\phi}\leqslant C\iint\limits_{{{Q^{ex}_{T}}}}\xi^{-3}|G|^{2}e^{2s\phi},
		\end{array}
		\end{equation}
		for large enough values of the parameters $s$ and $\lambda.$\\
		This can be achieved thanks to the estimates \eqref{estYH} and \eqref{finalSt3}. In this direction we introduce a function 
		$$\kappa= \kappa (x,z)=\begin{pmatrix}
		0\\z
		\end{pmatrix}$$
		on $\mathbb{T}_{L}\times(0,1).$ The function $\kappa$ verifies 
		\begin{equation}\label{rsatkappa}
		\kappa\cdot n=\left\{ \begin{array}{l}
		1\,\,\mbox{on}\,\,\mathbb{T}_{L}\times\{1\},\\
		0\,\,\mbox{on}\,\, \mathbb{T}_{L}\times\{0\},
		\end{array}\right.
		\end{equation}
		where $n$ denotes the unit outward normal to $\mathbb{T}_{L}\times\{0,1\}.$\\
		Now we perform the following calculations, using that $\partial_z (\xi^{-1}e^{2s\phi}) = 0$:
		\begin{equation}\label{1ststpYb}
		\begin{split}
		s^{2}\lambda^{3}\iint\limits_{{\mathbb{T}^{1}_{T}}}\xi^{-1}|Y|^{2}e^{2s\phi}&=s^{2}\lambda^{3}\iint\limits_{{{Q^{ex}_{T}}}}\mbox{div}(\kappa\xi^{-1}|Y|^{2}e^{2s\phi})
		\\
		&
		=s^{2}\lambda^{3}\iint\limits_{{{Q^{ex}_{T}}}}\mbox{div}(\kappa\xi^{-1}e^{2s\phi}) |Y|^{2}
		+ 
		2 s^{2}\lambda^{3}\iint\limits_{{{Q^{ex}_{T}}}} \kappa\xi^{-1}e^{2s\phi} Y \cdot \nabla Y
		\\ 
		&
		=s^{2}\lambda^{3}\iint\limits_{{{Q^{ex}_{T}}}}\xi^{-1}e^{2s\phi} |Y|^{2}+
		2 s^{2}\lambda^{3}\iint\limits_{{{Q^{ex}_{T}}}} \kappa\xi^{-1}e^{2s\phi} Y \cdot \nabla Y 
		\\
		& \leq
		C \left( s^{3}\lambda^{4}\iint\limits_{{{Q^{ex}_{T}}}}|Y|^{2}e^{2s\phi} +s\lambda^{2}\iint\limits_{{{Q^{ex}_{T}}}} \xi^{-2}|\nabla Y|^{2}e^{2s\phi}\right).
		\end{split}
		\end{equation}
		Finally combining the estimates \eqref{estYH}, \eqref{finalSt3} and \eqref{1ststpYb} we prove \eqref{bndestYH}. Hence we conclude the proof of the estimate \eqref{inYH}.
		
		\subsubsection{Proof of Theorem \ref{carlht}}\label{Subsubsec-Thm-carlht}
		
		Since we only assume that $f_{1}\in{Q^{ex}_{T}}$, $f_{2}\in L^{2}(\mathbb{T}^{1}_{T}),$  the solution of the problem \eqref{adjheateq} has to be considered in the sense of transposition.
		%
		In particular, for 
		$$
		Y\in L^{2}(0,T;H^{2}(\mathbb{T}_{L}\times(0,1)))\cap H^{1}(0,T;L^{2}(\mathbb{T}_{L}\times(0,1)))\cap C^{0}([0,T];H^{1}(\mathbb{T}_{L}\times(0,1))), 
		$$ 
		satisfying $\partial_z Y = 0$ on  $\mathbb{T}_T^1 \cup \mathbb{T}_T^0$ with $Y(\cdot, 0) = Y(\cdot, T) = 0$ in $\mathbb{T}_L \times (0,1)$, we have
		\begin{equation}\label{muYq}
		\iint\limits_{{{Q^{ex}_{T}}}}f_{1}Y=\iint\limits_{{{Q^{ex}_{T}}}}(\frac{\overline{\rho}}{\nu}{\partial_{t}Y-\Delta Y})q-\iint\limits_{{\mathbb{T}^{1}_{T}}}Yf_{2}
		\end{equation}
		We thus choose a particular function $G$ which satisfies \eqref{condG}, namely
		\begin{equation}\label{setGp}
		\begin{split}
		G=\xi^{3}qe^{-2s\phi}, 
		\end{split}
		\end{equation}
		and $Y$ the function given by Lemma \ref{sublemreg}, 
		so that \eqref{muYq} yields:
		\begin{equation}
		\label{muYq-bis}
		\iint\limits_{{{Q^{ex}_{T}}}}\xi^{3}|q|^2 e^{-2s\phi}
		= 
		\iint\limits_{{{Q^{ex}_{T}}}}f_{1}Y
		+ 
		\iint\limits_{{\mathbb{T}^{1}_{T}}}Yf_{2}
		- 
		\iint\limits_{{{Q^{ex}_{T}}}}q H \chi_\omega.
		\end{equation}
		
		With the choice \eqref{setGp} of $G,$ observe in particular that
		$$\iint\limits_{{{Q^{ex}_{T}}}}\xi^{-3}|G|^{2}e^{2s\phi}=\iint\limits_{{{Q^{ex}_{T}}}}\xi^{3}|q|^{2}e^{-2s\phi}.$$
		Besides \eqref{muYq-bis} furnishes, for all $\epsilon>0,$
		\begin{multline}\label{dermuYq}
		\iint\limits_{{{Q^{ex}_{T}}}}\xi^{3}|q|^{2}e^{-2s\phi}
		\leqslant 
		C \left(\epsilon s^{3}\lambda^{4}\iint\limits_{{{Q^{ex}_{T}}}}|Y|^{2}e^{2s\phi}+\frac{1}{\epsilon s^{3}\lambda^{4}}\iint\limits_{{{Q^{ex}_{T}}}}|f_{1}|^{2}e^{-2s\phi}
		+\frac{1}{\epsilon}\iint\limits_{{\omega_{T}}}\xi^{3}|q|^{2}e^{-2s\phi}
		\right.
		\\
		\left.+\epsilon\iint\limits_{{\omega_{T}}} \xi^{-3}|H|^{2}e^{2s\phi}
		+ \frac{1}{\epsilon s^{2}\lambda^{3}}\iint\limits_{{\mathbb{T}^{1}_{T}}}\xi|f_{2}|^{2}e^{-2s\phi}+\epsilon s^{2}\lambda^{3}\iint\limits_{{\mathbb{T}^{1}_{T}}}\xi^{-1}|Y|^{2}e^{2s\phi}\right).
		\end{multline}
		On the other hand with the particular choice \eqref{setGp} of $G,$ the inequality \eqref{inYH} in particular takes the form
		\begin{equation}\label{inYHwG}
		\begin{split}
		&\displaystyle s^{3}\lambda^{4}\iint\limits_{{{Q^{ex}_{T}}}}|Y|^{2}e^{2s\phi}
		+s^{2}\lambda^{3}\iint\limits_{{\mathbb{T}^{1}_{T}}}\xi^{-1}|Y|^{2}e^{2s\phi}+\iint\limits_{{\omega_{T}}}\xi^{-3}|H|^{2}e^{2s\phi}\leqslant C\iint\limits_{{{Q^{ex}_{T}}}}\xi^{3}|q|^{2}e^{-2s\phi}.
		\end{split}
		\end{equation}
		Incorporating \eqref{inYHwG} in \eqref{dermuYq} and choosing small enough value for the parameter $\epsilon,$ we prove that 
		\begin{equation}\label{estq1}
		\begin{array}{l}
		\displaystyle s^{3}\lambda^{4}\iint\limits_{{{Q^{ex}_{T}}}}{\xi}^{3}|q|^{2}e^{-2s\phi}\\
		\displaystyle\leqslant C\left( \iint\limits_{{{Q^{ex}_{T}}}}e^{-2s\phi}|f_{1}|^{2}+s\lambda\iint\limits_{{\mathbb{T}^{1}_{T}}}e^{-2s\phi}{\xi}|f_{2}|^{2}
		\displaystyle+s^{3}\lambda^{4}\iint\limits_{{\omega_{T}}}e^{-2s\phi}{\xi}^{3}|q|^{2} \right).
		\end{array}
		\end{equation}
		Now one needs to estimate the integral of $\nabla q$ on $Q^{ex}_{T}$ and the integral of $q$ on $\mathbb{T}^{1}_{T}$ to finish proving the inequality \eqref{obsnht}. 
		\\
		In order to do that, we perform a weighted energy estimate on $q$ by multiplying \eqref{adjheateq} by $\xi q e^{- 2 s \phi}$. This and the estimates \eqref{prilies}, \eqref{prilies*} lead to 
		\begin{align}
		&
		\iint\limits_{{{Q^{ex}_{T}}}}{\xi} |\nabla q|^{2}e^{-2s\phi}\notag\\
		&\leq
		C(1+\frac{1}{\epsilon}) s \lambda^2 \iint\limits_{{{Q^{ex}_{T}}}}{\xi^3} |q|^{2}e^{-2s\phi}
		+ C\epsilon\iint\limits_{{{Q^{ex}_{T}}}} \xi|\nabla q|^{2}e^{-2s\phi}+
		C  \iint\limits_{{{Q^{ex}_{T}}}}\xi |q| |f_1| e^{-2s\phi}
		+ 
		C \iint\limits_{{\mathbb{T}^{1}_{T}}}e^{-2s\phi}{\xi} |f_{2}| |q|
		\notag
		\\
		&
		\leq
		C(1+\frac{1}{\epsilon}) s \lambda^2 \iint\limits_{{{Q^{ex}_{T}}}}{\xi^3} |q|^{2}e^{-2s\phi}
		+ C\epsilon\iint\limits_{{{Q^{ex}_{T}}}} \xi|\nabla q|^{2}e^{-2s\phi}
		+ 
		\frac{C}{s \lambda^2}  \iint\limits_{{{Q^{ex}_{T}}}} |f_1|^2 e^{-2s\phi}
		+ 
		C \lambda \iint\limits_{{\mathbb{T}^{1}_{T}}}e^{-2s\phi}{\xi} |q|^2\notag\\
		&\quad+ 
		\frac{C}{\lambda} \iint\limits_{{\mathbb{T}^{1}_{T}}}e^{-2s\phi} \xi |f_2|^2,
		\label{Est-Nabla-q}
		\end{align}
		for some positive parameter $\epsilon.$\\
		Now, similarly as in \eqref{1ststpYb}, we can obtain
		\begin{align}
		\label{est-boundary}
		s \lambda \iint\limits_{{\mathbb{T}^{1}_{T}}}e^{-2s\phi}{\xi} |q|^2
		& \leq s \lambda \iint\limits_{{\mathbb{T}^{1}_{T}}}e^{-2s\phi}{\xi^2} |q|^2
		\\
		& \leq
		C 
		\left( 
		s^2\lambda^2 (1+\frac{1}{\varepsilon})
		\iint\limits_{{{Q^{ex}_{T}}}}{\xi}^{3} 
		|q|^{2}e^{-2s\phi}
		+ 
		{\varepsilon} \iint\limits_{{{Q^{ex}_{T}}}} \xi|\nabla q|^{2}e^{-2s\phi}
		\right).
		\notag
		\end{align}
		Therefore, choosing $\varepsilon >0$ small enough, we deduce from \eqref{Est-Nabla-q} and \eqref{est-boundary} that there exists a constant $C>0$ such that for all $s$ and $\lambda$ large enough, 
		\begin{multline*}
		\iint\limits_{{{Q^{ex}_{T}}}}{\xi} |\nabla q|^{2}e^{-2s\phi}
		+
		s \lambda \iint\limits_{{\mathbb{T}^{1}_{T}}}e^{-2s\phi}{\xi^2} |q|^2
		\\
		\leq
		C s^2 \lambda^2 \iint\limits_{{{Q^{ex}_{T}}}}{\xi^3} |q|^{2}e^{-2s\phi}
		+ 
		\frac{C}{s\lambda^2}  \iint\limits_{{{Q^{ex}_{T}}}} |f_1|^2 e^{-2s\phi}
		+ 
		\frac{C}{\lambda} \iint\limits_{{\mathbb{T}^{1}_{T}}}e^{-2s\phi} \xi |f_2|^2.
		\end{multline*}
		With the estimate \eqref{estq1}, this finishes the proof of Theorem \ref{carlht}.
		\subsubsection{Proof of Corollary \ref{carlht**}}\label{Subsubsec-Add-Cor}
		
		Let us introduce the new unknown:
		$$q_{1}=\frac{1}{s\lambda \xi}q.$$
		From \eqref{adjheateq} we obtain the following system satisfied by $q_1$:
		\begin{equation}\label{adjheateq*}
		\left\{ \begin{array}{ll}
		\displaystyle-\frac{\overline{\rho}}{\nu}\partial_{t}q_1-\Delta q_1=\frac{1}{s\lambda\xi}f_{1}+\mathcal{F}_{5}\quad&\mbox{in}\quad Q^{ex}_{T},\\
		\displaystyle\partial_{z}q_1=\frac{1}{s\lambda \xi}f_{2}\quad&\mbox{on}\quad\mathbb{T}^{1}_{T},\\
		\displaystyle\partial_{z}q_1=0\quad&\mbox{on}\quad\mathbb{T}^{0}_{T},\\
		\displaystyle q_1(.,T)=\frac{1}{s\lambda \xi}q_{T}\quad & \mbox{in}\quad \mathbb{T}_{L}\times(0,1),
		\end{array}\right.
		\end{equation}
		where
		\begin{equation}\label{exF5}
		\begin{array}{l}
		\displaystyle \mathcal{F}_{5}=\frac{\overline{\rho}}{\nu}\frac{1}{s\lambda \xi^{2}}\partial_{t}\xi q+\frac{1}{s\lambda \xi^{2}}\nabla q\cdot\nabla\xi+\frac{1}{s\lambda \xi^{2}}q\Delta\xi+\frac{2}{s\lambda \xi^{3}}|\nabla\xi|^{2}q-\frac{1}{s\lambda \xi^{2}}\nabla\xi\cdot\nabla q.
		\end{array}
		\end{equation}
		In view of the estimates \eqref{prilies*}, $\mathcal{F}_{5}$ satisfies
		\begin{equation}\label{estF5}
		\begin{array}{l}
		\displaystyle\iint\limits_{{{Q^{ex}_{T}}}}e^{-2s\phi}|\mathcal{F}_{5}|^{2}\leqslant C\left(\frac{\lambda^2}{s^{2}}\iint\limits_{{{Q^{ex}_{T}}}}e^{-2s\phi}\frac{1}{\xi}|q|^{2}+\frac{1}{s^{2}}\iint\limits_{{{Q^{ex}_{T}}}}e^{-2s\phi}\frac{1}{\xi^{2}}|\nabla q|^{2} \right).
		\end{array}
		\end{equation}
		Now we apply Theorem \ref{carlht} for the system \eqref{adjheateq*} to have the following
		\begin{equation}\label{obsnht**}
		\begin{split}
		\displaystyle
		&\iint\limits_{{{Q^{ex}_{T}}}}e^{-2s\phi}(s\lambda^{2}{\xi}|\nabla q_1|^{2}+s^{3}\lambda^{4}{\xi}^{3}|q_1|^{2})
		+s^{2}\lambda^{3}\iint\limits_{{\mathbb{T}^{1}_{T}}}e^{-2s\phi}{\xi}^{2}|q_1|^{2}\\
		&\leqslant C\left( \frac{1}{s^{2}\lambda^{2}}\iint\limits_{{{Q^{ex}_{T}}}}e^{-2s\phi}\frac{1}{\xi^{2}}|f_{1}|^{2}+\iint\limits_{{{Q^{ex}_{T}}}}e^{-2s\phi}|\mathcal{F}_{5}|^{2}+\frac{1}{s\lambda}\iint\limits_{{\mathbb{T}^{1}_{T}}}e^{-2s\phi}{\xi}\frac{1}{\xi^2}|f_{2}|^{2}\right.\\
		&\qquad
		\displaystyle
		\left. +s^{3}\lambda^{4}\iint\limits_{{\omega_{T}}}e^{-2s\phi}{\xi}^{3}|q_1|^{2}\right)
		\\
		&\displaystyle\leqslant C\left(
		\frac{1}{s^{2}\lambda^{2}}\iint\limits_{{{Q^{ex}_{T}}}}e^{-2s\phi}\frac{1}{\xi^{2}}|f_{1}|^{2}
		+\frac{\lambda^2}{s^{2}}\iint\limits_{{{Q^{ex}_{T}}}}e^{-2s\phi}\frac{1}{\xi}|q|^{2}
		+\frac{1}{s^{2}}\iint\limits_{{{Q^{ex}_{T}}}}e^{-2s\phi}\frac{1}{\xi^{2}}|\nabla q|^{2}
		\right.
		\\
		& \left.
		\quad+\frac{1}{s\lambda}\iint\limits_{{\mathbb{T}^{1}_{T}}}e^{-2s\phi}\frac{1}{\xi}|f_{2}|^{2}
		\displaystyle
		+s\lambda^{2}\iint\limits_{{\omega_{T}}}e^{-2s\phi}{\xi}|q|^{2}
		\right),
		\end{split}
		\end{equation}
		where the last step of \eqref{obsnht**} from the penultimate step follows by using \eqref{estF5} and the definition of $q_1.$ 
		Now, using $q = s \lambda \xi q_1$,  one further checks the following for sufficiently large vales of $s$ and $\lambda$
		\begin{multline}
		\label{Comparison-q-1-q}
		\iint\limits_{{{Q^{ex}_{T}}}}e^{-2s\phi}(\frac{1}{ s \xi} |\nabla q|^{2}+s\lambda^{2}{\xi}|q|^{2})
		+ \lambda \iint\limits_{{\mathbb{T}^{1}_{T}}}e^{-2s\phi}|q|^{2}
		\\
		\leq C
		\left(
		\iint\limits_{{{Q^{ex}_{T}}}}e^{-2s\phi}(s\lambda^{2}{\xi}|\nabla q_1|^{2}+s^{3}\lambda^{4}{\xi}^{3}|q_1|^{2})
		+s^{2}\lambda^{3}\iint\limits_{{\mathbb{T}^{1}_{T}}}e^{-2s\phi}{\xi}^{2}|q_1|^{2}
		\right).
		\end{multline}
		Combining \eqref{obsnht**} and \eqref{Comparison-q-1-q}, we get
		\begin{multline*}		
		\iint\limits_{{{Q^{ex}_{T}}}}e^{-2s\phi}(\frac{1}{ s \xi} |\nabla q|^{2}+s\lambda^{2}{\xi}|q|^{2})
		+ \lambda \iint\limits_{{\mathbb{T}^{1}_{T}}}e^{-2s\phi}|q|^{2}
		\\
		\displaystyle\leqslant C\left(
		\frac{1}{s^{2}\lambda^{2}}\iint\limits_{{{Q^{ex}_{T}}}}e^{-2s\phi}\frac{1}{\xi^{2}}|f_{1}|^{2}
		+\frac{\lambda^2}{s^{2}}\iint\limits_{{{Q^{ex}_{T}}}}e^{-2s\phi}\frac{1}{\xi}|q|^{2}
		+\frac{1}{s^{2}}\iint\limits_{{{Q^{ex}_{T}}}}e^{-2s\phi}\frac{1}{\xi^{2}}|\nabla q|^{2}
		\right.
		\\
		\left.
		\quad+\frac{1}{s\lambda}\iint\limits_{{\mathbb{T}^{1}_{T}}}e^{-2s\phi}\frac{1}{\xi}|f_{2}|^{2}
		\displaystyle
		+s\lambda^{2}\iint\limits_{{\omega_{T}}}e^{-2s\phi}{\xi}|q|^{2}
		\right).
		\end{multline*}
		Taking $s$ and $\lambda$ large enough, we conclude Corollary \ref{carlht**}.
		\subsection{Observability of an adjoint transport equation}\label{sectrans}
		In this section we derive an observability inequality for the adjoint transport equation
		\begin{equation}\label{adjtra}
		\left\{ \begin{array}{lll}
		&-\partial_{t}\sigma-\overline{u}_{1}\partial_{x}\sigma+\frac{P'(\overline{\rho})}{\nu}\sigma=f_{4}\,& \mbox{in}\, Q^{ex}_{T},
		\vspace{1.mm}\\
		& {\sigma}(\cdot,T)=\sigma_{T}\,& \mbox{in}\, \mathbb{T}_{L}\times(0,1).
		\end{array}\right.
		\end{equation}
		\begin{thm}
			\label{Thm-obs-transport}
			Let us recall the notations $Q^{ex}_{T}$ and $\omega_{T}$ introduced in \eqref{inshrthnd}. There exists a positive constant $C$ such that for all $\sigma_T \in L^2(\mathbb{T}_{L}\times(0,1))$, $f_{4}\in L^{2}({Q^{ex}_{T}})$ and for all values of the parameters $s\geqslant 1$ and $\lambda \geq 1$, the solution $\sigma$ of \eqref{adjtra} satisfies the following inequality
			\begin{equation}\label{obsadjtra}
			\begin{split}
			\|\xi^{-\frac{1}{2}}\sigma e^{-s\phi}\|^{2}_{L^{2}({Q^{ex}_{T}})}
			\leqslant C(\|\xi^{-\frac{1}{2}}f_{4}e^{-s\phi}\|^{2}_{L^{2}({Q^{ex}_{T}})}
			+\|\xi^{-\frac{1}{2}}\sigma e^{-s\phi}\|^{2}_{L^{2}(\omega_{T})}).
			\end{split}
			\end{equation}
			There exists a positive constant $C$ such that for all $\sigma_T \in H^1(\mathbb{T}_{L}\times(0,1))$, $f_{4}\in L^{2}(0,T;H^{1}(\mathbb{T}_{L}\times(0,1))),$ and for all values of the parameter $s\geqslant 1$  and $\lambda \geq 1,$ the solution $\sigma$ of \eqref{adjtra} satisfies the following inequality
			\begin{equation}\label{obsadjtra*}
			\begin{split}
			\|\xi^{-\frac{1}{2}}\nabla\sigma e^{-s\phi}\|^{2}_{L^{2}({Q^{ex}_{T}})}
			\leqslant C(\|\xi^{-\frac{1}{2}}\nabla f_{4}e^{-s\phi}\|^{2}_{L^{2}({Q^{ex}_{T}})}
			+\|\xi^{-\frac{1}{2}}\nabla\sigma e^{-s\phi}\|^{2}_{L^{2}(\omega_{T})}).
			\end{split}
			\end{equation}
			Consequently there exists a positive constant $C$ such that for all $f_{4}\in L^{2}(0,T;H^{1}(\mathbb{T}_{L}\times(0,1))),$ and for all values of the parameter $s\geqslant 1$  and $\lambda \geq 1,$ the solution $\sigma$ of \eqref{adjtra} satisfies the following inequality
			\begin{equation}\label{obsadjtra**}
			\begin{split}
			\|\xi^{-\frac{1}{2}}\partial_{t}\sigma e^{-s\phi}\|^{2}_{L^{2}({Q^{ex}_{T}})}
			\leqslant& C(\|\xi^{-\frac{1}{2}}\nabla f_{4}e^{-s\phi}\|^{2}_{L^{2}({Q^{ex}_{T}})}+\| \xi^{-\frac{1}{2}}f_{4}e^{-s\phi}\|^{2}_{L^{2}({Q^{ex}_{T}})}\\
			&+\|\xi^{-\frac{1}{2}}\sigma e^{-s\phi}\|^{2}_{L^{2}(\omega_{T})}+\|\xi^{-\frac{1}{2}}\nabla\sigma e^{-s\phi}\|^{2}_{L^{2}(\omega_{T})}).
			\end{split}
			\end{equation}
		\end{thm}
		\begin{proof}
			We first prove the inequality \eqref{obsadjtra}. The proof depends on a controllability estimate of the following problem
			\begin{equation}\label{dualcont}
			\left\{ \begin{array}{lll}
			&\displaystyle\partial_{t}\widetilde\sigma+\overline{u}_{1}\partial_{x}\widetilde\sigma+\frac{P'(\overline{\rho})}{\nu}\widetilde\sigma=\widetilde{f}_{4}+v_{\widetilde{\sigma}}\chi_{\omega}& \mbox{in}\, {Q^{ex}_{T}},
			\vspace{1.mm}\\
			& \displaystyle\widetilde{\sigma}(\cdot,0)=0\,& \mbox{in}\, \mathbb{T}_{L}\times(0,1),\\
			& \displaystyle\widetilde{\sigma}(\cdot,T)=0\,& \mbox{in}\, \mathbb{T}_{L}\times(0,1).
			\end{array}\right.
			\end{equation}
			and a duality argument.\\
			We suppose that the source term $\widetilde{f}_{4}$ satisfies:
			$$\left\|\xi^{\frac{1}{2}} e^{s\phi}\widetilde{f}_{4}\right\|_{L^{2}(Q^{ex}_{T})}<\infty.$$
			In order to obtain controllability estimates for the problem \eqref{dualcont} we will follow the arguments used in \cite[Theorem 3.5]{ervguglachp3} but here with different multipliers.\\
			The controlled trajectory $\widetilde{\sigma}$ and the control function $v_{\widetilde{\sigma}}\chi_{\omega}$ can be constructed by adapting \cite[eq. $(3.26)$ ]{ervguglachp3} and \cite[eq. $(3.28)$]{ervguglachp3} in a straight forward manner. To obtain \eqref{obsadjtra} we recall at a glance the strategy to construct the controlled trajectory $\widetilde{\sigma}.$ It is done by gluing the solutions to the forward problem
			 \begin{equation}\label{dualcontforward}
			 \left\{ \begin{array}{lll}
			 &\displaystyle\partial_{t}\widetilde\sigma_{f}+\overline{u}_{1}\partial_{x}\widetilde\sigma_{f}+\frac{P'(\overline{\rho})}{\nu}\widetilde\sigma_{f}=\widetilde{f}_{4}& \mbox{in}\, {Q^{ex}_{T}},
			 \vspace{1.mm}\\
			 & \displaystyle\widetilde{\sigma}_{f}(\cdot,0)=0\,& \mbox{in}\, \mathbb{T}_{L}\times(0,1),\\
			 \end{array}\right.
			 \end{equation}
			 and the backward problem 
			 \begin{equation}\label{dualcontbackward}
			 \left\{ \begin{array}{lll}
			 &\displaystyle\partial_{t}\widetilde\sigma_{b}+\overline{u}_{1}\partial_{x}\widetilde\sigma_{b}+\frac{P'(\overline{\rho})}{\nu}\widetilde\sigma_{b}=\widetilde{f}_{4}& \mbox{in}\, {Q^{ex}_{T}},
			 \vspace{1.mm}\\
			 & \displaystyle\widetilde{\sigma}(\cdot,T)=0\,& \mbox{in}\, \mathbb{T}_{L}\times(0,1),
			 \end{array}\right.
			 \end{equation}
			 using suitable cut off functions.\\
			 $\textit{Esimates on}$ $\widetilde{\sigma}:$ In view of the construction \cite[eq. (3.26) ]{ervguglachp3} of the controlled trajectory, one first proves estimates for $\widetilde{\sigma}_{f}$ and $\widetilde{\sigma}_{b}.$ In that direction we test \eqref{dualcontforward}$_{1}$ by $\xi e^{2s\phi}\widetilde{\sigma}_{f}$ to furnish
			 \begin{equation}\label{energysf}
			 \begin{array}{ll}
			 \displaystyle\frac{d}{dt}\left(\frac{1}{2}\int\limits_{\mathbb{T}_L \times (0,1)}\xi e^{2s\phi}|\widetilde{\sigma}_{f}|^{2}\right)\leqslant&\displaystyle\frac{1}{2}\int\limits_{\mathbb{T}_L \times (0,1)}|\widetilde{\sigma}_{f}|^{2}\left(-2\frac{P'(\overline{\rho})}{\nu}\xi e^{2s\phi}+\left(\partial_{t}+\overline{u}_{1}\partial_{x}\right)(\xi e^{2s\phi})\right)\\
			 &+\displaystyle \left(\int\limits_{\mathbb{T}_L \times (0,1)}\xi e^{2s\phi}|\widetilde{f}_{4}|^{2}\right)^{\frac{1}{2}}\left(\int\limits_{\mathbb{T}_L \times (0,1)}\xi e^{2s\phi}|\widetilde{\sigma}_{f}|^{2}\right)^{\frac{1}{2}}.
			 \end{array}
			 	\end{equation}
			  Further in view of \eqref{transportweight} and \eqref{theta} (especially the fact that $\theta(t)$ is constant in $[2T_{0},T-2T_{1}]$) one has
			  $$\left(\partial_{t}+\overline{u}_{1}\partial_{x}\right)(\xi e^{2s\phi})\leqslant 0\quad\mbox{in}\quad (\mathbb{T}_{L}\times(0,1))\times(0,T-2T_{1}).$$
			  Hence \eqref{energysf} and Gr\"{o}nwall's inequality at once gives 
			  \begin{equation}\label{Gronwall}
			  \begin{array}{l}
			  \displaystyle\|\xi^{\frac{1}{2}}e^{s\phi}\widetilde{\sigma}_{f}\|^{2}_{L^{\infty}(0,T-2T_{1};L^{2}(\mathbb{T}_{L}\times(0,1)))}\leqslant C\|\xi^{\frac{1}{2}}e^{s\phi}\widetilde{f}_{4}\|^{2}_{L^{2}(0,T-2T_{1};L^{2}(\mathbb{T}_{L}\times(0,1)))}.
			  \end{array}
			  	\end{equation}
			  	Similarly one proves
			  	 \begin{equation}\label{Gronwallback}
			  	 \begin{array}{l}
			  	 \displaystyle\|\xi^{\frac{1}{2}}e^{s\phi}\widetilde{\sigma}_{b}\|^{2}_{L^{\infty}(T_{0},T;L^{2}(\mathbb{T}_{L}\times(0,1)))}\leqslant C\|\xi^{\frac{1}{2}}e^{s\phi}\widetilde{f}_{4}\|^{2}_{L^{2}(T_{0},T;L^{2}(\mathbb{T}_{L}\times(0,1)))}.
			  	 \end{array}
			  	 \end{equation}
			  	 Using the estimates \eqref{Gronwall} and \eqref{Gronwallback} and adapting the expression \cite[eq. (3.26)]{ervguglachp3} of the controlled trajectory, one at once renders
			  	 \begin{equation}\label{estimatesgma}
			  	 \begin{array}{l}
			  	 \displaystyle\|\xi^{\frac{1}{2}}e^{s\phi}\widetilde{\sigma}\|^{2}_{L^{\infty}(0,T;L^{2}(\mathbb{T}_{L}\times(0,1)))}\leqslant C\|\xi^{\frac{1}{2}}e^{s\phi}\widetilde{f}_{4}\|^{2}_{L^{2}(Q^{ex}_{T})}.
			  	 \end{array}
			  	 \end{equation}
			  	 Further one can easily adapt the construction \cite[eq. (3.28)]{ervguglachp3} of the control function and use the estimates \eqref{Gronwall} and \eqref{Gronwallback} to furnish
			  	 \begin{equation}\label{estimatecontrol}
			  	 \begin{array}{l}
			  	 \displaystyle\|\xi^{\frac{1}{2}}e^{s\phi}v_{\widetilde\sigma}\|^{2}_{L^{2}(Q^{ex}_{T})}\leqslant C\|\xi^{\frac{1}{2}}e^{s\phi}\widetilde{f}_{4}\|^{2}_{L^{2}(Q^{ex}_{T})}.
			  	 \end{array}
			  	 \end{equation}
		Now the observability estimate \eqref{obsadjtra} follows from \eqref{estimatesgma} and \eqref{estimatecontrol} by using the following duality arguments:
	 	\begin{align*}
			& \displaystyle\|\xi^{-\frac{1}{2}}\sigma e^{-s\phi}\|_{L^{2}(Q^{ex}_{T})}
			\displaystyle=\sup_{\|\xi^{\frac{1}{2}}\widetilde{f}_{4}e^{s\phi}\|_{L^{2}(Q^{ex}_{T})}\leqslant 1}|\langle\widetilde{f}_{4},\sigma\rangle_{L^{2}(Q^{ex}_{T})}|\\
			&\displaystyle \leqslant \sup_{\|\xi^{\frac{1}{2}}\widetilde{f}_{4}e^{s\phi}\|_{L^{2}(Q^{ex}_{T})}\leqslant 1}(|\langle\widetilde{\sigma},f_{4}\rangle_{L^{2}(Q^{ex}_{T})}|+|\langle v_{\widetilde{\sigma}}\chi_{\omega},\sigma\rangle_{L^{2}(Q^{ex}_{T})}|)\\
			& \displaystyle \leqslant C( 
			\|\xi^{-\frac{1}{2}}f_{4}e^{-s\phi}\|_{L^{2}(Q^{ex}_{T})}
			+
			\|\xi^{-\frac{1}{2}}\sigma e^{-s\phi}\|_{L^{2}(\omega_{T})}).
			\end{align*}
			This provides the inequality \eqref{obsadjtra}.
			\\
			The inequality \eqref{obsadjtra*} can be obtained by applying \eqref{obsadjtra} to the system satisfied by $\nabla\sigma.$ Once we have \eqref{obsadjtra*}, the estimate \eqref{obsadjtra**} can be obtained directly by using the equation \eqref{adjtra}$_{1}.$
		\end{proof}
		The following result is a corollary to Theorem \ref{Thm-obs-transport} and corresponds to the weighted estimates of $\sigma$ and $\nabla\sigma$ in $L^\infty(L^{2})$ norms. The following result will be used in the next section, more precisely in the proof of Lemma \ref{Lem-Obs-sigma-q}, to obtain an estimate of $\sigma(\cdot,2T_{0})$ in $H^{1}(\mathbb{T}_{L}\times(0,1))$ at a point $
		2T_{0}$ intermediate to $0$ and $T.$
		\begin{corollary}\label{corollarytransport}
				Let us recall the notations $Q^{ex}_{T}$ and $\omega_{T}$ introduced in \eqref{inshrthnd}. There exists a positive constant $C$ such that for all $\sigma_T \in L^2(\mathbb{T}_{L}\times(0,1))$, $f_{4}\in L^{2}({Q^{ex}_{T}})$ and for all values of the parameters $s\geqslant 1$ and $\lambda \geq 1$, the solution $\sigma$ of \eqref{adjtra} satisfies the following inequality
				\begin{equation}\label{obsadjtralinfty}
				\begin{split}
				\|\xi^{-\frac{3}{2}}\sigma e^{-s\phi}\|^{2}_{L^{\infty}(0,T;L^{2}(\mathbb{T}_{L}\times(0,1)))}
				\leqslant C(\|\xi^{-\frac{1}{2}}f_{4}e^{-s\phi}\|^{2}_{L^{2}({Q^{ex}_{T}})}
				+s\lambda\|\xi^{-\frac{1}{2}}\sigma e^{-s\phi}\|^{2}_{L^{2}(\omega_{T})}).
				\end{split}
				\end{equation}
				There exists a positive constant $C$ such that for all $\sigma_T \in H^1(\mathbb{T}_{L}\times(0,1))$, $f_{4}\in L^{2}(0,T;H^{1}(\mathbb{T}_{L}\times(0,1))),$ and for all values of the parameter $s\geqslant 1$  and $\lambda \geq 1,$ the solution $\sigma$ of \eqref{adjtra} satisfies the following inequality
				\begin{equation}\label{obsadjtralinfty2}
				\begin{split}
				\|\xi^{-\frac{3}{2}}\nabla\sigma e^{-s\phi}\|^{2}_{L^{\infty}(0,T;L^{2}(\mathbb{T}_{L}\times(0,1)))}
				\leqslant C(\|\xi^{-\frac{1}{2}}\nabla f_{4}e^{-s\phi}\|^{2}_{L^{2}({Q^{ex}_{T}})}
				+s\lambda\|\xi^{-\frac{1}{2}}\nabla\sigma e^{-s\phi}\|^{2}_{L^{2}(\omega_{T})}).
				\end{split}
				\end{equation}
		\end{corollary}
		\begin{proof}
			We test \eqref{adjtra}$_{1}$ by $\xi^{-3}\sigma e^{-2s\phi}$ and use integration by parts to have:
			\begin{equation}\label{energysflinfty}
			\begin{array}{ll}
			&\displaystyle\frac{d}{dt}\left(\frac{1}{2}\int\limits_{\mathbb{T}_L \times (0,1)}\xi^{-3} e^{-2s\phi}|{\sigma}|^{2}\right)\\
			&\leqslant\displaystyle\frac{1}{2}\int\limits_{\mathbb{T}_L \times (0,1)}|{\sigma}|^{2}\left|\left(2\frac{P'(\overline{\rho})}{\nu}\xi^{-3} e^{-2s\phi}+\left(\partial_{t}+\overline{u}_{1}\partial_{x}\right)(\xi^{-3} e^{-2s\phi})\right)\right|\\
			&\displaystyle\quad+ \left(\int\limits_{\mathbb{T}_L \times (0,1)}\xi^{-3} e^{-2s\phi}|{f}_{4}|^{2}\right)^{\frac{1}{2}}\left(\int\limits_{\mathbb{T}_L \times (0,1)}\xi^{-3} e^{-2s\phi}|{\sigma}|^{2}\right)^{\frac{1}{2}}.
			\end{array}
			\end{equation}
			Further using $\xi\geqslant 1$ and the estimates \eqref{prilies} and \eqref{prilies*} one has
			$$\left|\left(\partial_{t}+\overline{u}_{1}\partial_{x}\right)(\xi^{-3} e^{-2s\phi})\right|\leqslant Cs\lambda\xi^{-1}e^{-2s\phi}.$$
			Hence the R.H.S of \eqref{energysflinfty} can be majorized by a constant multiple of 
			$$\displaystyle 
			s\lambda\int\limits_{\mathbb{T}_L \times (0,1)}\xi^{-1}e^{-2s\phi}|{\sigma}|^{2}+\int\limits_{\mathbb{T}_L \times (0,1)}\xi^{-1} e^{-2s\phi}|{f}_{4}|^{2}.$$
			Now integrating both sides of \eqref{energysflinfty} with respect to time in the interval $(0,t),$ for any $t\in(0,T),$ recalling that, $\xi^{-3}e^{-2s\phi}|\sigma|^{2}(\cdot,0)=0$ and making use of \eqref{obsadjtra} we obtain \eqref{obsadjtralinfty}.\\
			One proves \eqref{obsadjtralinfty2} by applying \eqref{obsadjtralinfty} to the equation satisfied by $\nabla\sigma.$
			
			\end{proof}

		%
		%
		%
		\section{Observability and unique continuation of the system \eqref{adjsys}}\label{obssys}
		%
		%
		
		The goal of this section is to prove Theorem \ref{main3} and Corollary \ref{uniquecontinuation}. In order to do that, as for the proof of Theorem \ref{lemincreg}, the key step is first to prove an observability result for the system \eqref{adjsysq}. 
		In fact, we will start by showing the following result: 
		\begin{lem}
			\label{Lem-Obs-sigma-q}
			There exists a constant $C >0$ such that for all $(\sigma, q, \psi)$ solving \eqref{adjsysq} with initial datum $(\sigma_T, q_T, \psi_T, \psi_T^1)$ having the regularity \eqref{indatadjq} and satisfying the compatibility conditions \eqref{compatonqT}, 
			\begin{multline}
			\label{penobst0}
			\|( \psi(\cdot, 0), \partial_t \psi(\cdot, 0)) \|_{H^3(\mathbb{T}_L \times (0,1))) \times H^1(\mathbb{T}_L \times (0,1))}
			+ 
			\|q(\cdot, 0)\|_{H^2(\mathbb{T}_L \times (0,1))}
			+ 
			\| \sigma(\cdot, 0) \|_{H^1(\mathbb{T}_L \times (0,1))}
			\\
			\leq 
			C \| \psi \|_{L^2(\omega_1^T)}
			+ 
			C \| q \|_{H^1(\omega_T)}
			+ 
			C \| \sigma \|_{H^1(\omega_T)},
			\end{multline}
			where $\omega_T = \omega \times (0,T)$, $\omega_1^T = \omega_1 \times (0,T).$
		\end{lem}
		Lemma \ref{Lem-Obs-sigma-q} is proved in Section \ref{Subsec-Lem-Obs-sigma-q} below, and follows from a suitable use of the various Carleman estimates proved in the previous section.
		\begin{samepage}
		Based on Lemma \ref{Lem-Obs-sigma-q}, it will be rather easy to derive an estimate on $v(\cdot, 0)$ in $H^2(\mathbb{T}_L \times (0,1))$ and conclude the proof of Theorem \ref{main3}, which will be done in Section \ref{Subsec-Proof-Thm-Main3}.
		\subsection{Proof of Lemma \ref{Lem-Obs-sigma-q}}\label{Subsec-Lem-Obs-sigma-q} 
		Note that, from Lemma \ref{lemconstp2}, for $(\sigma_T, q_T, \psi_T, \psi_T^1)$ having the regularity \eqref{indatadjq} and satisfying the compatibility conditions \eqref{compatonqT}, we have that $(\partial_t + \overline u_1 \partial_x) q_{|\mathbb{T}_L \times \{1\}}$ belongs to $L^2(\mathbb{T}^{1}_{T})$. We thus apply Theorem \ref{Carlbeamthm} with $f_\psi = (\partial_t + \overline u_1 \partial_x) q_{|\mathbb{T}_L \times \{1\}}$:
		\begin{align}\label{est-psi-in-terms-of-q}
		& \displaystyle s^{7}\lambda^{8}\iint\limits_{{\mathbb{T}^{1}_{T}}} {\xi}^{7}|{\psi}|^{2}e^{-2s\phi}+s^{5}\lambda^{6}\iint\limits_{{\mathbb{T}^{1}_{T}}} {\xi}^{5}|{\partial_{x}{\psi}}|^{2}e^{-2s\phi}
		+s^{3}\lambda^{4}\iint\limits_{{\mathbb{T}^{1}_{T}}} {\xi}^{3} ( |{\partial_{xx}{\psi}}|^{2}+ |{\partial_{t}{\psi}}|^{2}) e^{-2s\phi}\notag\\
		&\displaystyle+s\lambda^{2}\iint\limits_{{\mathbb{T}^{1}_{T}}} {\xi}( |{\partial_{tx}{\psi}}|^{2}+|{\partial_{xxx}{\psi}}|^{2})e^{-2s\phi} 
		+
		\displaystyle\frac{1}{s}\iint\limits_{{\mathbb{T}^{1}_{T}}} \frac{1}{{\xi}}(|{\partial_{tt}{\psi}}|^{2}+|\partial_{txx}\psi|^{2}+|{\partial_{xxxx}{\psi}}|^{2})e^{-2s\phi}\notag
		\\
		&
		\displaystyle\leqslant  
		C\iint\limits_{{\mathbb{T}^{1}_{T}}} |\partial_t q + \overline u_1 \partial_x q|^{2}e^{-2s\phi}
		+Cs^{7}\lambda^{8}\iint\limits_{{\omega_1^T}} {\xi}^{7}|{\psi}|^{2}e^{-2s\phi}.
		\end{align}
		From this estimate we deduce the estimate of $\partial_{t}(\partial_t \psi + \overline u_1 \partial_x \psi)$ and $\partial_{x}(\partial_t \psi + \overline u_1 \partial_x \psi)$. We also know from Lemma \ref{lemconstp2} that $\sigma \in C^0([0,T]; H^1(\mathbb{T}_L \times (0,1))\cap C^{1}([0,T];L^{2}(\mathbb{T}_{L}\times(0,1)))$, so we can apply Corollary \ref{carlht**} to $q$, $\partial_t q$ and $\partial_{x} q$:
		\begin{align}\label{obs-q}
		\displaystyle
		&\iint\limits_{{{Q^{ex}_{T}}}}e^{-2s\phi}(\frac{1}{s \xi}|\nabla q|^{2}+s\lambda^{2}{\xi}| q|^{2})
		\\
		&\leqslant C\left( \frac{1}{s^{2}\lambda^{2}}\iint\limits_{{{Q^{ex}_{T}}}}e^{-2s\phi}\frac{1}{\xi^{2}}| \sigma |^{2}+\frac{1}{s \lambda}\iint\limits_{{\mathbb{T}^{1}_{T}}}e^{-2s\phi}\frac{1}{\xi}|\partial_{t} \psi + \overline u_1 \partial_{x} \psi |^{2}
		\displaystyle
		+s\lambda^{2}\iint\limits_{{\omega_{T}}}e^{-2s\phi}{\xi}|q|^{2}
		\right)\notag,
		\end{align}
		
		\begin{align}\label{obs-dt-q}
		\displaystyle
		&\iint\limits_{{{Q^{ex}_{T}}}}e^{-2s\phi}(\frac{1}{s \xi}|\nabla \partial_t q|^{2}+s\lambda^{2}{\xi}|\partial_t q|^{2})
		+\lambda\iint\limits_{{\mathbb{T}^{1}_{T}}}e^{-2s\phi}|\partial_t q|^{2}\\
		&\leqslant C\left( \frac{1}{s^{2}\lambda^{2}}\iint\limits_{{{Q^{ex}_{T}}}}e^{-2s\phi}\frac{1}{\xi^{2}}|\partial_t \sigma |^{2}+\frac{1}{s \lambda}\iint\limits_{{\mathbb{T}^{1}_{T}}}e^{-2s\phi}\frac{1}{\xi}|\partial_{tt} \psi + \overline u_1 \partial_{xt} \psi |^{2}
		\displaystyle
		+s\lambda^{2}\iint\limits_{{\omega_{T}}}e^{-2s\phi}{\xi}|\partial_t q|^{2}
		\right)\notag,
		\end{align}
		\begin{align}\label{obs-dx-q}
		\displaystyle
		&\iint\limits_{{{Q^{ex}_{T}}}}e^{-2s\phi}(\frac{1}{s \xi}|\nabla \partial_x q|^{2}+s\lambda^{2}{\xi}|\partial_x q|^{2})
		+\lambda\iint\limits_{{\mathbb{T}^{1}_{T}}}e^{-2s\phi}|\partial_x q|^{2}\\
		&\leqslant C\left( \frac{1}{s^{2}\lambda^{2}}\iint\limits_{{{Q^{ex}_{T}}}}e^{-2s\phi}\frac{1}{\xi^{2}}|\partial_x \sigma |^{2}+\frac{1}{s \lambda}\iint\limits_{{\mathbb{T}^{1}_{T}}}e^{-2s\phi}\frac{1}{\xi}|\partial_{tx} \psi + \overline u_1 \partial_{xx} \psi |^{2}
		\displaystyle
		+s\lambda^{2}\iint\limits_{{\omega_{T}}}e^{-2s\phi}{\xi}|\partial_x q|^{2}
		\right)\notag.
		\end{align}
		Now, we apply the observability estimates of Theorem \ref{Thm-obs-transport} with $f_{4}=q:$
		\begin{equation}
		\label{obs-sigma-q}
		\begin{array}{ll}
		&\displaystyle\frac{1}{s\lambda}\iint\limits_{{{Q^{ex}_{T}}}} \frac{1}{\xi}(| \sigma |^{2} + |\nabla\sigma|^2 + |\partial_t \sigma|^2) e^{-2s\phi}
		\\
		&\displaystyle\leq
		C \frac{1}{s\lambda}\iint\limits_{{\omega_{T}}}\frac{1}{\xi}(| \sigma |^{2} + |\nabla
		\sigma|^2 ) e^{-2s\phi}
		+ 
		C 
		\frac{1}{s\lambda}\iint\limits_{{{Q^{ex}_{T}}}}\frac{1}{\xi} ( |q|^2 + |\nabla q|^2)e^{-2s\phi}.
		\end{array}
		\end{equation}
		Further applying the estimates \eqref{obsadjtralinfty} and \eqref{obsadjtralinfty2} of Corollary \ref{corollarytransport} with $f_{4}=q:$
		\begin{equation}\label{transportlinfty}
		\begin{array}{ll}
		&\displaystyle\frac{1}{s^{2}\lambda^{2}}\sup_{(0,T)}\left(\int\limits_{\mathbb{T}_{L}\times(0,1)}\frac{1}{\xi^{3}}(|\sigma|^{2}+|\nabla\sigma|^{2})e^{-2s\phi}\right)\\
		&\displaystyle \leqslant C \frac{1}{s\lambda}\iint\limits_{{\omega_{T}}}\frac{1}{\xi}(| \sigma |^{2} + |\nabla
		\sigma|^2 ) e^{-2s\phi}
		+ 
		C 
		\frac{1}{s^{2}\lambda^{2}}\iint\limits_{{{Q^{ex}_{T}}}}\frac{1}{\xi} ( |q|^2 + |\nabla q|^2)e^{-2s\phi}.
		\end{array}
		\end{equation}
		\end{samepage}

		\begin{samepage}
				Therefore, summing up \eqref{est-psi-in-terms-of-q}--\eqref{obs-q}--\eqref{obs-dt-q}--\eqref{obs-dx-q}--\eqref{obs-sigma-q}--\eqref{transportlinfty}, we obtain:
		\begin{align*}
		&\displaystyle s^{7}\lambda^{8}\iint\limits_{{\mathbb{T}^{1}_{T}}} {\xi}^{7}|{\psi}|^{2}e^{-2s\phi}+s^{5}\lambda^{6}\iint\limits_{{\mathbb{T}^{1}_{T}}} {\xi}^{5}|{\partial_{x}{\psi}}|^{2}e^{-2s\phi}
		+s^{3}\lambda^{4}\iint\limits_{{\mathbb{T}^{1}_{T}}} {\xi}^{3} ( |{\partial_{xx}{\psi}}|^{2}+ |{\partial_{t}{\psi}}|^{2}) e^{-2s\phi}
		\\
		&\displaystyle+s\lambda^{2}\iint\limits_{{\mathbb{T}^{1}_{T}}} {\xi}( |{\partial_{tx}{\psi}}|^{2}+|{\partial_{xxx}{\psi}}|^{2})e^{-2s\phi} 
		+
		\displaystyle\frac{1}{s}\iint\limits_{{\mathbb{T}^{1}_{T}}} \frac{1}{{\xi}}(|{\partial_{tt}{\psi}}|^{2}+|\partial_{txx}\psi|^{2}+|{\partial_{xxxx}{\psi}}|^{2})e^{-2s\phi}	
		\\
		&+
		\iint\limits_{{{Q^{ex}_{T}}}}e^{-2s\phi}(\frac{1}{s \xi}(|\nabla q|^{2}+|\nabla \partial_t q|^2 + |\nabla \partial_x q|^2) +s\lambda^{2}{\xi} (| q|^{2} + |\partial_t q|^2 + |\partial_x q|^2))
		\\
		&+ 
		\lambda\iint\limits_{{\mathbb{T}^{1}_{T}}}e^{-2s\phi}( |\partial_t q|^{2} + |\partial_x q|^2)+
		\frac{1}{s\lambda}\iint\limits_{{{Q^{ex}_{T}}}} \frac{1}{\xi}(| \sigma |^{2} + |\nabla\sigma|^2 + |\partial_t \sigma|^2) e^{-2s\phi}
		\\
		& +\frac{1}{s^{2}\lambda^{2}}\sup_{(0,T)}\left(\int\limits_{\mathbb{T}_{L}\times(0,1)}\frac{1}{\xi^{3}}(|\sigma|^{2}+|\nabla\sigma|^{2})e^{-2s\phi}\right)\\
		&
		\leq C\iint\limits_{{\mathbb{T}^{1}_{T}}} |\partial_t q + \overline u_1 \partial_x q|^{2}e^{-2s\phi}
		+Cs^{7}\lambda^{8}\iint\limits_{{\omega_1^T}} {\xi}^{7}|{\psi}|^{2}e^{-2s\phi}
		+ 
		C\frac{1}{s^{2}\lambda^{2}}\iint\limits_{{{Q^{ex}_{T}}}}e^{-2s\phi}\frac{1}{\xi^{2}}(|\sigma|^2\\
		& + |\partial_x \sigma |^{2} + |\partial_t \sigma|^2) 
		+C\frac{1}{s \lambda}\iint\limits_{{\mathbb{T}^{1}_{T}}}e^{-2s\phi}\frac{1}{\xi}( |\partial_{tt} \psi|^2 + |\partial_{xt} \psi|^2 + |\partial_{xx} \psi |^{2} + |\partial_t \psi|^2 + |\partial_x \psi|^2) 
		\\
		& \displaystyle
		+Cs\lambda^{2}\iint\limits_{{\omega_{T}}}e^{-2s\phi}{\xi}(|q|^2 + |\partial_x q|^{2}+ |\partial_t q|^2) 
		\\
		&+ 
		 C \frac{1}{s\lambda}\iint\limits_{{\omega_{T}}}\frac{1}{\xi}(| \sigma |^{2} + |\nabla
		 \sigma|^2 ) e^{-2s\phi}
		 + 
		 C 
		 \frac{1}{s\lambda}\iint\limits_{{{Q^{ex}_{T}}}}\frac{1}{\xi} ( |q|^2 + |\nabla q|^2)e^{-2s\phi} .	
		\end{align*}
		Using that $\xi \geq \theta\geq 1$, taking $s$ and $\lambda$ large enough, we can absorb all the terms in the right hand side which are not localized in the observation set, so that we deduce
		\begin{align}
		&\displaystyle s^{7}\lambda^{8}\iint\limits_{{\mathbb{T}^{1}_{T}}} {\xi}^{7}|{\psi}|^{2}e^{-2s\phi}+s^{5}\lambda^{6}\iint\limits_{{\mathbb{T}^{1}_{T}}} {\xi}^{5}|{\partial_{x}{\psi}}|^{2}e^{-2s\phi}
		+s^{3}\lambda^{4}\iint\limits_{{\mathbb{T}^{1}_{T}}} {\xi}^{3} ( |{\partial_{xx}{\psi}}|^{2}+ |{\partial_{t}{\psi}}|^{2}) e^{-2s\phi}\notag
		\\
		&\displaystyle+s\lambda^{2}\iint\limits_{{\mathbb{T}^{1}_{T}}} {\xi}( |{\partial_{tx}{\psi}}|^{2}+|{\partial_{xxx}{\psi}}|^{2})e^{-2s\phi} 
		+
		\displaystyle\frac{1}{s}\iint\limits_{{\mathbb{T}^{1}_{T}}} \frac{1}{{\xi}}(|{\partial_{tt}{\psi}}|^{2}+|\partial_{txx}\psi|^{2}+|{\partial_{xxxx}{\psi}}|^{2})e^{-2s\phi}\notag	
		\\
		&+
		\iint\limits_{{{Q^{ex}_{T}}}}e^{-2s\phi}(\frac{1}{s \xi}(|\nabla \partial_t q|^2 + |\nabla q|^2) +s\lambda^{2}{\xi} (| q|^{2}+|\partial_{t}q|^{2}) )
		+
		\frac{1}{s\lambda}\iint\limits_{{{Q^{ex}_{T}}}} \frac{1}{\xi}(| \sigma |^{2} + |\nabla\sigma|^2 + |\partial_t \sigma|^2) e^{-2s\phi} 
		\notag
		\\
		& +\frac{1}{s^{2}\lambda^{2}}\sup_{(0,T)}\left(\int\limits_{\mathbb{T}_{L}\times(0,1)}\frac{1}{\xi^{3}}(|\sigma|^{2}+|\nabla\sigma|^{2})e^{-2s\phi}\right)\notag\\
		&
		\leq 
		Cs^{7}\lambda^{8}\iint\limits_{{\omega_1^T}} {\xi}^{7}|{\psi}|^{2}e^{-2s\phi}
		+C s\lambda^{2}\iint\limits_{{\omega_{T}}}e^{-2s\phi}{\xi}(|q|^2 + |\partial_x q|^{2}+ |\partial_t q|^2) 
		\label{obs-system-sigma-q-psi}
		\\
		&
		\quad + 
	     C \frac{1}{s\lambda}\iint\limits_{{\omega_{T}}}\frac{1}{\xi}(| \sigma |^{2} + |\nabla
	     \sigma|^2 ) e^{-2s\phi}.
		\notag	
		\end{align}
		\end{samepage}
		
		We now fix $s$ and $\lambda$ such that \eqref{obs-system-sigma-q-psi} holds. We then use the fact that 
		\begin{align*}
		& e^{-2s\phi}\leqslant C\quad \mbox{on}\,\omega_{T}, 
		\qquad 
		\xi^{i}e^{-2s\phi}\leqslant C\quad \mbox{on}\,\,\omega_{T}\cup \omega^{T}_{1} \quad\mbox{for }\,i\in\{-1,1,7\},
		\\
		& 
		\exists c >0, \hbox{ s.t. } 
		e^{-2s\phi} \geq c \quad \mbox{on}\,(\mathbb{T}_{L}\times(0,1))\times(2T_{0},T-2T_{1})
		\\ 
		&\qquad \quad \mbox{ and } \ 
		\xi^i e^{-2s\phi}\geq c \quad \mbox{on}\,(\mathbb{T}_{L}\times\{1\})\times(2T_{0},T-2T_{1}) \quad\mbox{for }\,i\in\{-3,-1,1,3,5,7\}.
		\end{align*}
		%
		to deduce from \eqref{obs-system-sigma-q-psi} that
		\begin{multline}
		\label{obs-sigma-q-psi-after-carls}
		\| \psi \|_{L^2(2T_0, T - 2T_1; H^4(\mathbb{T}_L \times \{1\})) \cap H^2(2T_0, T - 2T_1; L^2(\mathbb{T}_L \times \{1\}))}
		+ 
		\|q\|_{H^1(2T_0, T - 2T_1; H^1(\mathbb{T}_L \times (0,1)))}
		\\
		+ 
		\| \sigma \|_{L^{\infty}(2T_0, T - 2T_1; H^1(\mathbb{T}_L \times (0,1))) \cap H^1(2T_0, T - 2T_1; L^2(\mathbb{T}_L \times (0,1)))}
		\\
		\leq 
		C \| \psi \|_{L^2(\omega_1^T)}
		+ 
		C \| q \|_{L^{2}(0,T;H^{1}(\omega))\cap H^{1}(0,T;L^{2}(\omega))}
		+ 
		C \| \sigma \|_{L^{2}(0,T;H^{1}(\omega))}.
		\end{multline}
		
		The next step consists in proving that there exists a constant $C$ such that
		\begin{multline}
		\label{penobst-2T0}
		\|( \psi(\cdot, 2 T_0), \partial_t \psi(\cdot, 2 T_0)) \|_{H^3(\mathbb{T}_L \times \{1\})) \times H^1(\mathbb{T}_L \times \{1\})}
		+ 
		\|q(\cdot,  2 T_0)\|_{H^2(\mathbb{T}_L \times (0,1))}\\
		+ 
		\| \sigma(\cdot, 2T_0) \|_{H^1(\mathbb{T}_L \times (0,1))}
		\leq 
		C \| \psi \|_{L^2(\omega_1^T)}
		+ 
		C \| q \|_{L^{2}(0,T;H^{1}(\omega))\cap H^{1}(0,T;L^{2}(\omega))}
		+ 
		C \| \sigma \|_{L^{2}(0,T;H^{1}(\omega))}.
		\end{multline}
		From \eqref{obs-sigma-q-psi-after-carls}, we have an estimate on $\psi$ in $L^2(2T_0, T - 2T_1; H^4(\mathbb{T}_L \times \{1\})) \cap H^2(2T_0, T - 2T_1; L^2(\mathbb{T}_L \times \{1\}))$ and thus by interpolation on $\psi \in C^0([2T_0, T - 2T_1]; H^3(\mathbb{T}_L \times \{1\})) \cap C^1([2T_0, T - 2T_1]; H^1(\mathbb{T}_L \times \{1\}))$, so that we obtain 
		\begin{equation}
		\label{obs-psi-2T0}
		\begin{array}{ll}
		&\displaystyle\|( \psi(\cdot, 2 T_0), \partial_t \psi(\cdot, 2 T_0)) \|_{H^3(\mathbb{T}_L \times \{1\})) \times H^1(\mathbb{T}_L \times \{1\})}\\
		&\displaystyle\leq 
		C \| \psi \|_{L^2(\omega_1^T)}
		+ 
		C \| q \|_{L^{2}(0,T;H^{1}(\omega))\cap H^{1}(0,T;L^{2}(\omega))}
		+ 
		C \| \sigma \|_{L^{2}(0,T;H^{1}(\omega))}.
		\end{array}
		\end{equation}
		Then, from \eqref{obs-sigma-q-psi-after-carls}, we also have an estimate on $q$ in $C^0([2T_0, T - 2T_1]; H^1(\mathbb{T}_L \times (0,1)))$, therefore 
		\begin{equation}
		\label{obs-q-T-T1}
		\|q(\cdot,  T - 2 T_1)\|_{H^1(\mathbb{T}_L \times (0,1))}
		\leq 
		C \| \psi \|_{L^2(\omega_1^T)}
		+ 
		C \| q \|_{L^{2}(0,T;H^{1}(\omega))\cap H^{1}(0,T;L^{2}(\omega))}
		+ 
		C \| \sigma \|_{L^{2}(0,T;H^{1}(\omega))}.
		\end{equation}
		Now $q$ satisfies the heat equation \eqref{adjsysq}$_{(2,3,4,5)}$ and has the regularity \eqref{regsgqps}$_{2}$. We thus apply Lemma \ref{parabolicregularization} with $f_{\sigma}=-\frac{P'(\overline{\rho})\overline{\rho}^{2}}{\nu}\sigma,$ 
		\begin{equation}\label{gpsi}
		g_{\psi}=\left\{ \begin{array}{l}
		-\overline{\rho}(\partial_{t}\psi+\overline{u}_{1}\partial_{x}\psi)\quad\mbox{on}\quad \mathbb{T}_{L}\times\{1\},\\
		0 \quad\qquad\qquad\qquad\quad\mbox{on}\quad \mathbb{T}_{L}\times\{0\},
		\end{array}\right.
		\end{equation}
		in the time interval $(2T_{0},T-2T_{1})$ and further use \eqref{obs-sigma-q-psi-after-carls} and \eqref{obs-q-T-T1} to deduce
		\begin{equation}
		\label{obs-q-T0}
		\|q(\cdot,  2T_0)\|_{H^2(\mathbb{T}_L \times (0,1))}
		\leq 
		C \| \psi \|_{L^2(\omega_1^T)}
		+ 
		C \| q \|_{L^{2}(0,T;H^{1}(\omega))\cap H^{1}(0,T;L^{2}(\omega))}
		+ 
		C \| \sigma \|_{L^{2}(0,T;H^{1}(\omega))}.
		\end{equation}
		Finally, in view of the assumptions \eqref{indatadjq} and \eqref{compatonqT} we know from \eqref{regsgqps} that $\sigma\in C^{0}([0,T];H^{1}(\mathbb{T}_{L}\times(0,1))).$ Hence using \eqref{obs-sigma-q-psi-after-carls}, one at once deduce that 
		\begin{equation}
		\label{obs-nabla-sigma-2T0}
		\| \sigma(\cdot, 2T_{0}) \|_{H^{1}(\mathbb{T}_L \times (0,1))}
		\leq 
		C \| \psi \|_{L^2(\omega_1^T)}
		+ 
		C \| q \|_{L^{2}(0,T;H^{1}(\omega))\cap H^{1}(0,T;L^{2}(\omega))}
		+ 
		C \| \sigma \|_{L^{2}(0,T;H^{1}(\omega))}.
		\end{equation}
		The combination of \eqref{obs-psi-2T0}--\eqref{obs-q-T0}--\eqref{obs-nabla-sigma-2T0} then concludes the proof of \eqref{penobst-2T0}. 
		\\
		We now use Lemma \ref{lemconstp2} to deduce \eqref{penobst0}. Since we are dealing with the solution $(\sigma,q,\psi)$ of the system \eqref{adjsysq} in the strong regularity framework \eqref{regsgqps}, compatibility conditions analogous to \eqref{compatonqT} is automatically satisfied at time $2T_{0}.$ Hence using Lemma \ref{lemconstp2} to solve \eqref{adjsysq} starting from the time $2T_0$  and the estimate \eqref{penobst-2T0} we conclude the proof of \eqref{penobst0}.
		\subsection{Proof of Theorem \ref{main3}}\label{Subsec-Proof-Thm-Main3}
		
		For $(\sigma_T, v_T, \psi_T, \psi_T^1)$ as in Theorem \ref{main3}, we start by solving \eqref{adjsysq} with initial datum $(\sigma_T, q_T, \psi_T, \psi_T^1)$, where $q_{T} = \nu \div v_{T} +\overline\rho \sigma_T$. Using Lemma \ref{Lem-Obs-sigma-q} and recalling that $q = \nu \div v +\overline\rho \sigma$, we deduce that 
		\begin{align}
		\| \sigma(\cdot,0) \|_{H^1(\mathbb{T}_L \times (0,1))}
		& + 
		\|\div v(\cdot, 0)\|_{H^1(\mathbb{T}_L \times (0,1))}
		+
		\|( \psi(\cdot, 0), \partial_t \psi(\cdot, 0)) \|_{H^3(\mathbb{T}_L \times \{1\})) \times H^1(\mathbb{T}_L \times \{1\})}
		\notag
		\\
		& \leq 
		C \| \psi \|_{L^2(\omega_1^T)}
		+ 
		C \| q \|_{L^{2}(0,T;H^{1}(\omega))\cap H^{1}(0,T;L^{2}(\omega))}
		+ 
		C \| \sigma \|_{L^{2}(0,T;H^{1}(\omega))}
		\label{obs-psi-divv-sigma}
		\\
		& \leq 
		C \| \psi \|_{L^2(\omega_1^T)}
		+ 
		C \| v \|_{L^2(0,T; H^2(\omega)) \cap H^1(0,T; H^{1}(\omega))}
		+ 
		C \| \sigma \|_{L^{2}(0,T;H^{1}(\omega))}.
		\notag
		\end{align}
		Thus, to obtain the result of Theorem \ref{main3}, it only remains to estimate $v(\cdot,0)$. As we already have an estimate on $\div v(\cdot, 0)$, we first focus on getting an estimate on $\curl v$.
		\\
		One now uses the system \eqref{sysv} to obtain the following set of equations solved by $\mbox{curl}v$:
		\begin{equation}\label{eqcurl}
		\left\{ \begin{array}{lll}
		&-\overline{\rho}(\partial_{t}(\mbox{curl}v)+\overline{u}_{1}\partial_{x}(\mbox{curl}v))-\mu\Delta(\mbox{curl}v)=0\quad&\mbox{in}\quad (\mathbb{T}_{L}\times(0,1))\times(0,T),\\
		&\mbox{curl}v=0\quad&\mbox{on}\quad (\mathbb{T}_{L}\times\{0,1\})\times(0,T),\\
		& \mbox{curl}v(\cdot,T)=\mbox{curl}v_{T}\quad &\mbox{in}\quad \mathbb{T}_{L}\times(0,1).
		\end{array}\right.
		\end{equation}
		Thus, $\curl v$ satisfies a parabolic heat type equation with homogeneous Dirichlet boundary condition. Classical observability estimates for the heat equation (see e.g. \cite{fursikov} or \cite{farnan}) immediately yields
		\begin{equation}
		\label{Obs-curl-v}
		\|\curl v(\cdot, 0)\|_{H^1(\mathbb{T}_L \times (0,1))}
		\leq
		C \| \curl v\|_{L^2(\omega_T)}
		\leq 
		C \| v \|_{L^2(0,T; H^1(\omega))}.
		\end{equation}
		Now we recover $v(\cdot,0)$ by solving the following elliptic problem at times $t=0$:
		\begin{equation}\label{decompl}
		\left\{ \begin{array}{ll}
		\Delta v(\cdot,0)=\nabla\div(v(\cdot,0))+\begin{pmatrix}
		\partial_{z}\\[2.mm]
		-\partial_{x}
		\end{pmatrix}(\curl v(\cdot,0))\quad\mbox{in}\quad & \mathbb{T}_{L}\times(0,1),\\
		v(\cdot, 0)\cdot n= \psi(\cdot, 0) \,& \mbox{on}\,\mathbb{T}_{L}\times\{1\},\\
		v(\cdot, 0) \cdot n=0\,& \mbox{on}\, \mathbb{T}_{L}\times\{0\},\\
		\displaystyle\curl v(\cdot, 0)=0\,&\mbox{on}\, \mathbb{T}_{L}\times\{0,1\}.	\end{array}\right.
		\end{equation}
		One can use standard elliptic regularity to deduce that 
		\begin{equation*}
		\begin{array}{ll}
		\| v (\cdot, 0)\|_{H^2(\mathbb{T}_L \times (0,1))}
		\leq
		C \|\curl v(\cdot, 0)\|_{H^1(\mathbb{T}_L \times (0,1))}  
		&+
		C\|\div v(\cdot, 0)\|_{H^1(\mathbb{T}_L \times (0,1))}\\
		&+
		C \| \psi(\cdot, 0)\|_{H^3(\mathbb{T}_L\times\{1\})}. 
		\end{array}
		\end{equation*}
		With the estimates \eqref{obs-psi-divv-sigma}--\eqref{Obs-curl-v} and this last estimate, we then deduce that 
		\begin{equation}
		\| v (\cdot, 0)\|_{H^1(\mathbb{T}_L \times (0,1))}
		\leq 
		C \| \psi \|_{L^2(\omega_1^T)}
		+ 
		C \| v \|_{L^2(0,T; H^2(\omega)) \cap H^1(0,T; H^{1}(\omega))}
		+ 
		C \| \sigma \|_{L^{2}(0,T;H^{1}(\omega))}.
		\end{equation}
		Together with \eqref{obs-psi-divv-sigma}, this concludes the proof of Theorem \ref{main3}.
	\subsection{Proof of Corollary \ref{uniquecontinuation}}\label{proofucont} The assumption that $(\sigma,v,\psi)=0$ in $(\omega_{T})^{2}\times \omega^{1}_{T}$ at once gives $(\sigma,q,\psi)=0$ in $(\omega_{T})^{2}\times \omega^{1}_{T}$ and hence form \eqref{obs-system-sigma-q-psi} we furnish:
	\begin{equation}\label{ucontsqp}
	\begin{array}{l}
	(\sigma,q,\psi)=0\quad\mbox{in}\quad ((\mathbb{T}_{L}\times(0,1))\times(0,T))^{2}\times ((\mathbb{T}_{L}\times\{1\})\times(0,T)).
	\end{array}
	\end{equation}
	The only task is to prove $v=0$ on $((\mathbb{T}_{L}\times(0,1))\times(0,T)).$ From \eqref{ucontsqp} it is easy to observe that
	$\mbox{div}\,v=0$ in $((\mathbb{T}_{L}\times(0,1))\times(0,T)).$ Then one considers the system \eqref{eqcurl} solved by $\mbox{curl}\,v$ use $\mbox{curl}\,v=0$ in $\omega_{T}$ and classical Carleman estimate for heat type equation from \cite{fursikov} or \cite{farnan} to infer $\mbox{curl}\,v=0$ in $((\mathbb{T}_{L}\times(0,1))\times(0,T))$. Since $v\in C^{0}([0,T];H^{2}(\mathbb{T}_{L}\times(0,1)))$ (follows from \eqref{imregv}) it satisfies a system of the form \eqref{decompl} for each $t\in(0,T)$ but with zero source terms since we already have:
	$$(\mbox{div}\,v,\mbox{curl}\,v,\psi)=0\quad\mbox{in}\quad ((\mathbb{T}_{L}\times(0,1))\times(0,T))^{2}\times ((\mathbb{T}_{L}\times\{1\})\times(0,T)).$$
	Considering this equations component wise and using that $v=0$ in $\omega_{T}$ one proves $v=0$ in $(\mathbb{T}_{L}\times(0,1))\times(0,T).$ This finishes the proof of Corollary \ref{uniquecontinuation}.
		\section{Appendix}\label{appendix}
	\subsection{Proof of Lemma \ref{lemconstp2}:}\label{appendix1}	This section is devoted for the proof of Lemma \ref{lemconstp2}.
	\begin{proof}[Proof of Lemma \ref{lemconstp2}]
		In this step we will prove the existence and regularity result for $(\sigma,q,\psi)$ solving the system \eqref{adjsysq}. In fact, we will prove that under the assumptions \eqref{indatadjq} and \eqref{compatonqT}, the system \eqref{adjsysq} admits a unique solution in the following functional framework \eqref{regsgqps}.
		%
		We will first prove a local in time existence result for the problem \eqref{adjsysq}. Then using the linearity of \eqref{adjsysq} we iterate the time steps in order to show \eqref{regsgqps}. Since the problem \eqref{adjsysq} is posed backward in time by local in time existence, we first work in some time interval of the form $(T-T_{0},T)$ for $T_{0}$ sufficiently small. 
		\\
		Let $0<T_{0}<T.$ We consider the system \eqref{adjsysq}. We are going to define a suitable map whose fixed point gives a solution to the system \eqref{adjsysq} in the time interval $(T-T_{0},T)$. We define
		\begin{equation}\label{ashsps}
		\begin{split}
		& \mathbb{H}^{T_{0}}_{1}=(L^{2}(T-T_{0},T;H^{1}(\mathbb{T}_{L}\times(0,1)))\cap H^{1}(T-T_{0},T;L^{2}(\mathbb{T}_{L}\times(0,1))))\\
		&\times (H^{1}(T-T_{0},T;H^{3/2}(\mathbb{T}_{L}\times\{1\}))
		\cap H^{7/4}(T-T_{0},T;L^{2}(\mathbb{T}_{L}\times\{1\}))\\
		&\cap L^{2}(T-T_{0},T;H^{5/2}(\mathbb{T}_{L}\times\{1\}))\cap H^{3/4}(T-T_{0},T;H^{1}(\mathbb{T}_{L}\times\{1\}))),
		\end{split}
		\end{equation}
		and for $(\widehat{\sigma},\widehat{\psi})\in\mathbb{H}^{T_{0}}_{1}$ satisfying the condition
		\begin{equation}\label{compatibilty}
		\begin{split}
		\displaystyle
		\widehat{\sigma}(\cdot,T)=\sigma_{T}\quad\mbox{and}\quad(\widehat{\psi},\partial_{t}\widehat{\psi})(\cdot,T)=(\psi_{T},\psi^{1}_{T}), 
		\end{split}
		\end{equation}
		we solve the system
		\begin{equation}\label{adjsysq2}
		\left\{ \begin{array}{ll}
		\displaystyle
		-\partial_{t}\sigma-\overline{u}_{1}\partial_{x}\sigma+\frac{P'(\overline{\rho})\overline{\rho}}{\nu}\sigma=\frac{P'(\overline{\rho})}{\nu}q\,& \mbox{in}\, (\mathbb{T}_{L}\times(0,1))\times(T-T_{0},T),
		\vspace{1.mm}\\
		\displaystyle
		-(\partial_{t}q+\overline{u}_{1}\partial_{x}q)-\frac{\nu}{\overline{\rho}}\Delta q-\frac{P'(\overline{\rho})\overline{\rho}}{\nu}q=-\frac{P'(\overline{\rho})\overline{\rho}^{2}}{\nu}\widehat{\sigma}\, &\mbox{in} \, (\mathbb{T}_{L}\times(0,1))\times(T-T_{0},T),
		\vspace{1.mm}\\
		\partial_{z}q=-\overline{\rho}(\partial_{t}\widehat{\psi}+\overline{u}_{1}\partial_{x}\widehat{\psi})\,& \mbox{on}\,(\mathbb{T}_{L}\times\{1\})\times(T-T_{0},T),\\[1.mm]
		\partial_{z}q=0\,& \mbox{on}\, (\mathbb{T}_{L}\times\{0\})\times(T-T_{0},T),
		\vspace{1.mm}\\
		q(\cdot,T)=q_{T}\,& \mbox{in} \, \mathbb{T}_{L}\times(0,1),
		\vspace{1.mm}\\
		{\sigma}(\cdot,T)=\sigma_{T}\,& \mbox{in}\, \mathbb{T}_{L}\times(0,1),
		\vspace{1.mm}\\
		\partial_{tt}\psi+\partial_{txx}\psi+\partial_{xxxx}\psi=(\partial_{t}+\overline{u}_{1}\partial_{x})q\,& \mbox{on}\, (\mathbb{T}_{L}\times\{1\})\times(T-T_{0},T),
		\vspace{1.mm}\\
		\psi(T)=\psi_{T}\quad \mbox{and}\quad\partial_{t}\psi(T)=\psi_{T}^{1}\,&\mbox{in}\, \mathbb{T}_{L}\times\{1\},
		\end{array}\right.
		\end{equation}
		This defines the following map:
		\begin{equation}\label{deffxmap}
		\begin{split}
		\mathcal{L}_{T_0}: (\widehat{\sigma},\widehat{\psi}) \in \mathbb{H}_{1}^{T_{0}} \longmapsto (\sigma,\psi).
		\end{split}
		\end{equation}
		We will show that $\mathcal{L}_{T_0}$ maps $\mathbb{H}_{1}^{T_{0}}$ into itself and is a contraction there. Observe that a fixed point of the map infers a solution to the system \eqref{adjsysq} in the time interval $(T-T_{0},T)$. In the sequel we will show that the map $\mathcal{L}_{T_0}$ admits a fixed point in $\mathbb{H}_{1}^{T_{0}},$ for $T_{0}$ sufficiently small.
		\\
		\subsection{ $\mathcal{L}_{T_0}$ maps $\mathbb{H}_{1}^{T_{0}}$ to itself}.
		In that direction we first claim that there exists a positive constant $C$ such that
		\begin{equation}\label{oriestq}
		\begin{split}
		&\|q\|_{L^{2}(T-T_{0},T;H^{3}(\mathbb{T}_{L}\times(0,1)))\cap H^{3/2}(T-T_{0},T;L^{2}(\mathbb{T}_{L}\times(0,1)))}\\
		&\leqslant C(\|(\partial_{t}\widehat{\psi}+\overline{u}_{1}\partial_{x}\widehat{\psi})\|_{L^{2}(T-T_{0},T;H^{3/2}(\mathbb{T}_{L}\times\{1\}))\cap H^{3/4}(T-T_{0},T;L^{2}(\mathbb{T}_{L}\times\{1\}))}\\
		&+\|\widehat{\sigma}\|_{L^{2}(T-T_{0},T;H^{1}(\mathbb{T}_{L}\times(0,1)))\cap H^{1}(T-T_{0},T;L^{2}(\mathbb{T}_{L}\times(0,1)))}+\|q_{T}\|_{H^{2}(\mathbb{T}_{L}\times(0,1))}))
		\end{split}
		\end{equation}
		To begin with, we will just use the following regularities for the non-homogeneous source term, boundary data and the initial condition
		\begin{equation}\label{dedubdreg}
		\left\{ \begin{array}{ll}
		&\widehat{\sigma}\in L^{2}(T-T_{0},T;L^{2}(\mathbb{T}_{L}\times(0,1))),\\
		&-\overline{\rho}(\partial_{t}\widehat{\psi}+\overline{u}_{1}\partial_{x}\widehat{\psi})\in L^{2}(T-T_{0},T;H^{1/2}(\mathbb{T}_{L}\times\{1\}))\\
		&\quad\qquad\qquad\qquad\qquad\qquad\cap H^{1/4}(T-T_{0},T;L^{2}(\mathbb{T}_{L}\times\{1\})),\\
		& q_T \in H^{1}(\mathbb{T}_{L}\times(0,1)).
		\end{array}\right.
		\end{equation}
		Hence using the regularities \eqref{dedubdreg} one can apply \cite[Theorem 5.3, p. 32]{liomag2chp3} to solve \eqref{adjsysq2}$_{2}$-\eqref{adjsysq2}$_{5}$ in the following functional framework
		\begin{equation}\label{regq}
		\begin{split}
		q\in L^{2}(T-T_{0},T;H^{2}(\mathbb{T}_{L}\times(0,1)))\cap H^{1}(T-T_{0},T;L^{2}(\mathbb{T}_{L}\times(0,1))).
		\end{split}
		\end{equation}
		Moreover there exists a positive constant $C$ independent of $T_{0}$ such that
		\begin{equation}\label{inregq}
		\begin{array}{ll}
		&\displaystyle\|q\|_{L^{2}(T-T_{0},T;H^{2}(\mathbb{T}_{L}\times(0,1)))\cap H^{1}(T-T_{0},T;L^{2}(\mathbb{T}_{L}\times(0,1)))}\\
		&\displaystyle\leqslant C(\|\widehat{\sigma}\|_{L^{2}(T-T_{0},T;L^{2}(\mathbb{T}_{L}\times(0,1)))}+\|(\nu\mbox{div}\,v_{T}+\overline{\rho}\sigma_{T})\|_{H^{1}(\mathbb{T}_{L}\times(0,1))}\\
		&\displaystyle\qquad+\|(\partial_{t}\widehat{\psi}+\overline{u}_{1}\partial_{x}\widehat{\psi})\|_{L^{2}(T-T_{0},T;H^{1/2}(\mathbb{T}_{L}\times\{1\}))\cap H^{1/4}(T-T_{0},T;L^{2}(\mathbb{T}_{L}\times\{1\}))}).
		\end{array}
		\end{equation}
		Let us explain how we obtain a constant $C$ independent of $T_{0}$ in the inequality \eqref{inregq}. The technique is inspired from \cite{rayvanchp3}. We extend $\widehat{\sigma}$ and $(\partial_{t}\widehat{\psi}+\overline{u}_{1}\partial_{x}\widehat{\psi})$ in $(0,T)$ by defining them zero in the time interval $(0,T-T_{0}).$ The extended functions are also denoted by the same notations $\widehat{\sigma}$ and $(\partial_{t}\widehat{\psi}+\overline{u}_{1}\partial_{x}\widehat{\psi}).$ It is easy to verify that
		\begin{equation}\nonumber
		\begin{split}
		\widehat{\sigma}\in L^{2}(0,T;L^{2}(\mathbb{T}_{L}\times(0,1)))\,\,\mbox{and}\,\, (\partial_{t}\widehat{\psi}+\overline{u}_{1}\partial_{x}\widehat{\psi})\in& L^{2}(0,T;H^{1/2}(\mathbb{T}_{L}\times\{1\}))\\
		&\cap H^{1/4}(0,T;L^{2}(\mathbb{T}_{L}\times\{1\})).
		\end{split}
		\end{equation}
		One then has the following
		\begin{equation}\label{exreg}
		\begin{split}
		\displaystyle
		&\|q\|_{L^{2}(T-T_{0},T;H^{2}(\mathbb{T}_{L}\times(0,1)))\cap H^{1}(T-T_{0},T;L^{2}(\mathbb{T}_{L}\times(0,1)))}\\
		&\leqslant\|q\|_{L^{2}(0,T;H^{2}(\mathbb{T}_{L}\times(0,1)))\cap H^{1}(0,T;L^{2}(\mathbb{T}_{L}\times(0,1)))}\\
		&\leqslant C(\|\widehat{\sigma}\|_{L^{2}(0,T;L^{2}(\mathbb{T}_{L}\times(0,1)))}+\|q_T\|_{H^{1}(\mathbb{T}_{L}\times(0,1))}\\
		&\qquad\qquad+\|(\partial_{t}\widehat{\psi}+\overline{u}_{1}\partial_{x}\widehat{\psi})\|_{L^{2}(0,T;H^{1/2}(\mathbb{T}_{L}\times\{1\}))\cap H^{1/4}(0,T;L^{2}(\mathbb{T}_{L}\times\{1\}))})\\
		& = C (\|\widehat{\sigma}\|_{L^{2}(T-T_{0},T;L^{2}(\mathbb{T}_{L}\times(0,1)))}+\|q_T\|_{H^{1}(\mathbb{T}_{L}\times(0,1))}\\
		&\qquad\qquad+\|(\partial_{t}\widehat{\psi}+\overline{u}_{1}\partial_{x}\widehat{\psi})\|_{L^{2}(T-T_{0},T;H^{1/2}(\mathbb{T}_{L}\times\{1\}))\cap H^{1/4}(T-T_{0},T;L^{2}(\mathbb{T}_{L}\times\{1\}))}),\\
		\end{split}
		\end{equation}
		where the constant $C$ (might depend of $T$) is independent of $T_{0}.$
		\\
		Now to prove \eqref{oriestq} we write the equations \eqref{adjsysq2}$_{2}$-\eqref{adjsysq2}$_{5}$ as follows:
		\begin{equation}\label{wreqqd}
		\left\{ \begin{array}{lll}
		&\displaystyle-\partial_{t}q-\frac{\nu}{\overline{\rho}}\Delta q
		\displaystyle=-\frac{P'(\overline{\rho})\overline{\rho}^{2}}{\nu}\widehat{\sigma}+\overline{u}_{1}\partial_{x}q+\frac{P'(\overline{\rho})\overline{\rho}}{\nu}q : =\mathcal{G}_{\widehat{\sigma},q}\, &\mbox{in} \, (\mathbb{T}_{L}\times(0,1))\times(T-T_{0},T),
		\vspace{1.mm}\\
		&\partial_{z}q=-\overline{\rho}(\partial_{t}\widehat{\psi}+\overline{u}_{1}\partial_{x}\widehat{\psi})\,& \mbox{on}\,(\mathbb{T}_{L}\times\{1\})\times(T-T_{0},T),\\[1.mm]
		&\partial_{z}q=0\,& \mbox{on}\, (\mathbb{T}_{L}\times\{0\})\times(T-T_{0},T),
		\vspace{1.mm}\\
		& q(\cdot,T)=q_{T}=\nu\mbox{div}v_{T}+\overline{\rho}\sigma_{T}\,& \mbox{in} \, \mathbb{T}_{L}\times(0,1),
		\end{array}\right.
		\end{equation}
		In view of \eqref{exreg} and the interpolation $L^{2}(T-T_{0},T;H^{2}(\mathbb{T}_{L}\times(0,1)))\cap H^{1}(T-T_{0},T;L^{2}(\mathbb{T}_{L}\times(0,1)))\hookrightarrow H^{\frac{1}{2}}(T-T_{0},T;H^{1}(\mathbb{T}_{L}\times(0,1)))$ one has the following regularity estimate of $\partial_{x}q$ by interpolation
		\begin{equation}\label{regdxq}
		\begin{split}
		&\|\partial_{x}q\|_{L^{2}(T-T_{0},T;H^{1}(\mathbb{T}_{L}\times(0,1)))\cap H^{1/2}(T-T_{0},T;L^{2}(\mathbb{T}_{L}\times(0,1)))}\\
		&\leqslant C(\|\widehat{\sigma}\|_{L^{2}(T-T_{0},T;L^{2}(\mathbb{T}_{L}\times(0,1)))}	+\|q_T\|_{H^{1}(\mathbb{T}_{L}\times(0,1))}
		\\
		&+\|(\partial_{t}\widehat{\psi}+\overline{u}_{1}\partial_{x}\widehat{\psi})\|_{L^{2}(T-T_{0},T;H^{1/2}(\mathbb{T}_{L}\times\{1\}))\cap H^{1/4}(T-T_{0},T;L^{2}(\mathbb{T}_{L}\times\{1\}))}
		),
		\end{split}
		\end{equation}
		for some positive constant $C$ independent of $T_{0}.$ Indeed, this can be obtained by performing the interpolation process in the time interval $(0,T)$ instead of $(T-T_{0},T).$ 
		Hence the assumption and for $(\widehat{\sigma},\widehat{\psi})\in\mathbb{H}^{T_{0}}_{1}$ ($\mathbb{H}^{T_{0}}_{1}$ is defined in \eqref{ashsps}), the obtained regularity \eqref{exreg} and \eqref{regdxq} implies that 
		\begin{equation}\label{regGsq}
		\begin{array}{l}
		\mathcal{G}_{\widehat{\sigma},q}\in L^{2}(T-T_{0},T;H^{1}(\mathbb{T}_{L}\times(0,1)))\cap H^{1/2}(T-T_{0},T;L^{2}(\mathbb{T}_{L}\times(0,1))) 
		\end{array}
		\end{equation}
		and
		\begin{equation}\label{regGsgq}
		\begin{split}
		&\|\mathcal{G}_{\widehat{\sigma},q}\|_{L^{2}(T-T_{0},T;H^{1}(\mathbb{T}_{L}\times(0,1)))\cap H^{1/2}(T-T_{0},T;L^{2}(\mathbb{T}_{L}\times(0,1)))}\\
		&\leqslant C(\|\widehat{\sigma}\|_{L^{2}(T-T_{0},T;H^{1}(\mathbb{T}_{L}\times(0,1)))\cap H^{1}(T-T_{0},T;L^{2}(\mathbb{T}_{L}\times(0,1))) }\\
		&\qquad+\|(\partial_{t}\widehat{\psi}+\overline{u}_{1}\partial_{x}\widehat{\psi})\|_{L^{2}(T-T_{0},T;H^{1/2}(\mathbb{T}_{L}\times\{1\}))\cap H^{1/4}(T-T_{0},T;L^{2}(\mathbb{T}_{L}\times\{1\}))}\\
		&\qquad+\|q_T\|_{H^{1}(\mathbb{T}_{L}\times(0,1))}),
		\end{split}
		\end{equation}
		for some positive constant $C$ independent of $T_{0}.$\\
		At this stage we will use the following regularities of the boundary and initial datum which follows from \eqref{ashsps} and \eqref{indatadjq}:
		\begin{equation}\label{regnewbin}
		\left\{ \begin{array}{ll}
		&-\overline{\rho}(\partial_{t}\widehat{\psi}+\overline{u}_{1}\partial_{x}\widehat{\psi})\in L^{2}(T-T_{0},T;H^{3/2}(\mathbb{T}_{L}\times\{1\}))\\
		&\qquad\qquad\qquad\qquad\quad\cap H^{3/4}(T-T_{0},T;L^{2}(\mathbb{T}_{L}\times\{1\})),\\
		& q_T \in H^{2}(\mathbb{T}_{L}\times(0,1)).
		\end{array}\right.
		\end{equation}
		Furthermore, in view of \eqref{compatonqT}, \eqref{regGsgq} and \eqref{regnewbin}, we can apply \cite[Theorem 5.3, p. 32]{liomag2chp3} to solve \eqref{wreqqd} in the functional framework
		\begin{equation}\label{regqim}
		\begin{split}
		q\in L^{2}(T-T_{0},T;H^{3}(\mathbb{T}_{L}\times(0,1)))\cap H^{3/2}(T-T_{0},T;L^{2}(\mathbb{T}_{L}\times(0,1))).
		\end{split}
		\end{equation}
		There exists a positive constant $C,$ such that the solution $q$ of the heat equation \eqref{wreqqd} satisfies the following estimate
		\begin{equation}\label{regsolht}
		\begin{split}
		&\|q\|_{L^{2}(T-T_{0},T;H^{3}(\mathbb{T}_{L}\times(0,1)))\cap H^{3/2}(T-T_{0},T;L^{2}(\mathbb{T}_{L}\times(0,1)))}\\
		& \leqslant C(\|\mathcal{G}_{\widehat{\sigma},q}\|_{L^{2}(T-T_{0},T;H^{1}(\mathbb{T}_{L}\times(0,1)))\cap H^{1/2}(T-T_{0},T;L^{2}(\mathbb{T}_{L}\times(0,1)))}\\
		&+\|(\partial_{t}\widehat{\psi}+\overline{u}_{1}\partial_{x}\widehat{\psi})\|_{L^{2}(T-T_{0},T;H^{3/2}(\mathbb{T}_{L}\times\{1\}))\cap H^{3/4}(T-T_{0},T;L^{2}(\mathbb{T}_{L}\times\{1\}))}\\
		&+\|q_T\|_{H^{2}(\mathbb{T}_{L}\times(0,1))})\\
		&\leqslant C(\|\widehat{\sigma}\|_{L^{2}(T-T_{0},T;H^{1}(\mathbb{T}_{L}\times(0,1)))\cap H^{1}(T-T_{0},T;L^{2}(\mathbb{T}_{L}\times(0,1))) }\\
		&+\|(\partial_{t}\widehat{\psi}+\overline{u}_{1}\partial_{x}\widehat{\psi})\|_{L^{2}(T-T_{0},T;H^{3/2}(\mathbb{T}_{L}\times\{1\}))\cap H^{3/4}(T-T_{0},T;L^{2}(\mathbb{T}_{L}\times\{1\}))}\\ &+\|q_T\|_{H^{2}(\mathbb{T}_{L}\times(0,1))}).
		\end{split}
		\end{equation}
	  In the final step of the estimate \eqref{regsolht} we have used \eqref{regGsgq}. Note that the constant $C$ in \eqref{regsolht} might depend on $T_{0}.$\\
		The regularity \eqref{regqim} implies that $q\in C^{0}([T-T_{0},T];H^{1}(\mathbb{T}_{L}\times(0,1))).$ Hence using the regularity assumption $\sigma_{T}\in H^{1}(\mathbb{T}_{L}\times(0,1))$, we obtain the following by solving \eqref{adjsysq2}$_{1}$ and \eqref{adjsysq2}$_{6},$
		\begin{multline}\label{spacesgma}
		\sigma \in \mathbb{H}^{T_{0}}_{2},
		\\
		\text{ where we have set } H^{T_0}_2 = C^{0}([T-T_{0},T];H^{1}(\mathbb{T}_{L}\times(0,1)))\cap C^{1}([T-T_{0},T];L^{2}(\mathbb{T}_{L}\times(0,1))).
		\end{multline}
		Moreover, we have the estimate
		\begin{equation}\label{estsgma}
		\begin{split}
		\displaystyle
		\|\sigma\|_{\mathbb{H}^{T_{0}}_{2}}\leqslant C(\|q\|_{L^{\infty}(T-T_{0},T;H^{1}(\mathbb{T}_{L}\times(0,1)))}+\|\sigma_{T}\|_{H^{1}(\mathbb{T}_{L}\times(0,1))}),
		\end{split}
		\end{equation}
		for some positive constant $C$ independent of $T_{0}.$ For the proofs of \eqref{spacesgma} and \eqref{estsgma} one can follow the arguments used in proving \cite[Lemma 2.4]{vallizakchp3}.\\
		Now one considers the equations \eqref{adjsysq2}$_{7}$-\eqref{adjsysq2}$_{8}.$ Using standard trace theorem and \eqref{regqim} one in particular has 
		\begin{equation}\label{qtrace}
		\begin{split}
		(\partial_{t}+\overline{u}_{1}\partial_{x})q\mid_{\mathbb{T}_{L}\times\{1\}}\in L^{2}(T-T_{0},T;L^{2}(\mathbb{T}_{L}\times\{1\})).
		\end{split}
		\end{equation}
		Now the regularity \eqref{qtrace} and the assumption \eqref{indatadjq} on $(\psi_{T},\psi^{1}_{T})$ furnish the following regularities for $\psi:$
		\begin{equation}\label{regpsi}
		\begin{split}
		\psi\in & L^{2}(T-T_{0},T;H^{4}(\mathbb{T}_{L}\times\{1\}))\cap H^{1}(T-T_{0},T;H^{2}(\mathbb{T}_{L}\times\{1\}))\\
		&\cap H^{2}(T-T_{0},T;L^{2}(\mathbb{T}_{L}\times\{1\})).
		\end{split}
		\end{equation}
		The above regularity result for $\psi$ is a consequence of the fact that the corresponding damped beam operator is the generator of an analytic semigroup. This result can be found in \cite{chenchp3} (see also \cite{raymondbeamchp3}). Using standard interpolation results it is not hard to observe from \eqref{regpsi} that
		\begin{multline}\label{regpsi2}
		\psi\in\mathbb{H}^{T_{0}}_{3}, 
		\\
		\hbox{ where we have set } \mathbb{H}^{T_{0}}_{3}=  H^{5/4}(T-T_{0},T;H^{3/2}(\mathbb{T}_{L}\times\{1\}))\cap H^{1}(T-T_{0},T;H^{2}((\mathbb{T}_{L}\times\{1\})))
		\\
		\cap H^{2}(T-T_{0},T;L^{2}(\mathbb{T}_{L}\times\{1\}))\cap H^{3/4}(T-T_{0},T;H^{5/2}((\mathbb{T}_{L}\times\{1\}))).
		\end{multline}
		Furthermore, one has the following 
		\begin{align}\label{estpsi}
		&\|\psi\|_{\mathbb{H}^{T_{0}}_{3}}\notag\\
		&\leqslant C(\|(\partial_{t}+\overline{u}_{1}\partial_{x})q\|_{L^{2}(T-T_{0},T;L^{2}(\mathbb{T}_{L}\times\{1\}))}+\|(\psi_{T},\psi^{1}_{T})\|_{H^{3}(\mathbb{T}_{L}\times\{1\})\times H^{1}(\mathbb{T}_{L}\times\{1\})}),
		\end{align}
		for some positive constant $C$ independent of $T_{0}.$ One can obtain a constant $C$ independent of $T_{0}$ in the inequality \eqref{estpsi} by defining $(\partial_{t}+\overline{u}_{1}\partial_{x})q$ equal zero in the times interval $(0,T-T_{0})$ and following the line of arguments already used in \eqref{exreg}.\\
		Hence from \eqref{estsgma} and \eqref{estpsi} one obtains that 
		\begin{equation}\label{comsgpsi}
		\begin{array}{ll}
		(\sigma,\psi)\in \mathbb{H}^{T_{0}}_{2}\times \mathbb{H}^{T_{0}}_{3}.
		\end{array}
		\end{equation}
		So far we have observed that for $0<T_{0}<T,$ $\mathcal{L}_{T_0}$ (defined in \eqref{deffxmap}) maps $\mathbb{H}^{T_{0}}_{1}$ to $\mathbb{H}^{T_{0}}_{2}\times\mathbb{H}^{T_{0}}_{3}$, which is obviously a subset of $\mathbb{H}^{T_{0}}_{1}.$
		 \subsection{Choice of $T_{0}$ small enough such that the map $\mathcal{L}_{T_0}$ admits a fixed point in the space $\mathbb{H}_{1}^{T_{0}}$} 
		
		Let us compare the space $\mathbb{H}_{1}^{T_{0}}$ with $\mathbb{H}_{2}^{T_{0}}\times\mathbb{H}_{3}^{T_{0}}$ to observe that there exists a constant $s>0$ such that
		\begin{equation}\label{precontraction}
		\begin{split}
		\|(\cdot,\cdot)\|_{\mathbb{H}_{1}^{T_{0}}}\leqslant CT_0^{s}\|(\cdot,\cdot)\|_{\mathbb{H}_{2}^{T_{0}}\times\mathbb{H}_{3}^{T_{0}}},
		\end{split}
		\end{equation}
		for some positive constant $C$ independent of $T_{0}.$ Now let 
		$(\widehat{\sigma}_{1},\widehat{\psi}_{1})\in \mathbb{H}_{1}^{T_{0}},$ $(\widehat{\sigma}_{2},\widehat{\psi}_{2})\in \mathbb{H}_{1}^{T_{0}}$ and both the pairs satisfy \eqref{compatibilty}. Let the triplets $(\sigma_{1},q_{1},\psi_{1})$ and $(\sigma_{2},q_{2},\psi_{2})$ are solutions to \eqref{adjsysq2} corresponding to $(\widehat{\sigma}_{1},\widehat{\psi}_{1})$ and $(\widehat{\sigma}_{2},\widehat{\psi}_{2})$ respectively. This implies $(\sigma_{1},\psi_{1}) = \mathcal{L}_{T_0} (\widehat{\sigma}_{1},\widehat{\psi}_{1})$ and $(\sigma_{2},\psi_{2}) = \mathcal{L}_{T_0} (\widehat{\sigma}_{2},\widehat{\psi}_{2})$.\\
		Considering the difference of the two linear systems solved by $(\sigma_{1},q_{1},\psi_{1})$ and $(\sigma_{2},q_{2},\psi_{2}),$ one obtains that $(\sigma_{d},q_{d},\psi_{d})=(\sigma_{1}-\sigma_{2},q_{1}-q_{2},\psi_{1}-\psi_{2})$ solves 
		\begin{equation}\label{adjsysdifference}
		\left\{ \begin{array}{ll}
		\displaystyle
		-\partial_{t}\sigma_{d}-\overline{u}_{1}\partial_{x}\sigma_{d}+\frac{P'(\overline{\rho})\overline{\rho}}{\nu}\sigma_{d}=\frac{P'(\overline{\rho})}{\nu}q_{d}\,& \mbox{in}\, (\mathbb{T}_{L}\times(0,1))\times(T-T_{0},T),
		\vspace{1.mm}\\
		-(\partial_{t}q_{d}+\overline{u}_{1}\partial_{x}q_{d})-\frac{\nu}{\overline{\rho}}\Delta q_{d}-\frac{P'(\overline{\rho})\overline{\rho}}{\nu}q_{d}=-\frac{P'(\overline{\rho})\overline{\rho}^{2}}{\nu}(\widehat{\sigma}_{1}-\widehat{\sigma}_{2})\, &\mbox{in} \, (\mathbb{T}_{L}\times(0,1))\times(T-T_{0},T),
		\vspace{1.mm}\\
		\partial_{z}q_{d}=-\overline{\rho}(\partial_{t}(\widehat{\psi}_{1}-\widehat{\psi}_{2})+\overline{u}_{1}\partial_{x}(\widehat{\psi}_{1}-\widehat{\psi}_{2}))\,& \mbox{on}\,(\mathbb{T}_{L}\times\{1\})\times(T-T_{0},T),\\[1.mm]
		\partial_{z}q=0\,& \mbox{on}\, (\mathbb{T}_{L}\times\{0\})\times(T-T_{0},T),
		\vspace{1.mm}\\
		q_{d}(\cdot,T)=0\,& \mbox{in} \, \mathbb{T}_{L}\times(0,1),
		\vspace{1.mm}\\
		{\sigma}_{d}(\cdot,T)=0\,& \mbox{in}\, \mathbb{T}_{L}\times(0,1),
		\vspace{1.mm}\\
		\partial_{tt}\psi_{d}+\partial_{txx}\psi_{d}+\partial_{xxxx}\psi_{d}=(\partial_{t}+\overline{u}_{1}\partial_{x})q_{d}\,& \mbox{on}\, (\mathbb{T}_{L}\times\{1\})\times(T-T_{0},T),
		\vspace{1.mm}\\
		\psi_{d}(T)=0\quad \mbox{and}\quad\partial_{t}\psi_{d}(T)=0\,&\mbox{in}\, \mathbb{T}_{L}\times\{1\},
		\end{array}\right.
		\end{equation}  
		\\
		The goal now is to show the following:
		\begin{equation}\label{contraction}
		\begin{array}{ll}
		\displaystyle\|\mathcal{L}_{T_{0}}(\widehat{\sigma}_{1},\widehat{\psi}_{1})-\mathcal{L}_{T_{0}}(\widehat{\sigma}_{2},\widehat{\psi}_{2})\|_{\mathbb{H}^{T_{0}}_{1}}=\|(\sigma_{1}-\sigma_{2},\psi_{1}-\psi_{2})\|_{\mathbb{H}_{1}^{T_{0}}}&\displaystyle=\|(\sigma_{d},\psi_{d})\|_{\mathbb{H}^{T_{0}}_{1}}\\ 
		&\displaystyle\leqslant CT_0^{s}\|(\widehat{\sigma}_{1}-\widehat{\sigma}_{2},\widehat{\psi}-\widehat{\psi}_{2})\|_{\mathbb{H}_{1}^{T_{0}}},
		\end{array}
		\end{equation}
		for some positive constant $C$ independent of $T_{0}.$\\
		One recalls that during the analysis of the system \eqref{adjsysq2}, we obtained \eqref{oriestq}-\eqref{estpsi} with the constants in the estimates independent of $T_{0},$ while the only exception was \eqref{regsolht}. Now we will explain how to obtain an estimate for $q_{d}=q_{1}-q_{2},$ analogous to \eqref{regsolht} but with a constant independent of $T_{0}.$ In that direction one first obtains an estimate similar to \eqref{exreg} with $q$ replaced by $q_{d},$ $\widehat{\sigma}$ replaced by $(\widehat{\sigma}_{1}-\widehat{\sigma}_{2}),$ $q_{T}$ by zero and $\widehat{\psi}$ by $(\widehat{\psi}_{1}-\widehat{\psi}_{2}).$ Then from \eqref{adjsysdifference} one computes the following system similar to \eqref{wreqqd}:
			\begin{equation}\label{wreqqdnew}
			\left\{ \begin{array}{lll}
			&\displaystyle-\partial_{t}q_{d}-\frac{\nu}{\overline{\rho}}\Delta q_{d}
			\displaystyle=-\frac{P'(\overline{\rho})\overline{\rho}^{2}}{\nu}(\widehat{\sigma}_{1}-\widehat{\sigma}_{2})+\overline{u}_{1}\partial_{x}q_{d}&\\
			&\displaystyle\qquad\qquad\qquad\qquad+\frac{P'(\overline{\rho})\overline{\rho}}{\nu}q _{d} =\mathcal{G}_{\widehat{\sigma}_{1}-\widehat{\sigma}_{2},q_{d}}\, &\mbox{in} \, (\mathbb{T}_{L}\times(0,1))\times(T-T_{0},T),
			\vspace{1.mm}\\
			&\partial_{z}q_{d}=-\overline{\rho}(\partial_{t}(\widehat{\psi}_{1}-\widehat{\psi}_{2})+\overline{u}_{1}\partial_{x}(\widehat{\psi}_{1}-\widehat{\psi}_{2}))\,& \mbox{on}\,(\mathbb{T}_{L}\times\{1\})\times(T-T_{0},T),\\[1.mm]
			&\partial_{z}q_{d}=0\,& \mbox{on}\, (\mathbb{T}_{L}\times\{0\})\times(T-T_{0},T),
			\vspace{1.mm}\\
			& q_{d}(\cdot,T)=0\,& \mbox{in} \, \mathbb{T}_{L}\times(0,1).
			\end{array}\right.
			\end{equation}
			Similar to \eqref{regGsq} and \eqref{regnewbin}$_{1},$ one can verify that:
			\begin{equation}\label{regularitydiff}
			\left\{ \begin{array}{l}
			\mathcal{G}_{\widehat{\sigma}_{1}-\widehat{\sigma}_{2},q_{d}}\in L^{2}(T-T_{0},T;H^{1}(\mathbb{T}_{L}\times(0,1)))\cap H^{1/2}(T-T_{0},T;L^{2}(\mathbb{T}_{L}\times(0,1))),\\
			\partial_{z}q_{d}\in  L^{2}(T-T_{0},T;H^{3/2}(\mathbb{T}_{L}\times\{0,1\}))\cap H^{3/4}(T-T_{0},T;L^{2}(\mathbb{T}_{L}\times\{0,1\})).
			\end{array}\right.
			\end{equation}
			One further observes that
			\begin{equation}\label{zeroatT}
			\left\{ \begin{array}{l}
			\displaystyle\mathcal{G}_{\widehat{\sigma}_{1}-\widehat{\sigma}_{2},q_{d}}(\cdot,T)=0\,\,\mbox{in}\,\,\mathbb{T}_{L}\times(0,1),\\
			\displaystyle\partial_{z}q_{d}\mid_{\mathbb{T}_{L}\times\{0,1\}}(\cdot,T)=0.
			\end{array}\right.
			\end{equation}
		We will now define extensions of $\mathcal{G}_{\widehat{\sigma}_{1}-\widehat{\sigma}_{2}}$ and the normal trace $\partial_{z}q_{d}$ in $(-T,T).$ We define:
    	\begin{equation}
		\widetilde{\mathcal{G}}_{\widehat{\sigma}_{1}-\widehat{\sigma}_{2},q_{d}}(\cdot,t)=\left\{ \begin{array}{lll}
		&\displaystyle\mathcal{G}_{\widehat{\sigma}_{1}-\widehat{\sigma}_{2},q_{d}}(\cdot,t)&\quad\mbox{in}\quad (T-T_{0},T),\\
		&\displaystyle \mathcal{G}_{\widehat{\sigma}_{1}-\widehat{\sigma}_{2},q_{d}}(\cdot,2(T-T_{0})-t)&\quad\mbox{in}\quad (T-2T_{0},T-T_{0}),\\
		&\displaystyle 0 &\quad\mbox{in}\quad (-T,T-T_{0})
		\end{array}\right.
		\end{equation}
	    In a similar way by reflection one defines $\widetilde{\partial_{z}q_{d}},$ the extension of $\partial_{z}q_{d}$ in $(-T,T).$\\
	    By virtue of \eqref{regularitydiff} and \eqref{zeroatT} it is not hard to check that 
	    \begin{equation}\label{regularitydiffnew}
	    \left\{ \begin{array}{l}
	    \widetilde{\mathcal{G}}_{\widehat{\sigma}_{1}-\widehat{\sigma}_{2},q_{d}}\in L^{2}(-T,T;H^{1}(\mathbb{T}_{L}\times(0,1)))\cap H^{1/2}(-T,T;L^{2}(\mathbb{T}_{L}\times(0,1))),\\
	    \widetilde{\partial_{z}q_{d}}\in  L^{2}(-T,T;H^{3/2}(\mathbb{T}_{L}\times\{0,1\}))\cap H^{3/4}(-T,T;L^{2}(\mathbb{T}_{L}\times\{0,1\})),
	    \end{array}\right.
	    \end{equation}
	    and most importantly
	    \begin{equation}\label{majorant}
	    \left\{ \begin{array}{ll}
	    &\displaystyle \|\widetilde{\mathcal{G}}_{\widehat{\sigma}_{1}-\widehat{\sigma}_{2},q_{d}}\|_{L^{2}(-T,T;H^{1}(\mathbb{T}_{L}\times(0,1)))\cap H^{1/2}(-T,T;L^{2}(\mathbb{T}_{L}\times(0,1)))}\\
	    &\displaystyle\leqslant 2\|{\mathcal{G}}_{\widehat{\sigma}_{1}-\widehat{\sigma}_{2},q_{d}}\|_{L^{2}(T-T_{0},T;H^{1}(\mathbb{T}_{L}\times(0,1)))\cap H^{1/2}(T-T_{0},T;L^{2}(\mathbb{T}_{L}\times(0,1)))},\\[4.mm]
	    &\displaystyle \|\widetilde{\partial_{z}q_{d}}\|_{L^{2}(-T,T;H^{3/2}(\mathbb{T}_{L}\times\{0,1\}))\cap H^{3/4}(-T,T;L^{2}(\mathbb{T}_{L}\times\{0,1\}))}\\
	    &\displaystyle\leqslant 2\|\partial_{z}q_{d}\|_{L^{2}(T-T_{0},T;H^{3/2}(\mathbb{T}_{L}\times\{0,1\}))\cap H^{3/4}(T-T_{0},T;L^{2}(\mathbb{T}_{L}\times\{0,1\}))},
	    \end{array}\right.
	    \end{equation}
	    with the constants in the inequalities \eqref{majorant} independent of $T_{0}.$ One can now solve \eqref{wreqqdnew} in the extended interval $(-T,T)$ and furnish
	    \begin{equation}\label{solvingqdext}
	    \begin{array}{ll}
	    &\displaystyle\|q_{d}\|_{L^{2}(T-T_{0},T;H^{3}(\mathbb{T}_{L}\times(0,1)))\cap H^{3/2}(T-T_{0},T;L^{2}(\mathbb{T}_{L}\times(0,1)))}\\
	    &\displaystyle \leqslant \|q_{d}\|_{L^{2}(-T,T;H^{3}(\mathbb{T}_{L}\times(0,1)))\cap H^{3/2}(-T,T;L^{2}(\mathbb{T}_{L}\times(0,1)))}\\
	    &\displaystyle \leqslant C\left(\|\widetilde{\mathcal{G}}_{\widehat{\sigma}_{1}-\widehat{\sigma}_{2},q_{d}}\|_{L^{2}(-T,T;H^{1}(\mathbb{T}_{L}\times(0,1)))\cap H^{1/2}(-T,T;L^{2}(\mathbb{T}_{L}\times(0,1)))}\right.\\
	    &\displaystyle\left.\qquad\qquad+\|\widetilde{\partial_{z}q_{d}}\|_{L^{2}(-T,T;H^{3/2}(\mathbb{T}_{L}\times\{0,1\}))\cap H^{3/4}(-T,T;L^{2}(\mathbb{T}_{L}\times\{0,1\}))}\right)\\
	    & \leqslant 2C\left( \|{\mathcal{G}}_{\widehat{\sigma}_{1}-\widehat{\sigma}_{2},q_{d}}\|_{L^{2}(T-T_{0},T;H^{1}(\mathbb{T}_{L}\times(0,1)))\cap H^{1/2}(T-T_{0},T;L^{2}(\mathbb{T}_{L}\times(0,1)))}\right.\\
	    &\displaystyle \left.\qquad\qquad+\|\partial_{z}q_{d}\|_{L^{2}(T-T_{0},T;H^{3/2}(\mathbb{T}_{L}\times\{0,1\}))\cap H^{3/4}(T-T_{0},T;L^{2}(\mathbb{T}_{L}\times\{0,1\}))}\right)\\
	    &\leqslant C\left(\|\widehat{\sigma}_{1}-\widehat{\sigma}_{2}\|_{L^{2}(T-T_{0},T;H^{1}(\mathbb{T}_{L}\times(0,1)))\cap H^{1}(T-T_{0},T;L^{2}(\mathbb{T}_{L}\times(0,1))) }\right.\\
	    &\qquad\left.+\|\partial_{t}(\widehat{\psi}_{1}-\widehat{\psi}_{2})+\overline{u}_{1}\partial_{x}(\widehat{\psi}_{1}-\widehat{\psi}_{2})\|_{L^{2}(T-T_{0},T;H^{3/2}(\mathbb{T}_{L}\times\{1\}))\cap H^{3/4}(T-T_{0},T;L^{2}(\mathbb{T}_{L}\times\{1\}))}\right),
	    \end{array}
	    \end{equation}
	    where the fourth step of \eqref{solvingqdext} from the third one follows by using \eqref{majorant}, the final step from the penultimate one uses an estimate similar to \eqref{regGsgq} and \eqref{wreqqdnew}$_{2}$ . We specify that the constants in the inequalities of \eqref{solvingqdext} only depends on the measure of $(-T,T)$ i.e only on $T$ but not on $T_{0}.$ Using interpolation it is evident from \eqref{solvingqdext} that $q_{d}\in C^{0}([T-T_{0},T];H^{1}(\mathbb{T}_{L}\times(0,1))).$ Further the interpolation can be first performed in the extended interval $(-T,T)$ and then take its restriction on $(T-T_{0},T)$ to furnish the following from \eqref{solvingqdext}:
	     \begin{equation}\label{linftyqd}
	     \begin{array}{ll}
	     &\displaystyle\|q_{d}\|_{L^{\infty}(T-T_{0},T;H^{1}(\mathbb{T}_{L}\times(0,1)))}\\
	     &\displaystyle\leqslant C\left(\|\widehat{\sigma}_{1}-\widehat{\sigma}_{2}\|_{L^{2}(T-T_{0},T;H^{1}(\mathbb{T}_{L}\times(0,1)))\cap H^{1}(T-T_{0},T;L^{2}(\mathbb{T}_{L}\times(0,1))) }\right.\\
	     &\displaystyle\qquad\left.+\|\partial_{t}(\widehat{\psi}_{1}-\widehat{\psi}_{2})+\overline{u}_{1}\partial_{x}(\widehat{\psi}_{1}-\widehat{\psi}_{2})\|_{L^{2}(T-T_{0},T;H^{3/2}(\mathbb{T}_{L}\times\{1\}))\cap H^{3/4}(T-T_{0},T;L^{2}(\mathbb{T}_{L}\times\{1\}))}\right),
	     \end{array}
	     \end{equation}
	    for some positive constant $C$ independent of $T_{0}.$ In a similar spirit we obtain the following trace estimate from \eqref{solvingqdext}:
	    \begin{equation}\label{traceqd}
	    \begin{array}{ll}
	    &\displaystyle\|(\partial_{t}+\overline{u}_{1}\partial_{x})q_{d}\|_{L^{2}(T-T_{0},T;L^{2}(\mathbb{T}_{L}\times\{1\}))}\\
	   &\displaystyle\leqslant C\left(\|\widehat{\sigma}_{1}-\widehat{\sigma}_{2}\|_{L^{2}(T-T_{0},T;H^{1}(\mathbb{T}_{L}\times(0,1)))\cap H^{1}(T-T_{0},T;L^{2}(\mathbb{T}_{L}\times(0,1))) }\right.\\
	    &\displaystyle\qquad\left.+\|\partial_{t}(\widehat{\psi}_{1}-\widehat{\psi}_{2})+\overline{u}_{1}\partial_{x}(\widehat{\psi}_{1}-\widehat{\psi}_{2})\|_{L^{2}(T-T_{0},T;H^{3/2}(\mathbb{T}_{L}\times\{1\}))\cap H^{3/4}(T-T_{0},T;L^{2}(\mathbb{T}_{L}\times\{1\}))}\right),
	    \end{array}
	    \end{equation}
	    for some positive constant $C$ independent of $T_{0}.$\\
	    Similar to \eqref{estsgma} and \eqref{estpsi} one also obtains the following for $\sigma_{d}$ and $\psi_{d}:$
	    \begin{equation}\label{estmatesdpd}
	    \left\{ \begin{array}{ll}
	    &\displaystyle \|\sigma_{d}\|_{\mathbb{H}^{T_{0}}_{2}}\leqslant C\|q_{d}\|_{L^{\infty}(T-T_{0},T;H^{1}(\mathbb{T}_{L}\times(0,1)))},\\[2.mm]
	    &\displaystyle \|\psi_{d}\|_{\mathbb{H}^{T_{0}}_{3}}
	    \leqslant C\|(\partial_{t}+\overline{u}_{1}\partial_{x})q_{d}\|_{L^{2}(T-T_{0},T;L^{2}(\mathbb{T}_{L}\times\{1\}))},
	    \end{array}\right.
	    \end{equation}
	    for some positive constant $C$ independent of $T_{0}.$ Using \eqref{linftyqd} and \eqref{traceqd} into \eqref{estmatesdpd} one concludes that
	    \begin{equation}\label{penultimateestiate}
	    \begin{array}{l}
	    \|(\sigma_{d},\psi_{d})\|_{\mathbb{H}^{T_{0}}_{2}\times \mathbb{H}^{T_{0}}_{3}}\leqslant C\|(\widehat{\sigma}_{1}-\widehat{\sigma}_{2},\widehat{\psi}_{1}-\widehat{\psi}_{2})\|_{\mathbb{H}^{T_{0}}_{1}},
	    \end{array}
	    \end{equation}
	    for some positive constant $C$ independent of $T_{0}.$\\
	    Using \eqref{precontraction} and \eqref{penultimateestiate} one concludes the proof of \eqref{contraction}.\\
		In view of \eqref{contraction}, for $T - T_{0}$ close to $T,$ $i.e$ for $T_{0}$ small enough, the map $\mathcal{L}_{T_0}$ is a contraction from $\mathbb{H}_{1}^{T_{0}}$ to itself. Hence applying the Banach fixed point theorem we obtain that for $T_{0}$ small enough, the map $\mathcal{L}_{T_0}$ admits a unique fixed point $(\widehat{\sigma},\widehat{\psi})$ in $\mathbb{H}_{1}^{T_{0}}$. As $(\widehat{\sigma},\widehat{\psi}) =\mathcal{L}_{T_0} (\widehat{\sigma},\widehat{\psi})$, we also have that $(\widehat{\sigma},\widehat{\psi})$ belongs to $\mathbb{H}_2^{T_0} \times \mathbb{H}_3^{T_0}$ and the following regularities coming from \eqref{regqim} and \eqref{regpsi}:
		\begin{equation}\label{regsgqps*}
		\left\{ \begin{array}{ll}
		&\sigma\in C^{0}([T-T_{0},T];H^{1}(\mathbb{T}_{L}\times(0,1)))\cap C^{1}([T-T_{0},T];L^{2}(\mathbb{T}_{L}\times(0,1))),\\
		& q\in L^{2}(T-T_{0},T;H^{3}(\mathbb{T}_{L}\times(0,1)))\cap H^{3/2}(T-T_{0},T;L^{2}(\mathbb{T}_{L}\times(0,1))),\\
		& \psi\in L^{2}(T-T_{0},T;H^{4}(\mathbb{T}_{L}\times\{1\}))\cap H^{1}(T-T_{0},T;H^{2}(\mathbb{T}_{L}\times\{1\}))\\
		&\qquad\quad\cap H^{2}(T-T_{0},T;L^{2}(\mathbb{T}_{L}\times\{1\})),
		\end{array}\right.
		\end{equation}
		provided $T_{0}$ is sufficiently small. Further \eqref{regsgqps*} infers that 
		\begin{equation}\label{regsgqps**}
		\left\{ \begin{array}{ll}
		& \sigma\in C^{0}([T-T_{0},T];H^{1}(\mathbb{T}_{L}\times(0,1))]),\,\,q\in C^{0}([T-T_{0},T];H^{2}(\mathbb{T}_{L}\times(0,1)))\\
		& \psi\in  C^{0}([T-T_{0},T];H^{3}(\mathbb{T}_{L}\times\{1\}))
		\cap C^{1}([T-T_{0},T];H^{1}(\mathbb{T}_{L}\times(0,1))).
		\end{array}\right.
		\end{equation}
		The continuities \eqref{regsgqps**} in time and the system \eqref{adjsysq} can be used to check the following compatibilities at time $t=T-T_{0}:$
		\begin{equation}\label{compatTT0}
		\begin{split}
		&(i)\,\partial_{z}q(\cdot,T-T_{0})=-\overline{\rho}(\partial_{t}\psi+\overline{u}_{1}\partial_{x}\psi)(\cdot,T-T_{0})\,\,\mbox{on}\,\,\mathbb{T}_{L}\times\{1\},\\
		&(ii)\,\partial_{z}q(\cdot,T-T_{0})=0\,\,\mbox{on}\,\,\mathbb{T}_{L}\times\{0\}.
		\end{split}
		\end{equation}
		Further one recalls that in proving \eqref{regsgqps*} we did no assumption on the size of the initial datum. In view of \eqref{regsgqps**}, \eqref{compatTT0} and using that the linearity of the system \eqref{adjsysq} the solution $(\sigma,q,\psi)$ can be extended to the time interval $(0,T)$ by iteration in order to prove \eqref{regsgqps}.\\
		This finishes the proof of \eqref{regsgqps} and thus of Lemma \ref{lemconstp2}.
	\end{proof}
		\subsection{A lemma on parabolic regularization}\label{appendix2}
		In this section we prove a result on parabolic regularization for a heat type equation with non homogeneous Neumann boundary condition. This result is in particular used in proving the inequality \eqref{obs-q-T0} from the information \eqref{obs-q-T-T1}.
		\begin{lem}\label{parabolicregularization}
			We recall the notations introduced in \eqref{inshrthnd}. Let
			\begin{equation}\label{fsigmagpsi}
			\left\{ \begin{array}{ll}
		&\displaystyle f_{\sigma}\in L^{2}(0,T;H^{1}(\mathbb{T}_{L}\times(0,1)))\cap H^{\frac{1}{2}}(0,T;L^{2}(\mathbb{T}_{L}\times(0,1))),\\
		&\displaystyle	g_{\psi}\in L^{2}(0,T;H^{\frac{3}{2}}(\mathbb{T}_L\times\{0,1\}))\cap H^{\frac{3}{4}}(0,T;L^{2}(\mathbb{T}_L\times\{0,1\})),
			\end{array}\right.
			\end{equation}
			 $q_{T}\in H^{1}(\mathbb{T}_{L}\times(0,1))$ and the following compatibility is satisfied:
			 \begin{equation}\label{compat}
			 \begin{array}{l}
			 \partial_{z}q_{T}=g_{\psi}(\cdot,T)\quad \mbox{in}\quad \mathbb{T}_{L}\times\{0,1\}.
			 \end{array}
			 \end{equation}
			Then for every $0<\epsilon\leqslant T,$ $q(\cdot,T-\epsilon)\in H^{2}(\mathbb{T}_{L}\times(0,1)),$ where $q$ solves
		\begin{equation}\label{heatequationap}
		\left\{ \begin{array}{ll}
		\displaystyle-(\partial_{t}q+\overline{u}_{1}\partial_{x}q)-\frac{\nu}{\overline{\rho}}\Delta q-\frac{P'(\overline{\rho})\overline{\rho}}{\nu}q
		=f_{\sigma}\, &\mbox{in} \, Q^{ex}_{T},
		\vspace{1.mm}\\
		\partial_{z}q=g_{\psi}\,& \mbox{on}\,\mathbb{T}^{1}_{T}\cup \mathbb{T}^{0}_{T},\\[1.mm]
		q(\cdot,T)=q_{T}\,& \mbox{in} \, \mathbb{T}_{L}\times(0,1),
		\end{array}\right.
		\end{equation} 
		with $P,$ $\overline{\rho},$ $\overline{u}_{1}$ introduced respectively in \eqref{1.2chp3} and \eqref{baru}. Further the following inequality holds:
		\begin{equation}\label{regularC0}
		\begin{array}{ll}
		&\displaystyle\|q(\cdot,T-\epsilon)\|_{H^{2}(\mathbb{T}_{L}\times(0,1))}\\
		&\displaystyle\leqslant C\left(\|q_{T}\|_{H^{1}(\mathbb{T}_{L}\times(0,1))}
		+\|f_{\sigma}\|_{L^{2}(0,T;H^{1}(\mathbb{T}_{L}\times(0,1)))\cap H^{\frac{1}{2}}(0,T;L^{2}(\mathbb{T}_{L}\times(0,1)))}\right.\\
		&\displaystyle\left.\quad+\|g_{\psi}\|_{L^{2}(0,T;H^{\frac{3}{2}}(\mathbb{T}_L\times\{0,1\}))\cap H^{\frac{3}{4}}(0,T;L^{2}(\mathbb{T}_L\times\{0,1\}))}\right),
		\end{array}
		\end{equation}
		for some positive constant $C.$
		\end{lem}
		\begin{proof}
			We first apply \cite[Theorem 5.3, p. 32]{liomag2chp3} (of course here backward in time) to get:
			\begin{equation}\label{regqappendix}
			\begin{array}{ll}
			&\displaystyle \|q\|_{L^{2}(0,T;H^{2}(\mathbb{T}_{L}\times(0,1)))\cap H^{1}(0,T;L^{2}(\mathbb{T}_{L}\times(0,1))))}\\
			&\displaystyle \leqslant C\left(\|q_{T}\|_{H^{1}(\mathbb{T}_{L}\times(0,1))}
			+\|f_{\sigma}\|_{L^{2}(Q^{ex}_{T})}\right.\\
			&\displaystyle\left.\quad+\|g_{\psi}\|_{L^{2}(0,T;H^{\frac{1}{2}}(\mathbb{T}_L\times\{0,1\}))\cap H^{\frac{1}{4}}(0,T;L^{2}(\mathbb{T}_L\times\{0,1\}))}\right),
			\end{array}
			\end{equation}
			for some positive constant $C.$ Now we introduce:
			\begin{equation}\label{defq1}
			\begin{array}{l}
			q_{1}(\cdot,t)=(T-t)q(\cdot,t)\quad\mbox{in}\quad Q^{ex}_{T}.
			\end{array}
			\end{equation}
			One observes that $q_{1}$ solves:
			\begin{equation}\label{heatequationap*}
			\left\{ \begin{array}{ll}
			\displaystyle-(\partial_{t}q_{1}+\overline{u}_{1}\partial_{x}q_{1})-\frac{\nu}{\overline{\rho}}\Delta q_{1}-\frac{P'(\overline{\rho})\overline{\rho}}{\nu}q_{1}
			=(T-t)f_{\sigma}+q\, &\mbox{in} \, Q^{ex}_{T},
			\vspace{1.mm}\\
			\partial_{z}q_{1}=(T-t)g_{\psi}\,& \mbox{on}\,\mathbb{T}^{1}_{T}\cup \mathbb{T}^{0}_{T},\\[1.mm]
			q_{1}(\cdot,T)=0\,& \mbox{in} \, \mathbb{T}_{L}\times(0,1),
			\end{array}\right.
			\end{equation} 
			 In view of \eqref{fsigmagpsi} and \eqref{regqappendix}:
			$$(T-t)f_{\sigma}+q\in L^{2}(0,T;H^{1}(\mathbb{T}_{L}\times(0,1)))\cap H^{\frac{1}{2}}(0,T;L^{2}(\mathbb{T}_{L}\times(0,1))) $$
			and 
			$$(T-t)g_{\psi}\in L^{2}(0,T;H^{\frac{3}{2}}(\mathbb{T}_L\times\{0,1\}))\cap H^{\frac{3}{4}}(0,T;L^{2}(\mathbb{T}_L\times\{0,1\})).$$
			Of course the compatibility $\partial_{z}q_{1}(\cdot,T)=0=(T-T)\partial_{z}g_{\psi}(\cdot,T)$ is satisfied on $\mathbb{T}_{L}\times\{0,1\}.$\\
			Hence once again applying \cite[Theorem 5.3, p. 32]{liomag2chp3}, we furnish that
			\begin{equation}\nonumber
			\begin{array}{l}
			q_{1}\in L^{2}(0,T;H^{3}(\mathbb{T}_{L}\times(0,1)))\cap H^{\frac{3}{2}}(0,T;(\mathbb{T}_{L}\times(0,1))).
			\end{array}
			\end{equation} 
			Consequently using \cite[Theorem 2.1, Section 2.2]{liomag2chp3} in the interval $(T-\epsilon,T)$ one obtains that
			$$q_{1}(\cdot,T-\epsilon)\in H^{2}(\mathbb{T}_{L}\times(0,1)).$$
			Hence the definition \eqref{defq1} clearly implies that $q(\cdot,T-\epsilon)\in H^{2}(\mathbb{T}_{L}\times(0,1)),$ for every $0<\epsilon\leqslant T$ and the estimate \eqref{regularC0} follows. 
			\end{proof}
		\bibliographystyle{plain}

		\end{document}